\documentclass[11pt,a4paper,intlimits]{amsart}

\usepackage{amsmath}
\usepackage{amsopn,amsfonts,amssymb,amsthm}

\usepackage[T1]{fontenc}
\usepackage[english]{babel}
\usepackage[utf8]{inputenc}
\usepackage{epsfig}
\usepackage{graphicx}
\DeclareGraphicsExtensions{.png.jpg}

\newtheorem{thm}{Theorem}[section]
\newtheorem{prop}[thm]{Proposition}
\newtheorem{lem}[thm]{Lemma}

\newtheorem{col}[thm]{Corollary}

\numberwithin{equation}{section}

\newcommand{\N}{\mathbb{N}}

\newcommand{\C}{\mathbb{C}}

\newcommand{\R}{\mathbb{R}}

\newcommand{\D}{\mathbb{D}}

\DeclareMathOperator{\im}{Im}
\DeclareMathOperator{\re}{Re}
\DeclareMathOperator{\dist}{dist}
\DeclareMathOperator{\diam}{diam}
\DeclareMathOperator{\HD}{HD}

\DeclareMathOperator{\intt}{int}

\newcommand{\sms}{\setminus}
\newcommand{\leeq}{\leqslant}
\newcommand{\greq}{\geqslant}

\newcommand{\la}{\lambda}
\newcommand{\ve}{\varepsilon}
\newcommand{\mc}{\mathcal}
\newcommand{\dt}{\divideontimes}

\begin{document}

\title[On the derivative of the Hausdorff dimension]{On the derivative of the Hausdorff dimension of the Julia sets for $z^2+c$}% at $c=-3/4$}

\author{Ludwik Jaksztas, Michel Zinsmeister}

\address{Faculty of Mathematics and Information Sciences, Warsaw University of Technology, Pl. Politechniki 1, 00-661 Warsaw, Poland}
\email{jaksztas@impan.gov.pl}

\address{Institut Denis Poisson, COST, Universit\'e d'Orl\'eans, BP6749 45067 Orl\'eans Cedex, France}
\email{zins@univ-orleans.fr}

\thanks{partially supported by Polish NCN grants NN201 607 640 and UMO-2014/13/B/ST1/01033}

%\keywords{Hausdorff dimension, Julia set}

\begin{abstract}
Let $d(c)$ denote the Hausdorff dimension of the Julia set $J_c$ of the polynomial $f_c(z)=z^2+c$. We will investigate behavior of the function $d(c)$ when real parameter $c$ tends to a parabolic parameter $c_0$. %The function $c\mapsto d(c)$ is real-analytic...% on the interval $(-5/4,-3/4)$, which is included in the $1/2$ bulb of the Mandelbrot set. The number $c=-3/4$ is the parameter at which, going from the right, the attracting fixed point bifurcates to an orbit of period two. In \cite{J} we studied $d'(c)$ when $c\rightarrow-3/4$ from the right. Here, under numerically verified assumption $d(-3/4)<4/3$, we will show that there exists $K_{-3/4}^->0$ such that $d'(c)(-3/4-c)^{-3d(-3/4)/2+2}\rightarrow-K_{-3/4}^-$, when $c$ tends to $-3/4$ from the left. In particular we obtain $d'(c)\rightarrow-\infty$. This case is much harder than considered in \cite{J}.
\end{abstract}

\maketitle

\section{Introduction}\label{sec:wprowadzenie}

For a polynomial $f$ of degree at least 2, we define the filled-in Julia set $K(f)$ as the set of all points that do not escape to infinity under iteration of $f$. The Julia set $J(f)$ is the boundary of $K(f)$, i.e.
   $$J(f)=\partial K(f)=\partial\{z\in\C:f^n(z)\nrightarrow\infty\}.$$

We will consider the family of quadratic polynomials of the form
   $$f_c(z)=z^2+c, \textrm{ where } c\in\C.$$
As usual, $J_c$ and $K_c$ abbreviate $J(f_c)$ and $K(f_c)$ respectively.

We define the Mandelbrot set $\mathcal M$ as the set of all parameters $c$ for which the Julia set $J_c$ is connected, or equivalently,
   $$\mathcal M=\{c\in\C:f_c^n(0)\nrightarrow\infty\}.$$

We are interested in the function $c\mapsto d(c)$, where $d(c)$ denotes the Hausdorff dimension of the Julia set $J_c$.

Recall that a polynomial $f:\overline\C\rightarrow\overline\C$ (or more generally a rational function) is called hyperbolic (expanding) if
$\exists_{n\in\N}\;\forall_{z\in J(f)}\;|(f^n)'(z)|>1.$

The function $d$ is real-analytic on each hyperbolic component of $\intt(\mathcal M)$ (consisting of parameters related to hyperbolic maps) as well as on the exterior of $\mathcal M$ (see \cite{R})% In particular $d$ is real-analytic on $\mathcal M^0$ (the largest component, bounded by the main cardioid) which is the set of all parameters $c$ for which there exists an attracting fixed point, as well as on so-called $1/2$ bulb $\mathcal M^{1/2}$, which consist of the parameters for which there exists a minimal period 2 attracting periodic orbit.

%Notice that the components $\mathcal M^{1/2}$ and $\mathcal M^0$ contain intervals $(-5/4,-3/4)$ and $(-3/4,1/4)$, respectively. The endpoints are parabolic parameters, i.e. parameters related to the maps with a parabolic periodic point. In particular $f'_{1/4}$, $f'_{-3/4}$ have parabolic fixed points, namely $f'_{1/4}(1/2)=1$ (one petal) and $f'_{-3/4}(-1/2)=-1$ (two petals).

M. Shishikura proved in \cite{Sh} that there exists a residual (hence dense) set of parameters $\partial\mathcal M$ such that $d(c)=2$. Therefore $d$ is not continuous at $c\in\partial\mathcal M$ if $d(c)<2$. In particular, it follows from \cite{U} that $d$ is not continuous at any parabolic parameter.

Nevertheless the Hausdorff dimension is continuous along some paths. This was first proved by O. Bodart and M. Zinsmeister (see \cite{BZ}) when the real parameter tends to $1/4$ from the left. Later, the continuity was proved (see \cite{Mii}) when $c$ approaches other parabolic parameters in a "good way". In particular the function $d|_\R$ is continuous on the interval $(c_{feig},1/4]$ (included in $\mathcal M$), where $c_{feig}\approx-1.401$ is the Feigenbaum parameter. Note that $d|_\R$ is not right-continuous at $1/4$, i.e. when $c$ approaches $1/4$ from outside of the Mandelbrot set (see \cite{DSZ}). For results concerning other parameters see \cite{GS} and \cite{Ri}.

In this paper we restrict $d$ to $\R$ and investigate the derivative of the Hausdorff dimension with respect $c$ converging to a parabolic parameter. This derivative has been studied in several papers:

First, G. Havard and M. Zinsmeister proved in \cite{HZ} that
\vspace{2.5mm}
\newline
\textbf{Theorem I.}
   \emph{There exist $c_0<1/4$ and $K>1$ such that for every $c\in(c_0, 1/4)$}
      $$\frac{1}{K}\Big(\frac{1}{4}-c\Big)^{d(\frac{1}{4})-\frac{3}{2}}\leeq d'(c)\leeq K\Big(\frac{1}{4}-c\Big)^{d(\frac{1}{4})-\frac{3}{2}}.$$
We know from \cite{HZi} that $d(1/4)<3/2$. Thus, $d'(c)\rightarrow+\infty$ when $c\rightarrow1/4^-$ (tends from the left). A similar problem in the case of the exponential family, i.e. parabolic map with one petal, was solved in \cite{HUZ}.

Next, it was proven, under the assumption $d(-3/4)<4/3$ (see \cite{J} and \cite{Ji}), that
\vspace{2.5mm}
\newline
\textbf{Theorem II.}
   \emph{There exist $c_0>-3/4$ and $K>1$ such that for every $c\in(-3/4,c_0)$}
      $$-K\Big(\frac{3}{4}+c\Big)^{\frac{3}{2}d(-\frac{3}{4})-2}\leeq d'(c)\leeq -\frac{1}{K}\Big(\frac{3}{4}+c\Big)^{\frac{3}{2}d(-\frac{3}{4})-2}.$$
   \emph{In particular $d'(c)\rightarrow-\infty$, when $c\rightarrow-3/4^+$ (tends from the right).}
\vspace{2.5mm}
\newline
\textbf{Theorem III.}
   \emph{There exists $K^-_{-3/4}>0$ such that,}
      $$\lim_{c\rightarrow-3/4^-} d'(c)\Big(-\frac{3}{4}-c\Big)^{-\frac{3}{2}d(-\frac{3}{4})+2}=-K^-_{-3/4}$$
   \emph{provided $d(-3/4)<4/3$}.

\begin{figure}[!h]
\begin{center}
\hspace{\stretch{1}}\includegraphics[width=6cm]{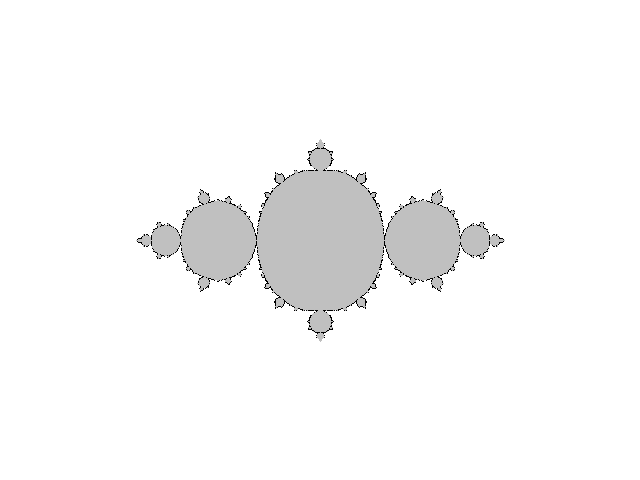}\hspace{\stretch{1}}\includegraphics[width=6cm]{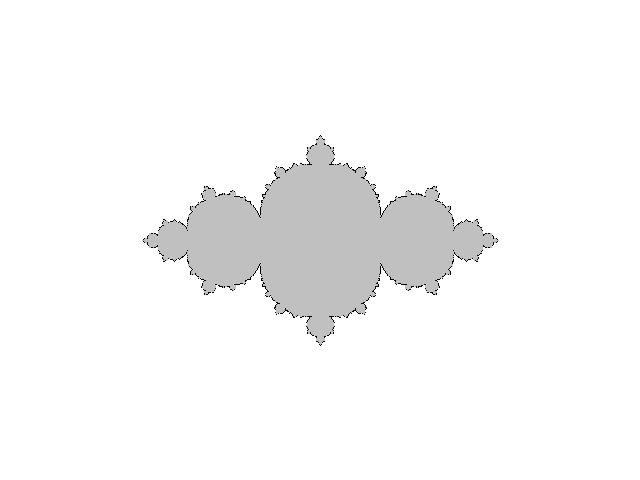}\hspace{\stretch{1}}
\caption{The bifurcation at $c_0=-0.75$: $c=-0.80;\, c=-0.70.$ }
\end{center}
\end{figure}
The main goal of this paper is  to prove the following generalization of Theorem II and III:

\begin{thm}\label{thm:twopetals}
   \emph{If $f_{c_0}^k$ has a parabolic fixed point with two petals then}
 \begin{enumerate}
   \item
      \emph{if $d(c_0)<4/3$ then there exist $K_{c_0}^+>K_{c_0}^->0$ such that}
         $$\lim_{c\rightarrow c_0^\pm}\frac{d'(c)}{|c_0-c|^{\frac32d(c_0)-2}}=-K_{c_0}^\pm,$$
   \item
      \emph{if $d(c_0)=4/3$ then there exists $K_{c_0}>0$ such that}
         $$\lim_{c\rightarrow c_0}\frac{d'(c)}{-\log|c_0-c|}=-K_{c_0},$$
   \item
      \emph{if $d(c_0)>4/3$ then there exists $K_{c_0}\in\R$ such that}
         $$\lim_{c\rightarrow c_0}d'(c)=K_{c_0}.$$
 \end{enumerate}
   \emph{In particular, if $d(c_0)>4/3$ then $d(c)$ is differentiable at $c_0$. Whereas, if $d(c_0)\leeq4/3$ then $d'(c)\rightarrow-\infty$, as $c\rightarrow c_0$.}
\end{thm}

Finally, concerning Hausdorff dimensions of connected parabolic Julia sets, we will use the important theorem, due to Anna Zdunik (see \cite{Zd}), which asserts that they greater than $1$.

The one petal case of Theorem I can also be generalized to small copies of the Mandelbrot set, but this renormalizable case will be considered in a future paper.

\textbf{Notation.}
$A\asymp B$ means that $K^{-1}\leeq A/B\leeq K$, where constant $K>1$ does not depend on $A$ and $B$ under consideration.

%$A(c,n)\approx B(c,n)$ means that for every $\ve>0$ there exist $N\in\N$, $\eta>0$ such that $1-\ve\leeq A(c,n)/B(c,n)\leeq 1+\ve$, where $n>N$, and $0<|c-c_0|<\eta$.

We will write $A_c(z)=\tilde o(B_c(z))$ if for every $\ve>0$ there exist $\eta>0$ and $U$ a neighborhood of the fixed point under consideration, such that $|A_c(z)/B_c(z)|\leeq \ve$, where $z\in U$ and $0<|c-c_0|<\eta$.

\section{Thermodynamical formalism}\label{sec:formalizmt}

%We shall repeat after \cite{J},... the main notion.

%Let $F$ be a polynomial of degree $d$ with connected Julia set, then there exists holomorphic a function $\Phi_F:\C\sms\overline\D\rightarrow\C\sms K(F)$ (called the B\"{o}ttcher coordinate) which is holomorphic, bijective, tangent to identity at infinity, and conjugates $T(s)=s^d$ to $F$ (i.e. $\Phi_F\circ T=F\circ\Phi_F$).

%If $K(F)$ is locally connected, then $\Phi_F$ has continuous extension to $\partial\D$ (Carath\'{e}odory's Theorem), thus we obtain semiconjugation $T|_{\partial\D}$ to $F|_K$. If there are no critical points in $K(F)$, then the set $K(F)$ is locally connected (more generally, $K(F)$ is locally connected if every critical point in $K(F)$ is preperiodic). Thus, $J_c$ is locally connected for every hyperbolic or parabolic parameter $c\in\mathcal M$.

The goal of this section is to establish the formula (\ref{eq:wzor}) which is the starting point of this work.

Let $F_\lambda$ be a holomorphic family of hyperbolic polynomials of degree $d$ with connected Julia sets, $\lambda\in\Lambda$ where $\Lambda$ is an open subset of $\C$. Note that we will be interested in families which are conjugated to $f_c^k$, $k\greq1$ (after reparametrization). Write
   $$\mathcal J_\lambda:=J(F_\lambda), \textrm{ and } \mc D(\lambda):=\HD(\mathcal J_\lambda).$$

If $\la_\dt\in\Lambda$, then there exist a holomorphic motion $\varphi$ of $\mc J_{\la_\dt}$ parametrized by $\Lambda$ (see \cite{Hub}), such that $\varphi_{\la_\dt}(s)=s$ and $\varphi_\la:\mc J_{\la_\dt}\rightarrow\mc J_\la$ conjugates $F_{\la_\dt}|_{\mc J_{\la_\dt}}$ to $F_\la|_{\mc J_\la}$ (i.e. $\varphi_\la\circ F_{\la_\dt}=F_\la\circ\varphi_\la$). Thus, the function $\la\mapsto\varphi_\la(s)$ is holomorphic for every $s\in\mc J_{\la_\dt}$.

%for every $\lambda\in\Lambda$ we can define the function $\Phi_\lambda\circ\Phi_{\lambda_0}^{-1}:J_{\lambda_0}\rightarrow J_\lambda$, which conjugates $F_{\lambda_0}|_{J_{\lambda_0}}$ to $F_\lambda|_{J_\lambda}$. Next, the map $(\lambda,z)\mapsto\Phi_\lambda\circ\Phi_{\lambda_0}^{-1}$ gives a holomorphic motion of $J_{\lambda_0}$ (see \cite{Hub}), therefore the function $\lambda\mapsto\Phi_\lambda\circ\Phi_{\lambda_0}^{-1}(z)$ is holomorphic for every $z\in J_{\lambda_0}$, as well as $\lambda\mapsto\Phi_\lambda(s)$ for every $s\in\partial\D$ and $\lambda\in\Lambda$

Now we use the thermodynamical formalism, which holds for hyperbolic rational maps. We will consider only such maps. Let $X=\mc J_{\la_\dt}$, $T=F_{\la_\dt}$, and let $\phi:X\rightarrow\R$ be a H\"{o}lder continuous function, to be often called a potential function. We will consider potentials of the form $\phi=-t\log|F'_\la(\varphi_\lambda)|$ for $\la\in\Lambda$.

{\em The topological pressure} can be defined as follows:
   $$P(T,\phi):=\lim_{n\rightarrow\infty}\frac{1}{n}\log\sum_{\overline s\in{T^{-n}(s)}}e^{S_n(\phi(\overline s))},$$
where $S_n(\phi)=\sum_{k=0}^{n-1}\phi\circ T^k$, and the limit exists and does not depend on $s\in X$.
If $\phi=-t\log|F'_\lambda(\varphi_\la)|$ and $\varphi_\la(\overline s)=\overline z$, then $e^{S_n(\phi(\overline s))}=|(F_\lambda^n)'(\overline z)|^{-t}$.
Hence
   $$P(T,-t\log|F'_\la(\varphi_\la)|) =\lim_{n\rightarrow\infty}\frac{1}{n}\log\sum_{\overline z\in{F_\lambda^{-n}(z)}}|(F_\lambda^n)'(\overline z)|^{-t}.$$
The function $t\mapsto P(T,-t\log|F'_\la(\varphi_\la)|)$ is decreasing from $+\infty$ to $-\infty$. In particular, there exists a unique $t_0$ such that $P(T,-t_0\log|F'_\la(\varphi_\la)|)\\=0$. By Bowen's Theorem (see \cite[Corollary 9.1.7]{PU} or \cite[Theorem 5.12]{Z}) we have
\begin{equation}\label{eq:acm}
   t_0=\mc D(\lambda).
\end{equation}
Thus, we have $P(T,-\mc D(\la)\log|F'_\la(\varphi_\la)|)=0$. Put $\phi_\la:=-\mc D(\la)\log|F'_\la(\varphi_\la)|$.

{\em The Ruelle} or {\em transfer operator} $\mathcal{L}_{\phi}:C^0(X)\rightarrow C^0(X)$ is defined as
   $$\mathcal{L}_{\phi}(u)(s):=\sum_{\overline s\in{T^{-1}(s)}}u(\overline s)e^{\phi(\overline s)}.$$
The Perron-Frobenius-Ruelle theorem \cite[Theorem 4.1]{Z} asserts that $\beta=e^{P(T,\phi)}$ is a single eigenvalue of $\mathcal{L}_{\phi}$ associated with an eigenfunction $\tilde h_{\phi}>0$. Moreover, there exists a unique probability measure $\tilde \omega_{\phi}$ such that $\mathcal{L}_{\phi}^*(\tilde \omega_{\phi})=\beta\tilde\omega_{\phi}$, where $\mathcal{L}_{\phi}^*$ is conjugated to $\mathcal{L}_{\phi}$.

For $\phi=\phi_\la$ we have $\beta=1$, and then $\tilde \mu_{\phi_\la}:=\tilde
h_{\phi_\la}\tilde\omega_{\phi_\la}$ is a $T$-invariant measure called an equilibrium state after normalization. But in the present work it will be more convenient not to normalize, contrarily to the tradition.
We denote by $\tilde{\omega_\la}$ and $\tilde{\mu_\la}$ the measures $\tilde\omega_{\phi_\la}$ and
$\tilde\mu_{\phi_\la}$ respectively (measures supported on $\mc J_{\la_\dt}$). Next, we take
$\mu_{\la}:=(\varphi_{\la})_*\tilde{\mu_{\la}}$, and $\omega_{\la}:=(\varphi_{\la})_*\tilde{\omega_{\la}}$ (image measures supported on
$\mc J_\la$).

So, the measure $\mu_\la$ is $F_\la$-invariant, whereas the measure $\omega_\la$ is called {\em $F_\la$-conformal with exponent
d(c)}, i.e. $\omega_\la$ is a Borel probability measure such that for every Borel subset $A\subset\mc J_\la$,
\begin{equation*}\label{eq:SP}
   \omega_\la(F_\la(A))=\int_A|F_\la'|^{\mc D(\la)}d\omega_\la,
\end{equation*}
provided $F_\la$ is injective on $A$.

%If $\beta=1$ (that is, $P(T,\phi)=0$, and in our case $\phi=-\mc D(\la)\log|F'_\lambda(\varphi_\la)|$), then $\tilde \mu_{\phi}:=\tilde h_{\phi}\tilde\omega_{\phi}$ is an invariant measure for $T$.

It follows from \cite[Proposition 6.11]{Z} or \cite[Theorem 4.6.5]{PU} that for every H\"{o}lder $\psi$ and $\phi$ at every $t\in\R$, we have
\begin{equation*}\label{eq:der}
   \frac{\partial}{\partial t}P(T,\psi+t\phi)= \frac{\int_{X}\phi\,d\tilde\mu_{\psi+t\phi}}{\tilde\mu_{\psi+t\phi}(X)}.
\end{equation*}

%Set $\phi_\la:=-\mc D(\la)\log|F'_\la(\varphi_\la)|$. We will use the notation $\tilde{\omega_\la}:=\tilde\omega_{\phi_\la}$ and $\tilde{\mu_\la}:=\tilde\mu_{\phi_\la}$ (measures supported on $X=\mc J_{\la_\dt}$). The bijections $\varphi_\la$ allows us to define related measures supported on $\mc J_\lambda$, which will be denoted by $\omega_\la$ and $\mu_\la$ respectively.

%The measure $\omega_\lambda$ is called {\em $F_\lambda$-conformal measure with exponent $\mc D(\lambda)$}, i.e. $\omega_\lambda$ is
%a Borel probability measure such that for every Borel subset $A\subset\mc J_\lambda$,
%\begin{equation*}
%\omega_\lambda(F_\lambda(A))=\int_A|F_\lambda'|^{\mc D(\lambda)}d\omega_\lambda,
%\end{equation*}
%provided $F_\lambda$ is injective on $A$.

%The measure $\mu_\lambda$ is $F_\lambda$-invariant, equivalent to $\omega_\lambda$, and is called the equilibrium state if is normalized. But we do not normalize neither $\mu_\lambda$ nor $\tilde{\mu_\lambda}$, which will be discussed in Section \ref{sec:miary}.

The above formula, and implicit function theorem (see \cite[Proposition 2.1]{J}) lead to
\begin{prop}\label{prop:wzor}\cite[Proposition 2.1]{HZ}
   Let us consider the family $\{F_\la\}_{\la\in\Lambda}$ as above, with the further assumptions that:
 \begin{itemize}
   \item[\textbf{-}]
   $\Lambda$ is symmetric with respect to $\R$,
   \item[\textbf{-}]
   $F_{\overline \la}(z)=\overline{F_\la(\overline z)}$.
 \end{itemize}

   If $\la\in\R$ is such that $F_\la$ is hyperbolic then
 \begin{equation}\label{eq:wzor}
    \mc D'(\la)=-\frac{\mc D(\la)}{\int_{\mc J_{\la_\dt}}\log|F'_\la(\varphi_\la)|d\tilde{\mu_\la}}\int_{\mc J_{\la_\dt}}\frac{\partial}{\partial \la}\log|F'_\la(\varphi_\la)|d\tilde{\mu_\la}.
 \end{equation}
\end{prop}

\section{The two petals case}\label{sec:twopetals}

\subsection{}
Let us assume that $f_{c_0}$, where $c_0\in\R$, has a parabolic cycle of length $k\greq1$, and let $\alpha_{c_0}$ be a point in this cycle. Then, $f_{c_0}^k$ has a parabolic fixed point with a multiplier $(f_{c_0}^k)'(\alpha_{c_0})=\pm1$. Since the critical orbit is real, the $+1$ case corresponds to one petal, and the $-1$ to two petals. Let us assume, from now on, that $(f_{c_0}^k)'(\alpha_{c_0})=-1$.

The parameter $c_0$ lies at the boundary of two hyperbolic components $W_l$, $W_r$ of $\intt\mathcal M$. The components $W_l$, $W_r$ are symmetric with respect to the real axis, $W_l$ is placed from the left side of $c_0$ whereas $W_r$ from the right.

Since $(f_{c_0}^k)'(\alpha_{c_0})\neq1$, there exists a neighborhood $U$ of $c_0$, and an holomorphic function $c\mapsto\alpha_c$ on $U$, such that for every $c\in U$ $\alpha_c$ is an element of a $k$-cycle for $f_c$. Since the Julia sets move holomorphically in $W_l$, and in $W_r$, we may assume that $W_l, W_r\subset U$.

This $k$-cycle is attracting for $c\in W_r$, and repelling for $c\in W_l$. Note that the attracting cycle for $c\in W_l$ has length $2k$.
Let
   $$\la(c):=(f_c^k)'(\alpha_c).$$
The function $\la$ is holomorphic on $U$, and its restriction to $W_r$, is a bijection onto $\D$, moreover $\la(c_0)=-1$. Since $\partial W_r$ is smooth at $c_0$ \cite{DH} we see that $\la'(c_0)\neq0$. Thus, from the fact that $|\la(c)|<1$, where $c\in W_r$, we obtain
   $$\la'(c_0)>0.$$
So, the function $\la$ is increasing in an interval $(c_0-\ve,c_0+\ve)$, for some $\ve>0$.

\subsection{}
We will need the following lemma:

\begin{lem}
The parabolic cycle of $f_{c_0}$ contains at least one point $\alpha$ where $$(f^k_{c_0})''(\alpha)\neq0.$$
\end{lem}

\begin{proof}
If $f_{c_0}$ has parabolic fixed point (i.e. $k=1$), then $f''_{c_0}\equiv2$. So, we can assume that $k>1$.

We will prove that if $(f^k_{c_0})''(\alpha)=0$, then $(f^k_{c_0})''(f_{c_0}(\alpha))\neq0$, where $\alpha$ is a point from the parabolic cycle.

Let $f^{-1}_{c_0}$ denote the inverse branch of $f_{c_0}$, such that $f^{-1}_{c_0}(f_{c_0}(\alpha))=\alpha$. Note that $(f^k_{c_0})'=-1$ at every point from the parabolic cycle, and $f''_{c_0}\equiv2$.

Since $(f^k_{c_0})'=(f^{-1}_{c_0}\circ f^k_{c_0}\circ f_{c_0})'= (f^{-1}_{c_0})'(f^k_{c_0}\circ f_{c_0})\cdot (f^k_{c_0})'(f_{c_0})\cdot f_{c_0}'$, we obtain
 \begin{multline}\label{eq:F''}
   (f^k_{c_0})''=(f^{-1}_{c_0})''(f^k_{c_0}\circ f_{c_0})\cdot((f^k_{c_0}\circ f_{c_0})')^2 \\ +(f^{-1}_{c_0})'(f^k_{c_0}\circ f_{c_0})\cdot (f^k_{c_0})''(f_{c_0})\cdot (f'_{c_0})^2 \\ +(f^{-1}_{c_0})'(f^k_{c_0}\circ f_{c_0})\cdot (f^k_{c_0})'(f_{c_0})\cdot f''_{c_0}.
 \end{multline}

Because $1=(f_{c_0}^{-1}\circ f_{c_0})'=(f_{c_0}^{-1})'(f_{c_0})\cdot f_{c_0}'$, we conclude that
   $$0=(f_{c_0}^{-1})''(f_{c_0})\cdot(f_{c_0}')^2+(f_{c_0}^{-1})'(f_{c_0})\cdot f_{c_0}''.$$
So, we have $(f_{c_0}^{-1})''(f_{c_0})\cdot (f_{c_0}^{-1})'(f_{c_0})<0$, therefore
   $$(f_{c_0}^{-1})''(f^k_{c_0}\circ f_{c_0}(\alpha))\cdot (f_{c_0}^{-1})'(f^k_{c_0}\circ f_{c_0}(\alpha))<0.$$
Finally, if $(f^k_{c_0})''(\alpha)=0$, then (\ref{eq:F''}) combined with the above inequality and the fact that $(f^k_{c_0})'(f_{c_0}(\alpha))\cdot f''_{c_0}(\alpha)=-2<0$, leads to $(f^k_{c_0})''(f_{c_0}(\alpha))\neq0$.
\end{proof}

So, we can assume that $(f_{c_0}^k)''(\alpha_{c_0})\neq0$. Later on, we will see that these estimates do not depend on choice of the point from the parabolic cycle of $f_{c_0}$ (see Proposition \ref{prop:reDF} and definitions from Sections \ref{sec:cylindry} and \ref{sec:generaltwopetals}`).

Conjugating $f_c^k$ by $t_c(z)= z-\alpha_c$ we obtain
   $$t_c\circ f_c^k\circ t_c^{-1}(z)=\la(c)z+a(c)z^2+b(c)z^3+O(z^4).$$
Since $(f_{c_0}^k)''(\alpha_{c_0})\neq0$, we see that $a(c_0)\neq0$. Then, after conjugating by $s_c(z)=z\cdot a(c)/a(c_0)$, we get
\begin{multline*}
   F_{\la(c)}(z):=s_c\circ t_c\circ f_c^k\circ t_c^{-1}\circ s_c^{-1}(z)\\
   =\la(c) z+a(c_0) z^2+b(c)\frac{a^2(c_0)}{a^2(c)}z^3+O(z^4).
\end{multline*}

Because the coefficients $a(c)$, $b(c)$ are polynomials, and of course $a(c)\rightarrow a(c_0)$, $b(c)\rightarrow b(c_0)$ when $c\rightarrow c_0$, we obtain
   $$F_{\la(c)}(z)=\la(c) z+a(c_0) z^2+b(c_0) z^3+O((c-c_0)z^3)+O(z^4).$$
Since $\la'(c_0)>0$, $c-c_0=O(\la(c)-\la(c_0))=O(\la(c)+1)$. Omitting $c$, $c_0$, and writing
   $$\delta_\la:=\la+1$$
(i.e. $\la=-1+\delta_\la$) we have
   $$F_\la(z)=\la z+a z^2+bz^3+O(\delta_\la z^3)+O(z^4).$$

Next, we get
\begin{equation}\label{eq:F0}
 \begin{array}{ll}
   F'_\la(z)=\la+2az+3bz^2+O(\delta_\la z^2)+O(z^3),\\
   F''_\la(z)=2a+6bz+O(\delta_\la z)+O(z^2).
 \end{array}
\end{equation}
and
\begin{equation}\label{eq:F1}
   \frac{\partial}{\partial\la}F_\la(z)=z+O(z^3),\;\;\;\;\;\; \frac{\partial}{\partial\la}F'_\la(z)=1+O(z^2).
\end{equation}
Note that
\begin{multline}\label{eq:F2}
   F_\la^2(z)=(1-2\delta_\la+\delta_\la^2)z+a(-1+\delta_\la)\delta_\la z^2-2(a^2+b)z^3\\+O(\delta_\la z^3)+O(z^4).
\end{multline}

Moreover, let us define
\begin{multline}\label{eq:ftilde}
   \tilde f_\la(z)=\tilde f_{\la(c)}(z):=s_c\circ t_c\circ f_c\circ t_c^{-1}\circ s_c^{-1}(z)\\
   =\frac{a(c_0)}{a(c)}z^2+2\alpha_c z+\frac{a(c)}{a(c_0)}(c+\alpha_c^2-\alpha_c).
\end{multline}
So, we have $\tilde f_\la^k=F_\la$.

\subsection{}\label{ssc:hm}
There exists $\ve>0$, such that the Julia set $\mc J_\la$ of $F_{\la}$ moves holomorphically on an open and disjoint sets $\Lambda_l\supset(-1-\ve,-1)$, and $\Lambda_r \supset(-1,-1+\ve)$.

Let us first fix $\la_l\in\Lambda_l$, $\la_r\in\Lambda_r$ (for instance taking the values corresponding to the center of the component). Then there exist two families of injections $$\varphi_\la^{l,r}:\mc J_{\la_{l,r}}\rightarrow\mc J_\la, \la\in\Lambda_{l,r},$$
conjugating $F_{\la_l}|_{\mc J_{\la_l}}$ to $F_\la|_{\mc J_\la}$ if $\la\in\Lambda_l$, and $F_{\la_r}|_{\mc J_{\la_r}}$ to $F_\la|_{\mc J_\la}$ if $\la\in\Lambda_r$.

Note that for every $s\in \mc J_{\la_{l,r}}$ the function $\la\rightarrow\varphi_\la(s)$ is holomorphic on $\Lambda_{l,r}$.

The families $\varphi_{\la}^{l,r}$, $\la\in\Lambda_{l,r}\cap\R$ are equicontinuous, and thus, taking uniform limit as $\la\rightarrow-1$ from the right or the left, there exist two functions $\varphi_{-1}^{l,r}:\mc J_{\la_{l,r}}\rightarrow\mc J_{-1}$ such that $\varphi_{-1}^{l,r}\circ F_{\la_{l,r}}=F_{-1}\circ\varphi_{-1}^{l,r}$

In the sequel, when the context is clear, we will allow ourselves to skip the subscript $l,r$, and denote $\la_{l,r}$ by $\la_\dt$.

\subsection{}
If $c\in\mathcal M\cap \R$ then the trajectory of the critical point of $f_c$ is included in the real line. So, the trajectories of the critical points of $F_{-1}$ (which tends to the parabolic points) are included in $\R$. Therefore the horizontal directions for the parabolic points are stable, whereas the vertical directions are unstable. Next, because
   $$F_{-1}^2(z)=z-2(a^2+b)z^3+O(z^4),$$
we conclude that $a^2+b>0$ (see the Fatou's flower theorem \cite{ADU}). The function $F_{-1}^2(z)$ is conjugated to $G(z)=z-2z^3+o(z^3)$ by $z\mapsto z\sqrt{a^2+b}$. So, let $A$ denote the scaling factor:
   $$A:=\sqrt{a^2+b}.$$

Let us assume that the parameter $\la$ is close to $-1$ (but $\la\neq-1$). Then, near the fixed point $0$, there exists a periodic orbit $\{p_\la^+,p_\la^-\}$ of period two for $F_\la$, such that $p_\la^{\pm}\rightarrow0$ when $\la\rightarrow-1$. We conclude from (\ref{eq:F2}) that $z=p_\la^{\pm}$ satisfies
   $$z^2=\frac{1}{a^2+b}\Big(-\delta_\la+\frac{\delta_\la^2}{2}\Big)+O(\delta_\la z)+O(z^3).$$
Since $|p_\la^\pm|$ is small, we see that $p_\la^\pm=O(\sqrt{\delta_\la})$, and then
   $$p_\la^\pm=\pm\frac{\sqrt{-\delta_\la}}{\sqrt{a^2+b}}+O(\delta_\la) =\pm\frac{\sqrt{-\delta_\la}}{A}+O(\delta_\la),$$
where in the case $\delta_\la>0$, we denote by $\sqrt{-\delta_\la}$ the principal square root i.e. $\sqrt{-\delta_\la}=i\sqrt{\delta_\la}$.

So, if $\la>-1$ (i.e. $\delta_\la>0$) then the cycle is repelling (hence $p_\la^\pm\in\mc J_\la$) and the periodic points are conjugated. For $\la<-1$ (i.e. $\delta_\la<0$) the cycle is attracting, and the periodic points are real.

\subsection{}
If $z\in \C^*$, then we shall assume that $\arg z\in(-\frac34\pi,\frac54\pi]$.
Let us define:
\begin{equation*}
 \begin{array}{llll}
   S^+(\theta,r):=\{z\in\C^*:|\arg z|\leeq\theta,\:|z|<r\}\cup\{0\},\\
   S^-(\theta,r):=\{z\in\C^*:|\arg z-\pi|\leeq\theta,\:|z|<r\}\cup\{0\},\\
   S^\uparrow(\theta,r):=\{z\in\C^*:|\arg z-\pi/2|\leeq\theta,\:|z|<r\}\cup\{0\},\\
   S^\downarrow(\theta,r):=\{z\in\C^*:|\arg z+\pi/2|\leeq\theta,\:|z|<r\}\cup\{0\}.
 \end{array}
\end{equation*}

Next, for $\delta_\la>0$:
\begin{equation*}
 \begin{array}{ll}
   \hat S^\uparrow(\theta,r):=\{p^+_\lambda+z:z\in S^\uparrow(\theta,r)\}\cap B(0,r),\\
   \hat S^\downarrow(\theta,r):=\{p^-_\lambda+z:z\in S^\downarrow(\theta,r)\}\cap B(0,r).
 \end{array}
\end{equation*}
Moreover, later on we will need:
\begin{equation*}
 \begin{array}{ll}
   \hat S^+(\theta,r):=\{\sqrt{\delta_\la}+z:z\in S^+(\theta,r)\}\cap B(0,r),\\
   \hat S^-(\theta,r):=\{-\sqrt{\delta_\la}+z:z\in S^-(\theta,r)\}\cap B(0,r).
 \end{array}
\end{equation*}

The Fatou's flower theorem (see \cite{ADU}) shows that the Julia set $\mc J_{-1}$ approaches the fixed point $0$ tangentially to the vertical direction. Now we state the perturbed version of this theorem.

\begin{lem}\label{lem:kat}
   For every $\theta>0$ there exist $r>0$ and $\eta>0$ such that
 \begin{enumerate}
    \item\label{lit:kat1}
      if $\delta_\la\in(-\eta,0]$, then
        $$(\mc J_\la\cap B(0,r))\subset (S^\uparrow(\theta,r)\cup S^\downarrow(\theta,r)),$$
    \item\label{lit:kat2}
      if $\delta_\la\in(0,\eta)$, then
        $$(\mc J_\lambda\cap B(0,r))\subset (\hat S^\uparrow(\theta,r)\cup \hat S^\downarrow(\theta,r))\subset (S^\uparrow(\theta,r)\cup S^\downarrow(\theta,r)).$$
 \end{enumerate}
\end{lem}

%\begin{col}\label{col:M}
%For every $\varepsilon>0$ there exists a neighborhood $V$ of $-1/2$ and $c_0<-3/4$ such that if $z\in V\cap J_c$, $c\in(c_0,-3/4]$ and $z\neq\alpha_c$, then the ratio of each two of the quantities
%$$|\im z|,\;\;\; |z-\alpha_c|,\;\;\; \frac12(\pi-|\Arg z|),$$
%belongs to $(1-\varepsilon,1+\varepsilon)$.
%\end{col}

%Repeating the proof of \cite[Lemma A.1]{J} we obtain:

%\begin{lem}\label{lem:1/2}
%There exist $c_0<-3/4$ such that for every $c\in(c_0,-3/4]$ we have
%$$\overline B(0,1/2)\subset K_c.$$
%Moreover if $c\in(c_0,-3/4)$, then $\overline B(0,1/2)\cap J_c=\emptyset$, whereas for $c=-3/4$ we have $\overline B(0,1/2)\cap J_c=\{-1/2,1/2\}$.
%\end{lem}

%Note that the above lemma implies that $|f_c'(z)|=|2z|>1$, thus $f_c$ is expanding for $n=1$.

\section{Fatou coordinates}\label{sec:fatou}

In this section we introduce coordinates that we will Fatou-coordinates, even if they do not conjugate to an exact translation. We prove, that in this coordinates the family $F_\la^2$ (after a modification) is close to the translation by 2, on the set $\mc J_\la$ near to the fixed point $0$. We shall use results of Buff and Tan Lei (see \cite{BT}).

For $\lambda\neq-1$ (i.e. $\delta_\la\neq0$) we define
   $$h_\la(z):=i\frac{\sqrt{-\delta_\la}}{A} \cdot\frac{(p_\la^+-p_\la^-)z}{(p_\la^++p_\la^-)z-2p_\la^+p_\la^-}.$$
If $\delta_\la>0$ then, as before, we take $\sqrt{-\delta_\la}=i\sqrt{\delta_\la}$. Notice that if $\delta_\la\rightarrow0$, then $h_\la(z)\rightarrow iz=:h_{-1}(z)$.

Write
   $$\hat F_\la=h_\la\circ F_\la\circ h^{-1}_\la.$$
Since $h_\la(0)=0$ and $h_\la(p_\la^\pm)=\pm i\sqrt{-\delta_\la}/A$, we conclude that
$0$ and $\pm\sqrt{\delta_\la}/A$ are the fixed points of $\hat F_\la^2$.

%First we conjugate $f^2_c$ by the map $\xi(z)=i(z+\frac12)$. Thus, the fixed points of $f_c^2$: $-\frac12\pm\sqrt{\rho_c}$, $\alpha_c=\frac12-\sqrt{1+\rho_c}$, are mapped onto $\pm i\sqrt{\rho_c}$, $i(1-\sqrt{1+\rho_c})$, respectively. Next, we would like to keep $\pm i\sqrt{\rho_c}$ and move $i(1-\sqrt{1+\rho_c})$ to 0, therefore we conjugate the family $\xi\circ f_c^2\circ\xi^{-1}$ by the family of homographies
%$$h_c:=\frac{i(1+\sqrt{1+\rho_c})z+\rho_c}{z+i(1+\sqrt{1+\rho_c})}.$$
%Write $F_c=h_c\circ\xi\circ f_c^2\circ\xi^{-1}\circ h^{-1}_c$. Hence, $\pm i\sqrt{\rho_c}$, $0$ are the fixed points of $F_c$, and $\alpha_c$ corresponds to 0.

The derivative $h_\la'$ is close to $i$ in a small neighborhood of $0$, whereas the distortion is close to 1. So, using Lemma \ref{lem:kat}, we see that $J(\hat F_\la)\cap B(0,r)$ is included in $S^+(\theta,r)\cup S^-(\theta,r)$. Moreover, if $\delta_\la>0$ then $J(\hat F_\la)\cap B(0,r)\subset\hat S^+(\theta,r)\cup\hat S^-(\theta,r)$.

%In order to avoid any misunderstanding, arguments of $\hat F_\la$ (and the Fatou coordinates) will be denoted by $\hat z$, while arguments of $F_\la$ by $z$.

We define the Fatou coordinates as follows (cf. \cite[Example 1]{BT}):
\begin{equation*}\label{eq:wspolrzednef}
   Z_\la(z):=\frac{1}{2\delta_\la}\log\Big(1-\frac{\delta_\la}{A^2z^2}\Big).
\end{equation*}
%Because we consider the principal value of the logarithm, $\log1=0$, we see that if $\rho_c\rightarrow0$, then $Z_c$ converges to the Fatou coordinate for $c=-3/4$,
%$$-\frac{1}{2\rho_c}\log\Big(1+\frac{\rho_c}{z^2}\Big)\rightarrow-\frac{1}{2 z^2}=:Z_{-3/4}(z).$$

The inverse functions are given by
\begin{equation}\label{eq:zet}
   z=Z^{-1}_\la(Z)=\frac{1}{A}\Big(\frac{\delta_\la}{1-e^{2\delta_\la Z}}\Big)^{\frac{1}{2}}.
\end{equation}

Set $S^\pm(\theta):=S^\pm(\theta,\infty)$, and then (cf. \cite{BT})
   $$S^\pm(\theta)_R:=S^\pm(\theta)\cap\{z\in\C:|\re z|>R\}.$$
If $\delta_\la<0$ then, as in \cite{Ji} we have $Z_\la(S^+(\theta))\subset S^-(2\theta)$, and next we can get $Z^{-1}_\la(S^-(2\theta)_{1/(4A^2r^2)})\supset S^\pm(\theta,r)$, where $\theta<\pi/8$ and $\delta_\la$ is close to $0$. In the case $\delta_\la>0$, we similarly obtain $Z_\la(\hat S^+(\theta))\subset S^-(2\theta)$ and $Z^{-1}_\la(S^-(2\theta)_{1/(4A^2r^2)})\supset\hat S^\pm(\theta,r)$. Thus, we can assume that $J(\hat F_\la)\cap B(0,r)$ is included in $Z^{-1}_\la(S^-(2\theta)_{1/(4A^2r^2)})$.

%Note that $Z_\la(S^+(\alpha))\subset S^-(2\alpha)$. Thus we can get $Z^{-1}_\la(S^-(2\alpha)_{1/(4A^2r^2)})\supset S^\pm(\alpha,r)$, where $\alpha<\pi/8$ and $\delta_\la$ is close to $0$. So, we can assume that $J(\hat F_\la)\cap B(0,r)$ is included in $Z^{-1}_\la(S^-(2\alpha)_{1/(4A^2r^2)})$.

\begin{lem}\label{lem:translation}
  For every $\theta\in(0,\pi/2)$ and $\varepsilon>0$ there exist $R>0$ and $\eta>0$ such that if $|\delta_\la|<\eta$, then
    $$\sup_{n\in\N,\:z\in Z^{-1}_\la(S^-(\theta)_R)} |Z_\la(\hat F^{-2n}_\la(z))-(Z_\la(z)-2n)|<\varepsilon n.$$
\end{lem}

\begin{proof}
Let us consider a family $G_\la$ of the form
   $$G_\la(z)=z+z(A^2z^2-\delta_\la)(1+s_\la(z)),$$
where $(\la,z)\mapsto s_\la(z)$ is holomorphic and $s_\la(0)=0$. We see from \cite[Lemma 5.1]{BT} that $G_\la$ can be approximated by the flow of differential equation $\dot z=z(A^2z^2-\delta_\la)$ ($Z_\la^{-1}$ is solution of this equation), and then for every $\theta\in(0,\pi/2)$, $\varepsilon>0$ and $R>0$ big enough, we can get
   $$\sup_{n\in\N,\:z\in Z^{-1}_\la(S^-(\theta)_R)} |Z_\la(G^{-n}_\la(z))-(Z_\la(z)-n)|<\varepsilon n.$$
The assumptions of \cite[Lemma 5.1]{BT} are satisfied because $G_\la$ is $\theta$-stable for every $\theta\in(0,\pi/2)$ (see \cite[Section 2, Example 1 (continued)]{BT}), and then the assumptions follows from \cite[proof of Lemma 5.3]{BT}.

Since $\pm \sqrt{\delta_\la}/A$ and $0$ are the fixed points of $\hat F_\la^2$, we conclude that in some neighborhood of 0, $\hat F_\la^2(z)-z$ can be written in the form $z(A^2z^2-\delta_\la)u_\la(z)$ where $u_\la(z)\neq0$, and next
   $$\hat F_\la^2(z)=z+z(A^2z^2-\delta_\la)u_\la(0)(1+w_\la(z)),$$
where $w_\la(0)=0$. Thus, analogously to $G_\la$, the family $(\hat F_\la^2)$ can be approximated by the flow of the equation $\dot z=u_\la(0)z(A^2z^2-\delta_\la)$, and we get
\begin{equation}\label{eq:u0}
   \sup_{n\in\N,\:z\in Z^{-1}_\la(S^-(\theta)_R)} |Z_\la(\hat F^{-2n}_\la(z))-(Z_\la(z)-u_\la(0)n)|<\varepsilon n.
\end{equation}
Since $F_{-1}^2(z)=z-2A^2z^3+o(z^3)$ and $h_{-1}(z)=iz$, we get
   $$h_{-1}\circ F_{-1}^2\circ h_{-1}^{-1}(z)=\hat F_{-1}^2(z)=z+2A^2z^3+o(z^3).$$
So, we see that $u_{-1}(0)=2$, and $u_\la(0)\rightarrow 2$ when $\la\rightarrow-1$. Thus, the assertion follows from (\ref{eq:u0}).
\end{proof}

\section{Cylinders}\label{sec:cylindry}

In this section, we will define a partition of a neighborhood of fixed/perio\-dic points of $F_\la$, which becomes the parabolic fixed points for $\la=-1$. Pieces of this partition will be called cylinders.

\subsection{}
Let us consider the family $f_c(z)=z^2+c$. First, we will define a partition of a subset of the Julia set $J_{-3/4}$.%, and next a partition of a subset of $J_c$, where $c\in[-1,0]$.

Let $f^{-1}$ denote the inverse branch of $f_{-3/4}$ which keeps the parabolic fixed point $-1/2$ ($f^{-1}$ can be defined between the external rays $\mc R(1/6),\mc R(5/12)$ and between the rays $\mc R(7/12),\mc R(5/6)$).

The set
   $$\mc R:=\{\mc R(1/3),\mc R(2/3)\},$$
consists of the external rays which land at the parabolic point. Next, let us consider two pairs of external rays which land at preparabolic points,
\begin{equation*}
 \begin{array}{ll}
    \mc R^+_0:=\{\mc R(5/24),\mc R(7/24)\},\\
    \mc R^-_0:=\{\mc R(17/24),\mc R(19/24)\}.
 \end{array}
\end{equation*}
Note that $f_{-3/4}^3(\mc R_0^\pm)=\mc R$. We define
\begin{equation*}
 \begin{array}{llll}
    \mc R^+_1:=f^{-1}(\mc R^-_0)=\{\mc R(17/48),\mc R(19/48)\},\\
    \mc R^-_1:=f^{-1}(\mc R^+_0)=\{\mc R(29/48),\mc R(31/48)\},\\
    \mc R^+_2:=f^{-1}(\mc R^-_1)=\{\mc R(29/96),\mc R(31/96)\},\\
    \mc R^-_2:=f^{-1}(\mc R^+_1)=\{\mc R(65/96),\mc R(67/96)\}.
 \end{array}
\end{equation*}
Next, we take $\mc R^+_{n+1}:=f^{-1}(\mc R^-_n)$ and $\mc R^-_{n+1}:=f^{-1}(\mc R^+_n)$.

We define the cylinder $C_1(-3/4)$ as the set which consists of four closed connected components: the part of the Julia set between the rays $\mc R(7/24)$, $\mc R(29/96)$ (i.e. between $\mc R^+_0$ and $\mc R^+_2$), the part between the rays $\mc R(17/48)$, $\mc R(19/48)$ (i.e. inside $\mc R^+_1$), and symmetrically in the lower half-plane, between the rays $\mc R(17/24)$, $\mc R(67/96)$ and $\mc R(31/48)$, $\mc R(29/48)$.

Next, for $n\greq1$ we take $C_{n+1}(-3/4):=f^{-1}(C_n(-3/4))$. Thus, $f_{-3/4}$ maps $C_{n+1}(-3/4)$ onto $C_{n}(-3/4)$ bijectively.
%Then, using external rays of the same arguments, we can also define partition of a subset of $J_c$ onto cylinders $C_n(c)$, where $c\in[-1,0]$.
Observe that this definition is slightly different from \cite{J} and \cite{Ji}.

\subsection{}\label{ssc:cylinders}
Let $\{\alpha^0,\alpha^1,...,\alpha^{k-1}\}$, where $\alpha^0=0$, be the set of all parabolic points of $F_{-1}$, where $\tilde f_\la(\alpha^j)=\alpha^{j+1}$, $0\leeq j\leeq k-2$ (it corresponds to the parabolic cycle of $f_{c_0}$).

Each parabolic point $\alpha^j$ belongs to a small Julia set $J(j)$ on which $F_{-1}$ is conjugated to $f_{-3/4}$. Since there are precisely two external rays of $\mc J_{-1}$ landing at $\alpha^j$, there is one-to-one correspondence between the external rays landing at preparabolic points of $\mc J_{-1}$ that are included in $J(j)$, and the external rays landing at preparabolic points of $J_{-3/4}$.

This correspondence allows us to define a partition of a neighborhood of $\alpha^j\in\mc J_{-1}$, for $0\leeq j\leeq k-1$, onto cylinders $\mc C^j_n(-1)\subset\mc J_{-1}$, where $n\greq1$.

Next, for $\la\in(-1-\ve,-1)$, there exist repelling fixed points $\alpha^j_\la$ of $F_\la$, such that $\alpha^j_\la\rightarrow\alpha^j$, when $\la\nearrow-1$. So, using external rays of the same arguments, we obtain partition of neighborhoods of $\alpha^j_\la$, onto cylinders $\mc C^j_n(\la)\subset\mc J_\la$.

For $\la\in(-1,-1+\ve)$, there exist repelling cycles $\{p_{\la}^{j+},p_{\la}^{j-}\}$, such that $p^{j\pm}_{\la}\rightarrow\alpha^j$, when $\la\searrow-1$.
Since the external rays landing at $p^{j\pm}_\la$ have the same arguments as the rays landing at $\alpha^j$, analogously as before, we obtain partition of neighborhoods of $p^{j\pm}_\la$, onto cylinders $\mc C^j_n(\la)\subset\mc J_\la$.
\begin{figure}[h!]
\begin{center}
  \includegraphics[width=8cm]{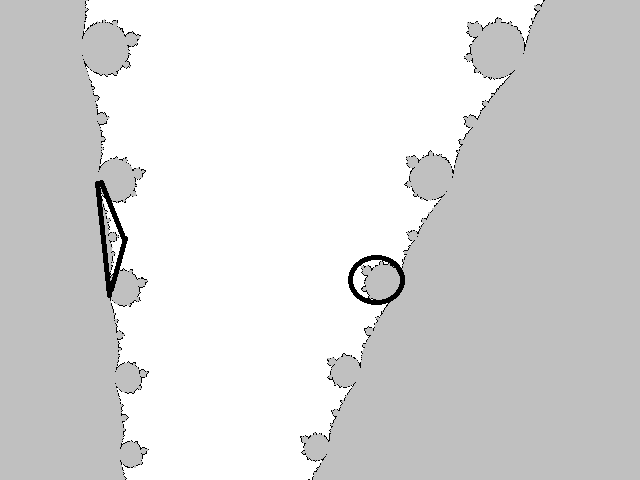}
  \caption{The "upper" part of a cylinder}
\end{center}
\end{figure}

The set of all points $\alpha^j_\la$ (if $\la<-1$) or $p^{j\pm}_\la$ (if $\la>-1$) will be denoted by 
\begin{equation*}
\textbf{P}(\la).
\end{equation*}

Note that $F_{\la}$ maps $\mc C^j_{n+1}(\la)$ onto $\mc C^j_{n}(\la)$ bijectively, for $\la\in(-1-\ve,-1+\ve)$. The union of the components of $\mc C^j_{n}(\la)$ included in the upper and lower half-plane, will be denoted by $\mc C^{j+}_{n}(\la)$ and $\mc C^{j-}_{n}(\la)$ respectively.

If $j=0$, the cylinders will be denoted by $\mc C_{n}(\la)$, $\mc C^+_{n}(\la)$ and $\mc C^-_{n}(\la)$.

For $N\in\N$, $n\greq1$, let us write
   $$\mathcal{M}_N^*(\la):=\bigcup_{n>N}\mc C_n(\la),\;\;\;\mathcal{M}_N(\la):=\overline{\mathcal{M}_N^*(\la)}=\bigcup_{n>N}\mc C_n(\la)\cup\{0\}.$$
and
   $$\textbf{C}_n(\la):=\bigcup_{0\leeq j\leeq k-1}\mc C^j_n(\la),\;\;\textbf{M}_N^*(\la):=\bigcup_{n>N} \textbf{C}_n(\la),\;\;\textbf{M}_N(\la):=\overline{\textbf{M}_N^*(\la)}.$$
Note that $\textbf{M}_N(\la)=\textbf{M}_N^*(\la)\cup\textbf{P}(\la)$. Next
   $$\textbf{B}_N(\la):=\mc J_\la\sms\textbf{M}_N(\la).$$
If $\la=\la_\dt$ we will write $\textbf{B}_N$, $\textbf{M}_N$, $\textbf{C}_n$, $\textbf{P}$, $\mc C_n$ etc.
%The sets $\mc M_N^*(c)$, $\mc M_N(c)$ and $\mc B_N(c)$ (for the family $f_c(z)=z^2+c$) we define analogously.
\begin{figure}[h!]
\begin{center}
  \includegraphics[width=10cm]{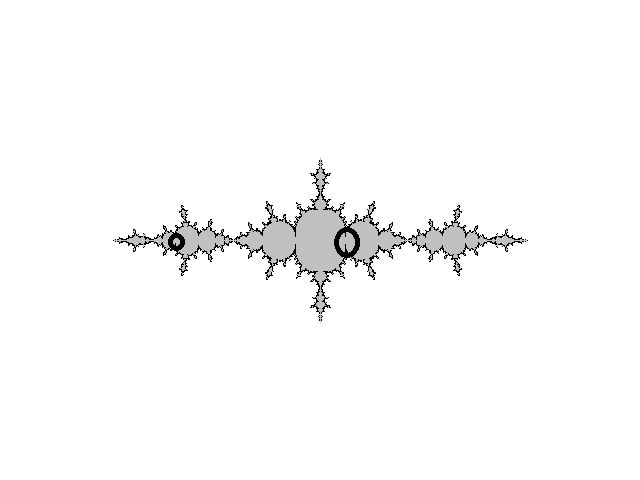}
  \caption {$c_0=-5/4$. Cycle of order $2$. The set $\textbf{M}_N$.}
\end{center}
\end{figure}

\subsection{}
Instead of the diameters of the cylinders, we use a quantity which will be called {\em the size of the cylinder} and denoted by $|\mc C_n(\la)|$, and which is more or less the diameter of $\mc C^+_n(\la)$ (or $\mc C^-_n(\la)$, by symmetry). More precisely, let $z_i$, where $i\in\N$, be the preparabolic point of $F_\la$ which is included in the small Julia set containing $\alpha^0=0$, and corresponds with the common landing point of the external rays from $\mc R^+_i$. Then, for $n\greq1$
   $$|\mc C_n(\la)|:=\frac{1}{2}\,\big|z_{n+1}-z_{n-1}\big|.$$

We have already seen that the set of the trajectories of the critical points is included in the real line. Thus, Lemma \ref{lem:kat} and the Koebe Distortion Theorem (see \cite{G}) imply the following propositions (cf. \cite[Lemma 5.6]{Ji}):

\begin{prop}\label{prop:dyst}
   There exist $K>1$ and $\eta>0$ such that if $|\delta_\la|<\eta$ and $n\greq1$, then
      $$K^{-1}\diam(\mc C_n^\pm(\la))\leeq|\mc C_n(\la)|\leeq K\diam(\mc C_n^\pm(\la)).$$
\end{prop}

Note that the constant $K$ depends on the family $(F_\la)$.

\begin{prop}\label{prop:C/im}
For every $\varepsilon>0$ there exists $N\in\N$ and $\eta>0$ such that
 \begin{equation*}
    |\mc C_n(\la)|\leeq\ve|\im z|,
 \end{equation*}
where $z\in{\mc C_n(\la)}$, $n>N$ and $|\delta_\la|<\eta$.
\end{prop}

\begin{prop}\label{prop:aad}
   For every $\ve>0$ there exist $N\in\N$, $\eta>0$ such that if $n>N$, $k\greq1$ and $|\delta_\la|<\eta$, then
 \begin{enumerate}
    \item\label{pit:aad1}
       $(\frac12-\varepsilon)|F_\la^2(z)-z|\leeq|\mc C_n(\la)|\leeq(\frac12+\varepsilon)|F_\la^2(z)-z|$, provided $z\in \mc C_n(\la)$,
    \item\label{pit:aad2}
       $(1-\varepsilon)\frac{|\mc C_n(\la)|}{|\mc C_{n+k}(\la)|}\leeq|(F_\la^k)'(z)|\leeq(1+\varepsilon)\frac{|\mc C_n(\la)|}{|\mc C_{n+k}(\la)|}$, provided $z\in \mc C_{n+k}(\la)$.
 \end{enumerate}
\end{prop}

%\begin{proof}
%Close to the fixed point $\alpha_c$ the Julia set is included in the angles $S_c^\pm(\theta,r)$ (see Lemma \ref{lem:kat}). If $\theta$ is small, then we see that diameters of the cylinders are small with respect to distance from the real line. Thus, we conclude that the distortion of $f_c$, and all inverse branches of $f_c^k$ is close to 1, and the lemma follows.
%\end{proof}

Let $\hat{\mc C}_n(\la)\subset J(\hat F_\la)$ denote the image of $\mc C_n(\la)$ under the map conjugating $F_\la$ to $\hat F_\la$. In the same way as in \cite[Lemma 5.2]{J}, we can prove:
\begin{prop}\label{prop:przes}
   For every $\varepsilon>0$ there exist $N\in\N$, $\eta>0$ such that for every $n\greq N$ and $|\delta_\la|<\eta$ if $z\in\hat{\mc C}_n(\la)$, then
 \begin{enumerate}
    \item\label{lit:przes1}
      $-(1+\varepsilon)n\leeq\re Z_\la(z)\leeq-(1-\varepsilon)n$,
    \item\label{lit:przes2}
      $|\im Z_\la(z)|\leeq -\varepsilon\re Z_\la(z)$.
 \end{enumerate}
\end{prop}

The next two results precise the location of the cylinder $\mc C_n(\la)$; they follow from \cite[Section 5]{Ji}

\begin{prop}\label{prop:rnn}
   For every $\varepsilon>0$ there exist $N\in\N$, $\eta>0$ such that for every $z\in{\mc C_n(\la)}$ and $n>N$ we have
 \begin{enumerate}
    \item\label{pit:rnn1}
       if $0<|\delta_\la|<\eta$ then
     \begin{equation*}
        (1-\varepsilon)\frac{\delta_\la\cdot e^{2n\delta_\la-\varepsilon n|\delta_\la|}}{e^{2n\delta_\la}-1}\leeq A^2\im^2z\leeq (1+\varepsilon)\frac{\delta_\la\cdot e^{2n\delta_\la+\varepsilon n|\delta_\la|}}{e^{2n\delta_\la}-1}.
     \end{equation*}
    \item\label{pit:rnn2}
       if $\delta_\la\in(0,\eta)$ then
     \begin{equation*}
        (1-\varepsilon)\frac{\delta_\la}{1-e^{-2n\delta_\la}}\leeq A^2\im^2z\leeq (1+\varepsilon)\frac{\delta_\la}{1-e^{-2n\delta_\la}}.
     \end{equation*}
 \end{enumerate}
\end{prop}

The above implies the following corollary:
\begin{col}\label{col:odl}
   There exist $K>1$, $\eta>0$ such that for every $|\delta_\la|<\eta$ and $z\in{\mc C_n(\la)}$, $n\greq1$
 \begin{enumerate}
    \item\label{cit:odl1}
       if $n|\delta_\la|\leeq1$ then $\:\:\:\frac1K\,n^{-1/2}\leeq|\im z|\leeq Kn^{-1/2}$,
    \item\label{cit:odl2}
       if $n\delta_\la>1$ (so $\delta_\la>0$) then $\:\:\:\frac1K\,{\delta_\la}^{1/2}\leeq|\im z|\leeq K{\delta_\la}^{1/2}$.
 \end{enumerate}
   Moreover, for every $\ve>0$ there exists $N\in\N$ such that
 \begin{enumerate}
    \item[(3)]\label{cit:odl3}
       if $n\delta_\la<-1$ (so $\delta_\la<0$) and $n>N$, then \\ $\frac1K\,|\delta_\la|^{1/2}e^{(1+\ve)n\delta_\la}\leeq|\im z|\leeq K|\delta_\la|^{1/2}e^{(1-\ve)n\delta_\la}$.
 \end{enumerate}
   %In particular for every $r>0$ there exist $N\in\N$, $\eta>0$ such that for $|\delta_\la|<\eta$, we have $\mathcal M_N(\la)\subset B(0,r)$.
\end{col}
In particular, we can assume that $\mathcal M_N(\la)$ is included in arbitrary small neighborhood of $0$.

%Repeating the proof of \cite[Corollary 5.4]{J}, we see that

%\begin{col}\label{col:N}
%There exists $c_0<-3/4$ and for every $N\in\N$ there exists $\varepsilon>0$ such that if $c\in(c_0,-3/4]$, then
%$$\big(J_c\cap B(0,1/2+\varepsilon)\big)\subset\big(\mathcal M_N(c)\cup-\mathcal M_N(c)\big).$$
%\end{col}

The next results give precise estimate of the size of the cylinder. They also follow from \cite[Section 5]{Ji}.

\begin{prop}\label{prop:rn}
   For every $\ve>0$ there exist $N\in\N$, $\eta>0$ such that if $z\in{\mc C_n(\la)}$, $n>N$ and $|\delta_\la|<\eta$, then
 \begin{equation*}
    (1-\varepsilon)A^2|\im z|^3e^{-2n\delta_\la-\ve n|\delta_\la|}\leeq|\mc C_n(\la)|\leeq(1+\varepsilon)A^2|\im z|^3e^{-2n\delta_\la+\ve n|\delta_\la|}.
 \end{equation*}
\end{prop}

%\begin{proof}
%As in \cite[proof of Lemma 6.1]{J} we see that for every $\varepsilon>0$ there exist $N$ and $c_0$ for which
%$$(1-\varepsilon)|(Z_c^{-1})'(Z)|<\frac12|F_c(z)-z|<(1+\varepsilon)|(Z_c^{-1})'(Z)|,$$
%provided $Z_c^{-1}(Z),z\in\hat C_n(c)$, $n>N$ and $c\in(c_0,-\frac34]$.

%Since modulus of derivative and distortion of the maps conjugating $f_c^2$ and $F_c$ is close to 1, quantity $|\hat F_c(z)-z|$ can be replaced by $|f_c^2(z)-z|$. Therefore, taking as $z$ the point $z_1$ used in the definition of size of cylinder, we get
%\begin{equation}\label{eq:szaac}
%(1-\varepsilon)|(Z_c^{-1})'(Z)|<|C_n(c)|<(1+\varepsilon)|(Z_c^{-1})'(Z)|.
%\end{equation}
%By (\ref{eq:zet}) we conclude that
%\begin{equation*}\label{roznz}
%|(Z_c^{-1})'(Z)|= \Big|\Big(\frac{\rho_ce^{2\rho_cZ}} {1-e^{2\rho_cZ}}\Big)^{\frac{3}{2}}{e^{-2\rho_c{Z}}}\Big|=|z|^3e^{-2\rho_c\re{Z}}.
%\end{equation*}
%Combining this with (\ref{eq:szaac}), we get the following
%$$(1-\varepsilon)|z|^3e^{-2\rho_c\re{Z}}< |C_n(c)|<(1+\varepsilon)|z|^3e^{-2\rho_c\re{Z}}.$$
%Because $|z|$ can be replaced by $|z-\alpha_c|$, so also by $|\im z|$ (see properties of conjugating maps and Corollary \ref{col:M}), it remains to estimate $\re(Z)$ by $-n$ up to a constant close to one (see Lemma \ref{lem:przes}).
%\end{proof}

\begin{prop}\label{prop:mod}
   There exist $K>1$, $\eta>0$ such that if $z\in{\mc C_n(\la)}$, $n\greq1$ and $|\delta_\la|<\eta$, then
 \begin{enumerate}
    \item\label{pit:mod1}
       hyperbolic estimate: if $n|\delta_\la|>1$, then
     \begin{equation*}
        \frac{1}{K}|\delta_\la|^{3/2}e^{-Kn|\delta_\la|}\leeq|\mc C_n(\la)|\leeq K|\delta_\la|^{3/2}e^{-\frac{1}{K}n|\delta_\la|},
     \end{equation*}
    \item\label{pit:mod2}
       parabolic estimate: if $n|\delta_\la|\leeq1$, then
     \begin{equation*}
        \frac{1}{K}n^{-3/2}\leeq|\mc C_n(\la)|\leeq{K}n^{-3/2}.
     \end{equation*}
 \end{enumerate}
\end{prop}

%\begin{proof}
%Because it is enough to consider cylinders $C_n(c)$ with indexes greater than some fixed $N\in\N$ (remains finitely many cylinders, so we can simply change constant in statement (\ref{pit:mod2})), proposition follows from Corollary \ref{col:odl} and Lemma \ref{lem:rn}.
%\end{proof}

%Repeating the proof of \cite[Corollary 6.4]{J}, we can easily obtain:

The above implies the following corollary:

\begin{col}\label{col:mod}
   There exist  $K,K_1,K_2>0$ and $\eta>0$ such that for every $n\greq1$ and $|\delta_\la|<\eta$, we have
 \begin{equation*}
    K_1|\delta_\la|^{3/2}e^{-Kn|\delta_\la|}\leeq|\mc C_n(\la)|\leeq K_2n^{-3/2}.
 \end{equation*}
\end{col}

We end this section by the following lemma, which will be used in Section \ref{sec:generaltwopetals}.
Recall that if $w\in\C^*$, then we assume that $\arg w\in[-\frac34\pi,\frac54\pi)$.
\begin{lem}\label{lem:arg}
   For every $\ve>0$ there exist $\eta>0$, and $\tilde n\in\N$, such that
 \begin{equation*}
    |\arg (F_\la^{j})'(z)|\leeq\ve,
 \end{equation*}
   where $|\delta_\la|<\eta$, $z\in \mc C_{n}(\la)$, and $j\in\N$ is an even number, satisfying $n>\tilde n+j$. Analogously, if $j$ is an odd number, then $\arg (F_\la^{j})'(z)$ is close to $\pi$.
\end{lem}

\begin{proof}
Let $\tilde n$ be an integer large enough, and let $z\in\mc C_n(\la)$, where $n>\tilde n$.
Thus, we can assume that for $z\in\mc M_{\tilde n}(\la)$, $F_\la'(z)$ is close to $-1$, so we see that $\arg F_\la'(z)$ is close to $\pi$. Then
\begin{equation*}
   \arg (F_\la^2)'(z)=\arg F_\la'(z)+\arg F_\la'(F_\la(z))-2\pi
   =\arg F_\la'(z)-\arg F_\la'(\overline{F_\la(z)}).
\end{equation*}
Let $z\in\mc C^\pm_n(\la)$, then $\overline{F_\la(z)}\in\mc C^\pm_{n-1}(\la)$. Since $z$ and $\overline{F_\la(z)}$ are separated from the critical points, we conclude that $F_\la'(z)$, and $F_\la'(\overline{F_\la(z)})$, are separated from zero. Therefore,
\begin{equation*}
   |\arg (F_\la^2)'(z)|
   <K_1|F_\la'(z)-F_\la'(\overline{F_\la(z)})|<K_2|z-\overline{F_\la(z)}|\leeq K_3|\mc C_n(\la)|.
\end{equation*}
We have $F^{2l}_\la(z)\in\mc C_{n-2l}(\la)$, so the chain rule and Corollary \ref{col:mod} leads to
\begin{equation*}
   \arg (F_\la^{2m})'(z)\leeq K_4\sum_{l=0}^{m-1}(n-2l)^{-3/2}\leeq K_5(n-2m)^{-1/2}.
\end{equation*}
So, for $n-2m>\tilde n$, we see that $(F_\la^{2m})'(z)$ is "close" to real positive numbers, whereas $(F_\la^{2m-1})'(z)$ is "close" to real negative numbers.
\end{proof}

\section{Invariant measures}\label{sec:miary}

We have already mentioned in Section \ref{sec:formalizmt} that for a hyperbolic polynomial there exists a unique normalized invariant measure equivalent to conformal (Hausdorff) measure.

In the parabolic case the situation is more complicated and depends on the dimension of the Julia set. More precisely Denker and Urba\'nski (see \cite{DU} and \cite{Ui}) have shown that if the polynomial $f$ has one and only one parabolic fixed point whose Leau-Fatou flower has $p$ petals then:
\begin{itemize}
  \item[-] if the Julia set of $f$ has dimension greater than $\frac{2p}{p+1}$, there exists a unique normalized invariant measure equivalent to the conformal measure,
  \item[-] if the dimension belongs to $[1,\frac{2p}{p+1}]$, then there exists a unique (up to a positive factor) $\sigma$-finite invariant measure equivalent to the conformal measure.
\end{itemize}

So, in our case, if $f_{c_0}$ has a parabolic cycle with two petals, then the existence of a normalized invariant measure absolutely continuous to Hausdorff is equivalent to $d(c_0)>4/3$.

Having this in mind, we now estimate $\mu(\mc C_n(\la))=\tilde \mu_n(\mc C_n)$, for $\la\neq-1$ (close to $-1$), and $n$ large.

The first result follows from \cite[Lemma 7.3]{J} and \cite[Lemma 6.2]{Ji}

\begin{lem}\label{lem:L}
   There exist $D>1$ and $\eta>0$ such that for every $N\in\N$ there exists a constant $K(N)>1$ for which
      $$D^{-1}\leeq\frac{d{\mu_\la}}{d{\omega_\la}}\Big|_{\textbf{B}_N(\la)}\leeq K(N),$$
   provided $|\delta_\la|<\eta$.
\end{lem}

\begin{col}\label{col:lim}
   For every neighborhood $U$ of the parabolic cycle, and every continuous function $g:\C\rightarrow\R$ such that $g|_U=0$, we have
      $$\lim_{\la\rightarrow-1}\int g\,d{\mu_\la}=\int g\,d\mu_{-1}.$$
\end{col}

%Put $h^*:=h^*_{-3/4}(-1/2)$.

The next lemmas are Lemma 6.4 and 6.5 in \cite{Ji}.

\begin{lem}\label{lem:sumofmeasures}
   There exists constant $H>0$, such that for every $\varepsilon>0$ there exist $N\in\N$, $\eta>0$ such that if $n>N$ and $|\delta_\la|<\eta$ then
      $$(1-\varepsilon)H\sum_{m=n}^{\infty}\tilde{\omega_\la}(\mc C_m)
      \leeq\tilde{\mu_\la}(\mc C_n)\leeq(1+\varepsilon)H\sum_{m=n}^{\infty}\tilde{\omega_\la}(\mc C_m).$$
\end{lem}
Note that the constant $H$ depends on the choice of the fixed point $\alpha_c$, that was moving to $0$.

For $h\greq1$, $|\ve|\leeq1$ and $u\neq0$ let us define
\begin{equation}\label{eq:A}
   \Lambda_\ve^h(u):=\Big(\frac{e^{u}}{|e^{2u}-1|^{3/2}}\Big)^he^{\ve u}= \Big(\frac{e^{-2u}}{|1-e^{-2u}|^{3/2}}\Big)^he^{\ve u}.
\end{equation}
If $h=\mc D(\la)$, then $\mc D(\la)-\ve$ is separated from zero for $\la$ close to $-1$. Thus, for all sufficiently small $s>0$, we get
   $$\int_{s}^{\infty}\Lambda_\delta^{\mc D(\la)}(u)du\asymp
   \int_{-\infty}^{-s}\Lambda_\delta^{\mc D(\la)}(u) du\asymp s^{-3\mc D(\la)/2+1}.$$
\begin{lem}\label{lem:est}
   There exists $M>0$, and for every $\varepsilon\in(0,1)$ there exist $N\in\N$, $\eta>0$ such that
 \begin{enumerate}
    \item\label{lit:est1}
       if $\delta_\la\in(0,\eta)$, and $n>N$, then
     \begin{equation*}
        (1-\varepsilon)M\int^\infty_{n\delta_\la}\Lambda_{-\varepsilon}^{\mc D(\la)}(u) du\leeq\frac{\tilde{\mu_\la}(\mc C_n)}{\delta_\la^{3\mc D(\la)/2-1}}\leeq (1+\varepsilon)M\int^\infty_{n\delta_\la}\Lambda_{\varepsilon}^{\mc D(\la)}(u) du,
     \end{equation*}
    \item\label{lit:est2}
       if $\delta_\la\in(-\eta,0)$, and $n>N$, then
     \begin{equation*}
        (1-\varepsilon)M\int_{-\infty}^{n\delta_\la} \Lambda_{-\varepsilon}^{\mc D(\la)}(u) du\leeq\frac{\tilde{\mu_\la}(\mc C_n)}{|\delta_\la|^{3\mc D(\la)/2-1}}\leeq (1+\varepsilon)M \int_{-\infty}^{n\delta_\la}\Lambda_{\varepsilon}^{\mc D(\la)}(u) du.
     \end{equation*}
 \end{enumerate}
\end{lem}

Analogously as in \cite[Lemma 6.6 (2)]{Ji}, we get

\begin{lem}\label{lem:mm}
   If $\mc D(\la)<4/3$, then for every $\ve>0$ there exists $\eta>0$ such that
 \begin{enumerate}
   \item\label{lit:mm1}
       if $\delta_\la\in(0,\eta)$, then
     \begin{equation*}
        \sum_{n=1}^{\infty}\bigg|\tilde{\mu_\la}(\mc C_n)-M\delta_\la^{\frac{3}{2}\mc D(\la)-1}\int^{\infty}_{n\delta_\la}\Lambda_0^{\mc D(\la)}(u)du\bigg|\leeq \ve\delta_\la^{\frac32\mc D(\la)-2},
     \end{equation*}
    \item\label{lit:mm2}
       if $\delta_\la\in(-\eta,0)$, then
     \begin{equation*}
        \sum_{n=1}^{\infty}\bigg|\tilde{\mu_\la}(\mc C_n)-M\delta_\la^{\frac{3}{2}\mc D(\la)-1}\int_{-\infty}^{n\delta_\la}\Lambda_0^{\mc D(\la)}(u)du\bigg|\leeq \ve|\delta_\la|^{\frac32\mc D(\la)-2},
     \end{equation*}
 \end{enumerate}
\end{lem}

Lemma \ref{lem:est}, or Lemma \ref{lem:sumofmeasures} and Proposition \ref{prop:mod}, lead to

\begin{prop}\label{prop:meas}
   There exist $K>1$ and $\eta>0$ such that
 \begin{enumerate}
   \item\label{pit:meas1}
      if $n|\delta_\la|>1$, then
    \begin{equation*}
       \frac1K|\delta_\la|^{\frac32\mc D(\la)-1}e^{-Kn|\delta_\la|}\leeq \tilde{\mu_\la}(\mc C_n)\leeq K|\delta_\la|^{\frac32\mc D(\la)-1}e^{-\frac1K n|\delta_\la|},
    \end{equation*}
   \item\label{pit:meas2}
      if $n|\delta_\la|\leeq1$, then
    \begin{equation*}
       \frac1K n^{-\frac32\mc D(\la)+1}\leeq \tilde{\mu_\la}(\mc C_n)\leeq K n^{-\frac32\mc D(\la)+1},
    \end{equation*}
 \end{enumerate}
   where $|\delta_\la|<\eta$ and $n\greq1$.
\end{prop}

Lemma \ref{lem:est} leads to

\begin{prop}\label{prop:Csmall}
   There exists $K>0$ and for every $\ve>0$ there exist $\alpha>0$ and $N\in\N$ such that
 \begin{equation*}
     (1-\ve)K n^{-\frac32\mc D(\la)+1}\leeq \tilde{\mu_\la}(\mc C_n)\leeq (1+\ve)K n^{-\frac32\mc D(\la)+1},
 \end{equation*}
   where $N<n\leeq[\alpha/|\delta_\la|]$.
\end{prop}

From the construction of the invariant measure \cite[Section 6]{Ji} we get:

\begin{lem}\label{lem:C}
   For every $\ve>0$ there exist $N\in\N$ and $\eta>0$, such that
 $$(1-\ve)\:\tilde\mu_\la(\mc C_n^j)\leeq\tilde\mu_\la(\mc C_n)\leeq(1+\ve)\:\tilde\mu_\la(\mc C_n^j),$$
   where $n>N$, $|\delta_\la|<\eta$, and $1\leeq j\leeq k-1$.
\end{lem}

The next results follow from Proposition \ref{prop:meas} and Lemma \ref{lem:C}.

\begin{lem}\label{lem:mj}
   There exist $K>1$, such that for every $N\in\N$, there exist $\eta>0$, such that
 \begin{enumerate}
   \item\label{lit:mj1}
       if $\mc D(\la)<4/3$ and $0<|\delta_\la|<\eta$, then
     \begin{equation*}
        \frac{K^{-1}}{2-\frac{3}{2}\mc D(\la)}|\delta_\la|^{\frac{3}{2}\mc D(\la)-2}\leeq\tilde{\mu_\la}(\textbf{M}_N)\leeq \frac{K}{2-\frac{3}{2}\mc D(\la)}|\delta_\la|^{\frac{3}{2}\mc D(\la)-2},
     \end{equation*}
    \item\label{lit:mj2}
       if $\mc D(\la)=4/3$ and $0<|\delta_\la|<\eta$, then
     \begin{equation*}
        -K^{-1}\log|\delta_\la|\leeq\tilde{\mu_\la}(\textbf{M}_N)\leeq -K\log|\delta_\la|,
     \end{equation*}
   %For every $\ve>0$ there exist $N\in\N$, $\eta>0$ such that
    \item\label{lit:mj3}
       if $\mc D(\la)>4/3$ and $|\delta_\la|<\eta$, then
     \begin{equation*}
        \frac{K^{-1}}{\frac{3}{2}\mc D(\la)-2}N^{-\frac{3}{2}\mc D(\la)+2}\leeq \tilde{\mu_\la}(\textbf{M}_N)\leeq \frac{K}{\frac{3}{2}\mc D(\la)-2}N^{-\frac{3}{2}\mc D(\la)+2}.
     \end{equation*}
       In particular, for every $\ve>0$ there exist $N\in\N$, $\eta>0$ such that $\tilde{\mu_\la}(\textbf{M}_N)<\ve$, where $|\delta_\la|<\eta$.
 \end{enumerate}
\end{lem}

\begin{lem}\label{lem:a/d}
   For every $\ve>0$, $N\in\N$ and $\alpha>0$ there exists $\eta>0$ such that if $|\mc D(\la)-4/3|\leeq \eta$, and $0<|\delta_\la|<\eta$ then
 \begin{equation*}
    \tilde{\mu_\la}(\textbf{M}_{[\alpha/|\delta_\la|]})\leeq \ve\tilde{\mu_\la}(\textbf{M}_N\sms\textbf{M}_{[\alpha/|\delta_\la|]}).
 \end{equation*}
   Moreover, if $\mc D(\la)\rightarrow4/3$ when $\delta_\la\rightarrow0$, then
 \begin{equation*}
    \lim_{\delta_\la\rightarrow0} \tilde{\mu_\la}(\textbf{M}_N\sms\textbf{M}_{[\alpha/|\delta_\la|]})=\infty.
 \end{equation*}
\end{lem}

\begin{lem}\label{lem:ee}
   There exist $K>0$, $\eta>0$ and $\ve_0>0$ such that such that for every $0<\ve<\ve_0$ and $N\in\N$
 \begin{enumerate}
  \item\label{lit:ee1}
     $\;\;\;\;\sum_{n=N+1}^{\infty}\:(e^{\ve n|\delta_\la|}-1)\:\tilde{\mu_\la}(\textbf{C}_n)\leeq\ve K\tilde{\mu_\la}(\textbf{M}_N)$,
  \item\label{lit:ee2}
     $\;\;\;\;\:\:\:\:\:\:\sum_{n=N+1}^{\infty}\:e^{\ve n|\delta_\la|}\:\tilde{\mu_\la}(\textbf{C}_n)\leeq K\tilde{\mu_\la}(\textbf{M}_N)$,
 \end{enumerate}
   where $|\delta_\la|<\eta$.
\end{lem}

\section{Estimation of the denominator}\label{sec:den}

In order to prove Proposition \ref{prop:den}, we will need the following lemma:

\begin{lem}\label{lem:exp}
There exist $\eta>0$ and for every $N\in\N$ there exist $C>0$ and $\rho>1$ such that for every $z\in\mc J_\la$, $n\in\N$
   $$|(F^n_\la)'(z)|>C\rho^{k_n(z)},$$
where $|\delta_\la|\leeq\eta$, and
   $$k_n(z)=\#\{0\leeq j\leeq n-1:F^j_\la(z)\in \textbf{B}_N(\la)\}.$$
\end{lem}

\begin{proof}
Fix $N\in\N$ and $\eta>0$ small enough. We conclude from Proposition \ref{prop:aad} (\ref{pit:aad2}) that there exist $C>0$, such that if $z\in \textbf{C}_{N+j}(\la)$, $|\delta_\la|\leeq\eta$ and $1\leeq n \leeq j-1$, then
   $$|(F^n_\la)'(z)|>C.$$
Thus, if the whole trajectory is included in $\textbf{M}_N(\la)$, the assertion holds. Moreover, the assertion also holds if $z\in \textbf{B}_N(\la)$ and the rest of the trajectory is included in $\textbf{M}_N(\la)$ (so in particular $z\in F_\la^{-1}(\textbf{M}_N(\la))$).

Next, because of the chain rule, we can assume that $z\in\textbf{B}_{N+1}(\la)$, $F_\la(z)\in\textbf{B}_{N}(\la)$ and $F^n_\la(z)\in\textbf{B}_N(\la)$. That is
   $$z, F^{n-1}_\la(z)\in\textbf{X}_N(\la):=\textbf{B}_{N+1}(\la)\sms F_\la^{-1}(\textbf{M}_N(\la)).$$

Let $P(F_\la)$ be the postcritical set of $F_\la$, i.e. closure of the strict forward orbits of the critical points of $F_\la$.
Let $||F_\la'||$ denote the norm of the derivative with respect to the hyperbolic metric on $\hat\C\sms P(F_\la)$.

Since $P(F_\la)$ is a proper subset of $Q(F_\la)=F^{-1}_\la(P(F_\la))$, \cite[Theorem 3.5]{M} leads to
 \begin{equation}\label{eq:>1}
    ||F'_\la(z)||>1,
 \end{equation}
provided $F_\la(z)\notin P(F_\la)$.

Since $\varphi_\la$ converges uniformly, the set
   $$\{(z,\delta_\la)\in\C\times [-\eta,\eta]:z\in \overline{\textbf{X}_N(\la)}\}$$
is compact. Moreover, $\overline{\textbf{X}_N(\la)}$ is separated from the postcritical set, thus there exist $\rho>1$ such that
 \begin{equation}\label{eq:greq1}
    ||F'_\la(z)||>\rho,
 \end{equation}
where $z\in\overline{\textbf{X}_N(\la)}$ and $|\delta_\la|\leeq \eta$. So, if $z, F^{n-1}_\la(z)\in\textbf{X}_N(\la)$, then we conclude from (\ref{eq:>1}) and (\ref{eq:greq1}) that
   $$||(F_\la^n)'(z)||>\rho^{k_n(z)-1}.$$
The statement follows from the fact that hyperbolic metric is comparable to Euclidian metric on the set $\textbf{B}_{N+1}(\la)$.
\end{proof}

\begin{prop}\label{prop:den}
There exists a constant $\chi>0$ such that
   $$\int_{\mc J_{\la_\dt}}\log|F'_\la(\varphi_\la)|d\tilde{\mu_\la}\rightarrow\chi,$$
where $\la\rightarrow-1$.
\end{prop}

\begin{proof}
%The functions $\varphi_\la$ converge uniformly to $\varphi_{-1}$, when $\la\rightarrow-1$.
First, note that $\log|F'_\la(\varphi_\la)|$ converges uniformly to $\log|F'_{-1}(\varphi_{-1})|$. So, Corollary \ref{col:lim} leads to
\begin{equation}\label{eq:bb}
\int_{\textbf{B}_N}\log|F'_\la(\varphi_\la)|d\tilde{\mu_\la}\rightarrow \int_{\textbf{B}_N}\log|F'_{-1}(\varphi_{-1})|d\tilde\mu_{-1}.
\end{equation}

Now we estimate the integral over $\textbf{M}_N$.
We have $F_\la'(0)=-1+\delta_\la$. Because $F_\la'$ is a Lipschitz function on filled-in Julia set, we get
$$|F_\la'(z)+1|\leeq |\delta_\la|+K_1\dist(0,z),$$
for suitable constant $K_1$.
Since $\dist(0,z)$ is close to $|\im z|$, Corollary \ref{col:odl} gives us
\begin{equation*}
 \begin{array}{ll}
    |F_\la'(z)+1|\leeq |\delta_\la|+K_2n^{-1/2}\leeq K_3n^{-1/2}\;\;\; \textrm{ if }\; n|\delta_\la|\leeq1,\\
    |F_\la'(z)+1|\leeq |\delta_\la|+K_2|\delta_\la|^{1/2}\leeq K_3|\delta_\la|^{1/2}\; \textrm{ if }\; n|\delta_\la|>1.
 \end{array}
\end{equation*}
Thus, Proposition \ref{prop:meas}, and an easy computation, lead to
 \begin{multline*}
   \int_{\textbf{M}_N}\big|\log|F'_\la(\varphi_\la)|\big|d\tilde{\mu_\la}\leeq K_4\sum_{n=N+1}^{\infty}|F'_\la(\varphi_\la)-1|\tilde{\mu_\la}(\textbf{C}_n)\\
   \leeq K_5\sum_{n=N+1}^{[1/|\delta_\la|]}n^{-1/2}n^{-\frac32\mc D(\la)+1} +K_5\sum_{n=[1/|\delta_\la|]+1}^{\infty}|\delta_\la|^{1/2}|\delta_\la|^{\frac32\mc D(\la)-1}e^{-\frac1K n|\delta_\la|}\\
   \leeq K_6 N^{\frac32-\frac32\mc D(\la)}+K_6 |\delta_\la|^{\frac32\mc D(\la)-\frac32}.
 \end{multline*}
Since $\mc D(-1)>1$ and the constant $K_6$ does not depend on $N$, we see that the integral over $\textbf{M}_N$ is less than an arbitrary $\ve>0$, for $N$ large enough, and $\delta_\la$ close to $0$.
Thus we conclude from (\ref{eq:bb}) that there exists
\begin{equation}
\chi:=\lim_{\la\rightarrow-1}\int_{\mc J_{\la_\dt}}\log|F'_\la(\varphi_\la)|d\tilde{\mu_\la}.
\end{equation}

Now we prove that $\chi>0$. Fix $\delta_\la$ close to 0, and $N\in\N$. Since the measure $\tilde{\mu_\la}$ is $F_{\la_\dt}$-invariant, Lemma \ref{lem:exp} gives us
\begin{equation*}
\int_{\mc J_{\la_\dt}}\log|(F^n_\la)'(\varphi_\la)|d\tilde{\mu_\la}\greq \log(C)\,\tilde{\mu_\la}(\mc J_{\la_\dt})+n\log(\rho)\,\tilde{\mu_\la}(\textbf{B}_{N}).
\end{equation*}
Therefore, we see that
\begin{equation*}
\int_{\mc J_{\la_\dt}}\log|F'_\la(\varphi_\la)|d\tilde{\mu_\la}=\frac1n\int_{\mc J_{\la_\dt}}\log|(F^n_\la)'(\varphi_\la)|d\tilde{\mu_\la}\greq \frac12\log(\rho)\,\tilde{\mu_\la}(\textbf{B}_{N})>0,
\end{equation*}
for $n$ large enough. Since $\rho>1$, the statement follows from the fact that $\tilde{\mu_\la}(\textbf{B}_{N})\rightarrow\tilde\mu_{-1}(\textbf{B}_{N})>0$.
%\emph{Step 2.}
%We can assume that $|(F_\la)'(z)|>1$ where $z\in\textbf{M}_N(\la)$ ($N+\tilde n$?).
%We can find $\tilde n\in\N$ such that $|(F^{\tilde n}_\la)'(z)|>C^{-1}$, where $z\in\mc C_{N+\tilde n}(\la)$.
%Thus, if $z\in\textbf{M}_{N+\tilde n}(\la)$, we get
%$$|(F^n_\la)'(z)|>\rho^{k_n(z)}.$$
%For $z\in\textbf{B}_{N+\tilde n}(\la)$ we can use Lemma \ref{lem:exp}, so
%$$\int_{\mc J_{\la_\dt}}\log|(F^n_\la)'(\varphi_\la)|d\tilde{\mu_\la}\greq \log(C)\tilde{\mu_\la}(\textbf{B}_{N+\tilde n}(\la))+n\log(\rho)\tilde{\mu_\la}(\textbf{B}_{N}(\la)).$$
%Therefore, we see that
%$$\frac1n\int_{\mc J_{\la_\dt}}\log|(F^n_\la)'(\varphi_\la)|d\tilde{\mu_\la}\greq \frac12\log(\rho)\tilde{\mu_\la}(\textbf{B}_{N}(\la))>0,$$
%for $n$ large enough.
\end{proof}

\section{Estimates close to the parabolic point $0$}\label{sec:generaltwopetals}

In this section, we will estimate a "principal part" of the function
\begin{equation}\label{eq:ff}
   \frac{\partial}{\partial\la}\log\big|F'_\la(\varphi_\la)\big|= \re\Big(\frac{\frac{\partial}{\partial \la} (F_\la'(\varphi_\la))}{F_\la'(\varphi_\la)}\Big),
\end{equation}
close to the parabolic point $0$ (cf. formula \ref{eq:wzor}). Note that
\begin{equation}\label{eq:fff}
   \frac{\partial}{\partial \la} (F_\la'(\varphi_\la))= \Big(\frac{\partial}{\partial \la} F_\la'\Big)(\varphi_\la)+\frac{\partial}{\partial\la}\varphi_\la\cdot F_\la''(\varphi_\la).
\end{equation}

\subsection{}
We define
   $$\dot\varphi_\la(s):=\frac{\partial}{\partial\la}\varphi_\la(s).$$
Since the Julia set moves holomorphically on $\Lambda_l$ and $\Lambda_r$ (see Section \ref{ssc:hm}), the derivative is well defined if $\la\in(-1-\ve,-1)\cup(-1,-1+\ve)$, for suitably chosen $\ve>0$. Moreover
   $$\varphi_\la(F_{\la_\dt}(s))=F_\la(\varphi_\la(s)).$$
Thus, differentiating  both sides with respect to $\la$, we get
   $$\dot\varphi_\la(F_{\la_\dt}(s))=
   \Big(\frac{\partial}{\partial\la}F_\la\Big)(\varphi_\la(s))+
   F_\la'(\varphi_\la(s))\cdot\dot\varphi_\la(s).$$
Hence,
   $$\dot\varphi_\la(s)=
   -\frac{(\frac{\partial}{\partial\la}F_\la)(\varphi_\la(s))}{F_\la'(\varphi_\la(s))}+
   \frac{\dot\varphi_\la(F_{\la_\dt}(s))}{F_\la'(\varphi_\la(s))}.$$
Next, replacing $s$ by $F_{\la_\dt}(s), F^2_{\la_\dt}(s),\ldots,F^{m-1}_{\la_\dt}(s)$, and using the fact that $\varphi_\la(F_{\la_\dt}^{j-1}(s))=F_\la^{j-1}(\varphi_\la(s))$, we obtain
\begin{equation}\label{eq:sumas}
   \dot\varphi_\la(s)=
   -\sum_{j=1}^{m}\frac{(\frac{\partial}{\partial\la}F_\la)
   (F_\la^{j-1}(\varphi_\la(s)))}{(F_\la^j)'(\varphi_\la(s))}+
   \frac{\dot\varphi_\la(F_{\la_\dt}^{m}(s))}{(F_\la^m)'(\varphi_\la(s))}.
\end{equation}
Since $\frac{\partial}{\partial\la}F_\la(z)=z+O(z^3)$, we rewrite the above formula as follows:
\begin{multline}\label{eq:101}
   \dot\varphi_\la(s)=
   -\sum_{j=1}^{m}\frac{F_\la^{j-1}(\varphi_\la(s))}{(F_\la^j)'(\varphi_\la(s))} \\ -\sum_{j=1}^{m}\frac{(\frac{\partial}{\partial\la}F_\la)(F_\la^{j-1}(\varphi_\la(s)))- F_\la^{j-1}(\varphi_\la(s))} {(F_\la^j)'(\varphi_\la(s))} +\frac{\dot\varphi_\la(F_{\la_\dt}^{m}(s))}{(F_\la^m)'(\varphi_\la(s))}.
\end{multline}

\subsection{}
Let $s\in \textbf{C}_n$, where $n\greq1$. The function $\dot\Psi_\la$ is defined on the set $\textbf{M}_0^*$ as follows:
\begin{equation}\label{eq:Psi}
   \dot\Psi_\la(s):=-\sum_{j=1}^{n}\frac{(\frac{\partial} {\partial\la}F_\la)(F_\la^{j-1}(\varphi_\la(s)))} {(F_\la^j)'(\varphi_\la(s))}.
\end{equation}
Next, if $\varphi_\la(s)=z\in\mc C_n(\la)$, where $n\greq1$. We define the function $\dot\psi_\la(z)$ on $\mathcal{M}_0^*(\la)$ (i.e. close to the parabolic point $0$), by setting
\begin{equation}\label{eq:psi}
   \dot\psi_\la(z):=-\sum_{j=1}^{n}\frac{F_\la^{j-1}(z)}{(F_\la^j)'(z)}.%= -\sum_{j=1}^{n}\frac{\varphi_\la(F_{\la_\dt}^{j-1}(s))}{(F_\la^j)'(\varphi_\la(s))}.
\end{equation}
We are going to study $\dot\psi_\la(z)$, which is a "principal part" of $\dot\varphi_\la(s)$. %(cf. formula (\ref{eq:101}) for $m=n$).
But first, for $z\in\mathcal{M}_0^*(\la)$, let us write
\begin{equation}\label{eq:beta}
   \beta_\la(z):=\im z\cdot\im\dot\psi_\la(z).
\end{equation}
So we see that $\beta_\la(z)$ is related to $\im \dot\psi_\la(z)$. We will prove that  $\beta_\la(z)$ is also related to $\re\dot\psi_\la(z)$, as well as to the function from (\ref{eq:ff}) (see subsections \ref{ssc:re}, \ref{ssc:ff}). The function $\beta_\la(z)$ will be estimated at the end of this section (see subsection \ref{ssc:beta})

Next, in Section \ref{sec:pp}, we will prove that the function $\dot\Psi_\la$ is "close" to $\dot\psi_\la$, and we will consider the other points of the parabolic cycle.

\subsection{}\label{ssc:re}
Note that, for $n_0<n$, we have
   $$\dot\psi_\la(z)=
   -\sum_{j=1}^{n_0}\frac{F_\la^{j-1}(z)}{(F_\la^j)'(z)}+
   \frac{\dot\psi_\la(F_\la^{n_0}(z))}{(F_\la^{n_0})'(z)}.$$
In particular, if $n_0=1$, we get
\begin{equation}\label{eq:onestep}
   \dot\psi_\la(z)\cdot F_\la'(z)=-z+\dot\psi_\la(F_\la(z)).
\end{equation}

It follows from definition of $\tilde o$ and Corollary \ref{col:odl}, that
$$A_\la(z)=\tilde o(B_\la(z)),$$
if for every $\ve>0$ there exist $\eta>0$ and $N\in\N$ such that $|A_\la(z)/B_\la(z)|\leeq \ve$, where $z\in\mc C_n(\la)$, $n>N$ and $0<|\la|<\eta$.

\begin{lem}\label{lem:const}
   We have:
 \begin{enumerate}
    \item\label{lit:const1}
       if $z_1,z_2\in\mc C^+_n(\la)$ (or $z_1,z_2\in\mc C^-_n(\la)$), then
          $$\dot\psi_\la(z_1)-\dot\psi_\la(z_2)=\tilde o(1),$$
    \item\label{lit:const2}
       if $z\in\mathcal M_N^*(\la)$, then
          $$\dot\psi_\la(z)-\overline{\dot\psi_\la(F_\la(z))}=\tilde o(1).$$
 \end{enumerate}
\end{lem}

\begin{proof}
The first statement can be proven in a similar way as in \cite[Lemma 9.1]{Ji}, but we do not have to group terms by pairs (see also the proof of Lemma \ref{lem:const2}).
The second statement follows from the first.
\end{proof}

It follows from Lemma \ref{lem:kat}, that $\re z=\tilde o(\im z)$. So, (\ref{eq:F0}) gives us
\begin{equation}\label{eq:reimF'}
 \begin{array}{ll}
    \re F_\la'(z)=-1+\tilde o(1),\\
    \im F_\la'(z)=2a\im z+\tilde o(\re z).
 \end{array}
\end{equation}

\begin{lem}\label{lem:realpart}
   If $z\in\mathcal{M}_0^*(\la)$, then
%For every $\varepsilon>0$ there exist $N\in\N$, $\la_\dt>0$ such that if $z\in\mathcal M_N(\la)$ and $|\delta_\la|<\la_\dt$, then
\begin{equation*}
   \re\dot\psi_\la(z)=-a\beta_\la(z)+\tilde o(1)+\tilde o(\beta_\la(z)).
\end{equation*}
%In particular $...<\re\dot\psi_\la(z)<...$.
\end{lem}

\begin{proof}
It follows from (\ref{eq:onestep}) that
   $$\re\dot\psi_\la(z)\cdot\re F_\la'(z)-\im\dot\psi_\la(z)\cdot\im F_\la'(z)=-\re z+\re\dot\psi_\la(F_\la(z)),$$
and then
   $$\re\dot\psi_\la(F_\la(z))-\re\dot\psi_\la(z)\cdot\re F_\la'(z)=\re z-\im\dot\psi_\la(z)\cdot\im F_\la'(z).$$
Using Lemma \ref{lem:const} (\ref{lit:const2}), and (\ref{eq:reimF'}) we can estimate both sides of the above expression:
\begin{equation*}
 \begin{array}{ll}
    \re\dot\psi_\la(F_\la(z))-\re\dot\psi_\la(z)\cdot\re F_\la'(z)=2\re\dot\psi_\la(z)+\tilde o(\re\dot\psi_\la(z))+\tilde o(1),\\
    \re z-\im\dot\psi_\la(z)\cdot\im F_\la'(z)=-2a\beta_\la(z)+\tilde o(\beta_\la(z))+\tilde o(1).
 \end{array}
\end{equation*}
Hence
   $$2\re\dot\psi_\la(z)=-2a\beta_\la(z)+\tilde o(1)+\tilde o(\beta_\la(z))+\tilde o(\re\dot\psi_\la(z)).$$
We see that $\re\dot\psi_\la(z)$ cannot be to large with respect to $\beta_\la(z)$, so the lemma follows.
\end{proof}

\subsection{}\label{ssc:ff}
Now, we will estimate a "principal part" of (\ref{eq:ff}).

First, if $u_\la$ is a complex function defined on a subset of $\mc J_{\la_\dt}$, then
\begin{equation}\label{eq:D2}
   \mathfrak{D}^2_{F_\la}(\varphi_\la,u_\la):=\Big(\frac{\partial}{\partial \la} F_\la'\Big)(\varphi_\la)+u_\la F_\la''(\varphi_\la).
\end{equation}
Thus, using (\ref{eq:fff}) and (\ref{eq:ff}), we get
\begin{equation}\label{eq:D2cor}
   \frac{\partial}{\partial \la} \big(F_\la'(\varphi_\la)\big)= \mathfrak{D}^2_{F_\la}(\varphi_\la,\dot\varphi_\la) \textrm{, and }
   \frac{\partial}{\partial\la}\log\big|F'_\la(\varphi_\la)\big|= \re\Big(\frac{\mathfrak{D}^2_{F_\la}(\varphi_\la,\dot\varphi_\la)} {F_\la'(\varphi_\la)}\Big).
\end{equation}
Next, using (\ref{eq:F0}) and (\ref{eq:F1}), we see that
\begin{equation*}
   \mathfrak{D}^2_{F_\la}(\varphi_\la,\dot\varphi_\la)= 1+\dot\varphi_\la(2a+6b\varphi_\la)+O(\delta_\la\varphi_\la\dot\varphi_\la)+ O(\varphi_\la^2\dot\varphi_\la)+O(\varphi_\la^2).
\end{equation*}

Now, instead of $\dot\varphi_\la$, we will deal with the function $\dot\psi_\la$. If $\varphi_\la(s)=z$, then $\dot\psi_\la(\varphi_\la(s))=\dot\psi_\la(z)$, therefore
%Now, instead of $\dot\varphi_\la(s)$, we will deal with %$\dot\psi_\la(\varphi_\la(s))=\dot\psi_\la(z)$. So, we will consider
%\begin{multline*}
   %\mathfrak{D}^2_{F_\la}(\varphi_\la,\dot\psi_\la(\varphi_\la))=
   %1+\dot\psi_\la(\varphi_\la)(2a+6b\varphi_\la)\\ %+O(\delta_\la\varphi_\la\dot\psi_\la(\varphi_\la))+ %O(\varphi_\la^2\dot\psi_\la(\varphi_\la))+O(\varphi_\la^2).
%\end{multline*}
%Thus, if $\varphi_\la(s)=z$, we obtain
\begin{multline}\label{eq:D}
   \mathfrak{D}^2_{F_\la}(\varphi_\la,\dot\psi_\la(\varphi_\la))= \mathfrak{D}^2_{F_\la}(z,\dot\psi_\la(z))\\
   =1+\dot\psi_\la(z)(2a+6bz)
   +O(\delta_\la z\dot\psi_\la(z))+O(z^2\dot\psi_\la(z))+O(z^2).
\end{multline}

The two following Propositions are the most important steps in the proof of the Theorem \ref{thm:twopetals} in the case $d(c_0)<4/3$.

\begin{prop}\label{prop:re=}
If $\varphi_\la(s)=z\in\mathcal{M}_0^*(\la)$, then
%   For every $\varepsilon>0$ there exist $c_0<-3/4$, $N\in\N$ such that if $z\in\mathcal M_N(c)$ and $c\in(c_0,-3/4)$, then
 \begin{equation*}
    \re\Big(\frac{\mathfrak{D}^2_{F_\la}(\varphi_\la,\dot\psi_\la(\varphi_\la))} {F_\la'(\varphi_\la)}\Big) = 6A^2\beta_\la(\varphi_\la)-1+\tilde o(1)+\tilde o(\beta_\la(\varphi_\la)).
 \end{equation*}
\end{prop}

\begin{proof}
The fact that $\re z=\tilde o(\im z)$, and Lemma \ref{lem:realpart} gives us
\begin{equation*}
 \begin{array}{lll}
    \re(z\cdot\dot\psi_\la(z))&=-\im z\im\dot\psi_\la(z)+\re z\re\dot\psi_\la(z)\\&=-\beta_\la(z)+\tilde o(\im z)+\tilde o(\beta_\la(z)),\\
    \im(z\cdot\dot\psi_\la(z))&=\re z\im\dot\psi_\la(z)+\im z\re\dot\psi_\la(z)=\tilde o(\im z)+\tilde o(\beta_\la(z)),
 \end{array}
\end{equation*}
and
\begin{equation*}
   z^2\cdot\dot\psi_\la(z)=\tilde o(\im^2 z)+\tilde o(\beta_\la(z)).
\end{equation*}
So, (\ref{eq:D}) leads to
\begin{equation*}\label{eq:reimD}
 \begin{array}{ll}
    \re \mathfrak{D}^2_{F_\la}(z,\dot\psi_\la(z))=1+2a\re\dot\psi_\la(z)-6b\beta_\la(z)+\tilde o(\im z)+\tilde o(\beta_\la(z)),\\
    \im \mathfrak{D}^2_{F_\la}(z,\dot\psi_\la(z))=2a\im\dot\psi_\la(z)+\tilde o(\im z)+\tilde o(\beta_\la(z)).
 \end{array}
\end{equation*}
Next, using (\ref{eq:reimF'}) and Lemma \ref{lem:realpart}, we obtain
\begin{equation*}
 \begin{array}{ll}
    \re \mathfrak{D}^2_{F_\la}(z,\dot\psi_\la(z))\cdot\re F_\la'(\varphi_\la)=&
    -\big(1+2a(-a\beta_\la(z))-6b\beta_\la(z)\big)\\&+\tilde o(1)+\tilde o(\beta_\la(z))\\
    \im \mathfrak{D}^2_{F_\la}(z,\dot\psi_\la(z))\cdot\im F_\la'(\varphi_\la)=&
    4a^2\beta_\la(z)+\tilde o(\im^2 z)+\tilde o(\beta_\la(z)).
 \end{array}
\end{equation*}
Finally, we have
%If $\varphi_\la(s)=z$, then using the fact that $F_\la'(\varphi_\la)=-1+\tilde o(1)$ we get
\begin{multline*}
   \re\Big(\frac{\mathfrak{D}^2_{F_\la}(\varphi_\la,\dot\psi_\la(\varphi_\la))} {F_\la'(\varphi_\la)}\Big)\\= \frac{\re \mathfrak{D}^2_{F_\la}(z,\dot\psi_\la(z))\cdot\re F_\la'(\varphi_\la)+\im \mathfrak{D}^2_{F_\la}(z,\dot\psi_\la(z))\cdot\im F_\la'(\varphi_\la)}{|F'_\la(\varphi_\la)|^2}\\ =-1+2a^2\beta_\la(z)+6b\beta_\la(z)+4a^2\beta_\la(z)+\tilde o(1)+\tilde o(\beta_\la(z)).
\end{multline*}
Since $A^2=a^2+b$, the assertion follows.
\end{proof}

\subsection{}\label{ssc:beta}
Now we will estimate the function $\beta_\la$.

Let us consider the function $\Gamma:\R\rightarrow\R$,
\begin{equation}\label{eq:f}
   \Gamma(x)=\frac14\Big(1+\frac{\sinh 2x-2x}{\cosh 2x-1}\Big)=\frac{e^{2x}-1-2x}{2(e^{2x}-1)^2}\:e^{2x},
\end{equation}
where $x\neq0$, and $\Gamma(0)=1/4$. We see that $\Gamma$ is continuous, whereas $\lim_{x\rightarrow\infty}\Gamma(x)=1/2$ and $\lim_{x\rightarrow-\infty}\Gamma(x)=0$.
Next we get
   $$\Gamma'(x)=\frac{x\cosh x-\sinh x}{2\sinh^3x},$$
where $x\neq0$, and $\Gamma'(0)=1/6$. So, the function $\Gamma$ is monotone increasing, and we conclude that $\Gamma(x)\in (0,1/2)$ for all $x\in\R$.
\begin{figure}
\begin{center}
\includegraphics[width=8cm]{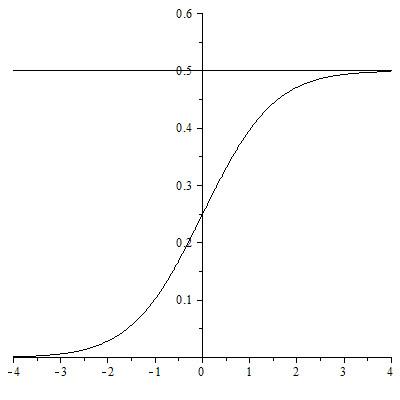}
\caption{The graph of the function $\Gamma$}
\end{center}
\end{figure}

\begin{prop}\label{prop:beta}
   For every $\varepsilon>0$ there exist $\eta>0$, $N\in\N$ such that if $z\in\mc C_n(\la)$, $n>N$ and $0<|\delta_\la|<\eta$ then
 \begin{equation*}
    \frac{1}{A^2}\:\Gamma(n\delta_\la)e^{-\ve n|\delta_\la|}-\varepsilon\leeq\beta_\la(z)\leeq \frac{1}{A^2}\:\Gamma(n\delta_\la)e^{\ve n|\delta_\la|}+\varepsilon.
 \end{equation*}
%Moreover, $\beta_c(z)\in...$.
\end{prop}

\begin{proof}
Fix $\ve>0$.
Let $\tilde n$ be an integer large enough, and let $z\in\mc C_n(\la)$, where $n>\tilde n$. Moreover, we can assume that $\im z>0$.

{\em Step 1.}
Lemma \ref{lem:kat} and Lemma \ref{lem:arg} gives us
\begin{equation}\label{eq:a1}
   (1-\ve)\frac{|\im F_\la^{j-1}(z)|}{|(F_\la^j)'(z)|}\leeq \im\bigg(-\frac{F_\la^{j-1}(z)}{(F_\la^j)'(z)}\bigg) \leeq(1+\ve)\frac{|\im F_\la^{j-1}(z)|}{|(F_\la^j)'(z)|},
\end{equation}
where $n>\tilde n+j$, and $\delta_\la$ is close to $0$. Next, Proposition \ref{prop:aad} (\ref{pit:aad2}) leads to
\begin{multline*}
   (1-2\ve)|\mc C_n(\la)|\frac{|\im F_\la^{j-1}(z)|}{|\mc C_{n-j}(\la)|}\leeq\im\bigg(-\frac{F_\la^{j-1}(z)}{(F_\la^j)'(z)}\bigg)\\ \leeq(1+2\ve)|\mc C_n(\la)|\frac{|\im F_\la^{j-1}(z)|}{|\mc C_{n-j}(\la)|}.
\end{multline*}
Since $|\im F_\la^{j}(z)|$ is close to $|\im F_\la^{j-1}(z)|$, and $F_\la^{j}(z)\in\mc C_{n-j}(\la)$, Proposition \ref{prop:rn} gives us
\begin{multline}\label{eq:a2}
   (1-3\varepsilon)|\mc C_n(\la)|\frac{e^{2(n-j)\delta_\la}}{A^2|\im F_\la^{j}(z)|^2}\:e^{-\ve (n-j)|\delta_\la|}
   \leeq\im\bigg(-\frac{F_\la^{j-1}(z)}{(F_\la^j)'(z)}\bigg)\\
   \leeq(1+3\varepsilon)|\mc C_n(\la)|\frac{e^{2(n-j)\delta_\la}}{A^2|\im F_\la^{j}(z)|^2}\:e^{\ve (n-j)|\delta_\la|}.
\end{multline}
Using Proposition \ref{prop:rnn} (\ref{pit:rnn1}) we get
\begin{multline}\label{eq:s0}
   (1-4\varepsilon)|\mc C_n(\la)|\frac{e^{2(n-j)\delta_\la}-1}{\delta_\la}\:e^{-2\ve(n-j)|\delta_\la|}
   \leeq\im\bigg(-\frac{F_\la^{j-1}(z)}{(F_\la^j)'(z)}\bigg)\\
   \leeq(1+4\varepsilon)|\mc C_n(\la)|\frac{e^{2(n-j)\delta_\la}-1}{\delta_\la}\:e^{2\ve(n-j)|\delta_\la|},
\end{multline}
where $n-j>\tilde n$.

{\em Step 2.}
Now we will estimate $\sum_{j=1}^{n-\tilde n} (e^{2(n-j)\delta_\la}-1)$. We have
\begin{equation*}
   \int_{0}^{n-\tilde n}(e^{\gamma(n-x)\delta_\la}-1)\:dx =\frac{1}{\gamma\delta_\la}(e^{\gamma n\delta_\la}-e^{\gamma\tilde n\delta_\la})-(n-\tilde n),
\end{equation*}
Since $(e^{2m\delta_\la}-1)/(e^{2(m+1)\delta_\la}-1)$ is close to $1$, where $m\greq\tilde n$ (for sufficiently large $\tilde n$, and small $\delta_\la$), the sum can be estimated by the integral as follows:
\begin{multline}\label{eq:s1}
   (1-\ve)\Big(\frac{1}{2\delta_\la}(e^{2n\delta_\la}-e^{2\tilde n\delta_\la})-(n-\tilde n)\Big)\leeq\sum_{j=1}^{n-\tilde n}\big(e^{2(n-j)\delta_\la}-1\big)\\
   \leeq(1+\ve)\Big(\frac{1}{2\delta_\la}(e^{2n\delta_\la}-e^{2\tilde n\delta_\la})-(n-\tilde n)\Big).
\end{multline}

Let
   $$g(x):=\frac{e^x-1}{x}.$$
If we put $g(0)=1$, then $g$ is a positive increasing function on $\R$, and
\begin{equation}\label{eq:s2}
   \frac{1}{2\delta_\la}(e^{2n\delta_\la}-e^{2\tilde n\delta_\la})-(n-\tilde n)=n\big(g(2n\delta_\la)-g(2\tilde n\delta_\la)\big).
\end{equation}
For small $\delta_\la$ and $n>N$, where $N$ is large enough, $g(2\tilde n\delta_\la)-1$ is small with respect to $g(2 n\delta_\la)-1$, thus we get
\begin{equation}\label{eq:s3}
   (1-\ve)\big(g(2n\delta_\la)-1\big)\leeq g(2n\delta_\la)-g(2\tilde n\delta_\la)\\
   \leeq(1+\ve)\big(g(2n\delta_\la)-1\big).
\end{equation}
Combining (\ref{eq:s1}) with (\ref{eq:s2}) and (\ref{eq:s3}), we obtain
\begin{equation}\label{eq:s4}
   (1-\ve)^2n\big(g(2n\delta_\la)-1\big)\leeq\sum_{j=1}^{n-\tilde n}\big(e^{2(n-j)\delta_\la}-1\big)\leeq(1+\ve)^2n\big(g(2n\delta_\la)-1\big).
\end{equation}

{\em Step 3.}
For every $1\leeq j\leeq n-\tilde n$, we have $e^{-2\ve n|\delta_\la|}< e^{\pm2\ve(n-j)|\delta_\la|}< e^{2\ve n|\delta_\la|}$. Thus, (\ref{eq:s0}) and (\ref{eq:s4}) leads to
\begin{multline*}
   (1-6\ve)|\mc C_n(\la)|\frac{n\big(g(2n\delta_\la)-1\big)}{\delta_\la}\:e^{-2\ve n|\delta_\la|} \leeq \im\bigg(-\sum_{j=1}^{n-\tilde n}\frac{F_\la^{j-1}(z)}{(F_\la^j)'(z)}\bigg)\\
   \leeq(1+7\ve)|\mc C_n(\la)|\frac{n\big(g(2n\delta_\la)-1\big)}{\delta_\la}\:e^{2\ve n|\delta_\la|}.
\end{multline*}
By the definition of the function $g$,
\begin{multline*}
   (1-6\ve)|\mc C_n(\la)|\Big(\frac{e^{2 n\delta_\la}-1- 2n\delta_\la}{2\delta^2_\la}\Big)\:e^{-2\ve n|\delta_\la|}
   \leeq\im\bigg(-\sum_{j=1}^{n-\tilde n}\frac{F_\la^{j-1}(z)}{(F_\la^j)'(z)}\bigg)\\
   \leeq(1+7\ve)|\mc C_n(\la)|\Big(\frac{e^{2 n\delta_\la}-1-2n\delta_\la}{2\delta^2_\la}\Big)\:e^{2\ve n|\delta_\la|}.
\end{multline*}
Multiplying the above inequalities by $\im z$, and using Proposition \ref{prop:rn}, we obtain
\begin{multline*}
   (1-7\ve)A^2\im^4z\Big(\frac{e^{2 n\delta_\la}-1- 2n\delta_\la}{2\delta^2_\la}\Big)\:e^{-2n\delta_\la-3\ve n|\delta_\la|}\\
   \leeq\im z\cdot\im\bigg(-\sum_{j=1}^{n-\tilde n}\frac{F_\la^{j-1}(z)}{(F_\la^j)'(z)}\bigg)\\
   \leeq(1+8\ve)A^2\im^4z\Big(\frac{e^{2 n\delta_\la}-1-2n\delta_\la}{2\delta^2_\la}\Big)\:e^{-2n\delta_\la+3\ve n|\delta_\la|}.
\end{multline*}
Next, Proposition \ref{prop:rnn} (\ref{pit:rnn1}) leads to
\begin{multline}\label{eq:nnn}
   (1-8\ve)\frac{1}{A^2}\Big(\frac{e^{2 n\delta_\la}-1-2n\delta_\la}{2(e^{2n\delta_\la}-1)^2}\Big)\:e^{2n\delta_\la-4\ve n|\delta_\la|}\\
   \leeq\im z\cdot\im\bigg(-\sum_{j=1}^{n-\tilde n}\frac{F_\la^{j-1}(z)}{(F_\la^j)'(z)}\bigg)\\
   \leeq(1+9\ve)\frac{1}{A^2}\Big(\frac{e^{2 n\delta_\la}-1-2n\delta_\la}{2(e^{2n\delta_\la}-1)^2}\Big)\:e^{2n\delta_\la+4\ve n|\delta_\la|},
\end{multline}
where $n>N$ and $\delta_\la$ is close to $0$.

{\em Step 4.}
We have
\begin{equation*}
   \bigg|\sum_{j=n-\tilde n+1}^{n}\frac{F_\la^{j-1}(z)}{(F_\la^j)'(z)}\bigg|\leeq K(\tilde n).
\end{equation*}
Therefore, we can assume that
\begin{equation*}
   \bigg|\im z\cdot\im\bigg(-\sum_{j=n-\tilde n+1}^{n}\frac{F_\la^{j-1}(z)}{(F_\la^j)'(z)}\bigg)\bigg|\leeq\ve,
\end{equation*}
where $z\in\mc C_n(\la)$ and $n>N$.
So, combining the above with (\ref{eq:nnn}), and (\ref{eq:psi}), %\ref{eq:nn}
we obtain
\begin{multline*}
   (1-8\ve)\frac{1}{A^2}\Big(\frac{e^{2 n\delta_\la}-1-2n\delta_\la}{2(e^{2n\delta_\la}-1)^2}\Big)\:e^{2n\delta_\la-4\ve n|\delta_\la|}-\ve
   \leeq\im z\cdot\im\dot\psi_\la(z)
   \\
   \leeq(1+9\ve)\frac{1}{A^2}\Big(\frac{e^{2 n\delta_\la}-1-2n\delta_\la}{2(e^{2n\delta_\la}-1)^2}\Big)\:e^{2n\delta_\la+4\ve n|\delta_\la|}+\ve.
\end{multline*}
Thus, the statement follows from definitions (\ref{eq:f}), (\ref{eq:beta}) and the fact that the function $\Gamma$ is bounded.
\end{proof}

We end this section with an easy corollary:

\begin{col}\label{col:modpsi}
   For every $\ve>0$ there exists $K>0$ and $\eta>0$ such that
 \begin{equation*}
    \big|\dot\psi_\la(z)\big|\leeq K\frac{e^{\ve n|\delta_\la|}}{|\im z|}
 \end{equation*}
   where $z\in\mc C_n(\la)$, $n\greq1$ and $0<|\delta_\la|<\eta$.
\end{col}

\begin{proof}
Fix $\ve>0$. We conclude from definition (\ref{eq:beta}), combined with Lemma \ref{lem:realpart}, that there exists $K_1>1$, such that
   $$\big|\dot\psi_\la(z)\big|\leeq \frac{\beta_\la(z)}{|\im z|}+K_1\beta_\la(z)+\ve.$$
Next, since the function $|\Gamma|$ is bounded by $1/2$, Proposition \ref{prop:beta} leads to
\begin{equation*}
   \big|\dot\psi_\la(z)\big|\leeq K_{2}\frac{e^{\ve n|\delta_\la|}}{|\im z|},
\end{equation*}
where $z\in\mc C_n(\la)$, $n>N$, $0<|\delta_\la|<\eta$, for suitably chosen $N\in\N$ and $\eta>0$.
Moreover, increasing the constant $K_{2}$ if necessary, we can assume that the above inequality holds for $z\in \mc C_n(\la)$ where $n\greq1$.
\end{proof}

Note that the constant in the above corollary can be chosen independently of $\ve>0$.

\section{Controlling the other points of the parabolic cycle}\label{sec:pp}

The goal of this section is to establish Proposition \ref{prop:reDF}.

\begin{lem}\label{lem:xi}
   There exists $K>0$, and for every $\ve>0$ there exists $\eta>0$, such that
 \begin{equation*}
    |\dot\Psi_\la(s)-\dot\psi_\la(z)|\leeq K|\im z|\: e^{\ve n|\delta_\la|},
 \end{equation*}
   where $\varphi_\la(s)=z\in\mc C_n(\la)$, $n\greq1$ and $0<|\delta_\la|<\eta$.
\end{lem}

\begin{proof}
Fix $\tilde\ve>0$, and let $z\in\mc C_n(\la)$. First, note that
\begin{equation}\label{eq:F3}
   \Big|\big(\frac{\partial}{\partial\la}F_\la\big)(F_\la^{j-1}(z))- F_\la^{j-1}(z)\Big|= O\big(|F_\la^{j-1}(z)|^3\big),
\end{equation}
(cf. (\ref{eq:F1}) and definitions (\ref{eq:Psi}), (\ref{eq:psi})).

{\em Step 1.}
We can assume that $|F_\la^{j-1}(z)|$ is close to $|\im F_\la^{j-1}(z)|$ (see Lemma \ref{lem:kat}).
Thus, similarly as in the proof of Proposition \ref{prop:beta} ({\em Step 1.}), using Proposition \ref{prop:aad} (\ref{pit:aad2}) and Proposition \ref{prop:rn} (cf. right hand sides of (\ref{eq:a1}) and (\ref{eq:a2})), we can get
\begin{equation*}\label{eq:F^3}
   \bigg|\frac{(F_\la^{j-1}(z))^3}{(F_\la^j)'(z)}\bigg|
   \leeq(1+\tilde\ve)|\mc C_n(\la)|\frac{e^{2(n-j)\delta_\la}}{A^2}\:e^{\tilde\ve (n-j)|\delta_\la|}.
\end{equation*}
where $n-j>N$, $|\delta_\la|<\eta$, for suitably chosen $N\in\N$ and $\eta>0$.

Observe that we can find a constant $\tilde K>0$ such that the above inequality, with $(1+\tilde\ve)$ replaced by $\tilde K$, holds also for $0\leeq n-j\leeq N$.
Let $K_1=2\tilde K$. Then for every $\ve>0$ there exist $\eta>0$ such that
\begin{equation}\label{eq:C3}
   \bigg|\frac{(F_\la^{j-1}(z))^3}{(F_\la^j)'(z)}\bigg|
   \leeq K_1|\mc C_n(\la)|\frac{e^{2(n-j)\delta_\la}}{A^2}\:e^{\ve (n-j)|\delta_\la|},
\end{equation}
where $n-j\greq0$ and $|\delta_\la|<\eta$.

{\em Step 2.}
Using Proposition \ref{prop:rn} and Proposition \ref{prop:rnn} (\ref{pit:rnn1}), we obtain
\begin{equation*}
   |\mc C_n(\la)|\leeq \frac{1+\tilde\ve}{A}\bigg|\frac{\delta_\la e^{2n\delta_\la}}{e^{2n\delta_\la}-1}\bigg|^{3/2} e^{-2n\delta_\la}\:e^{\tilde\ve n|\delta_\la|},
\end{equation*}
where $n>N$ and $|\delta_\la|<\eta$. As before, we can find a constant $K_2>0$ such that for every $\ve>0$ there exists $\eta>0$ such that
\begin{equation*}
   |\mc C_n(\la)|\leeq \frac{K_2}{A}\bigg|\frac{\delta_\la e^{2n\delta_\la}}{e^{2n\delta_\la}-1}\bigg|^{3/2} e^{-2n\delta_\la}\:e^{\ve n|\delta_\la|},
\end{equation*}
where $n\greq1$, $|\delta_\la|<\eta$.
So, combining this with (\ref{eq:C3}), we get
\begin{equation*}
   \bigg|\frac{(F_\la^{j-1}(z))^3}{(F_\la^j)'(z)}\bigg|
   \leeq K_3\bigg|\frac{\delta_\la e^{2n\delta_\la}}{e^{2n\delta_\la}-1}\bigg|^{3/2} e^{-2j\delta_\la}\:e^{\ve (2n-j)|\delta_\la|},
\end{equation*}
%for $n\greq1$ and $|\delta_\la|<\eta$ (where $\eta$ is sufficiently close to $0$).

{\em Step 3.}
We have $e^{\ve(2n-j)|\delta_\la|}<e^{2\ve n|\delta_\la|}$, therefore summing from $j=1$ to $n$, we get
\begin{equation*}
   \sum_{j=1}^{n}\bigg|\frac{(F_\la^{j-1}(z))^3}{(F_\la^j)'(z)}\bigg|
   \leeq K_4\bigg|\frac{\delta_\la e^{2n\delta_\la}}{e^{2n\delta_\la}-1}\bigg|^{3/2} \frac{1-e^{-2n\delta_\la}}{2\delta_\la}\:e^{2\ve n|\delta_\la|}.
\end{equation*}
Note that $e^{2n\delta_\la}(e^{2n\delta_\la}-1)^{-1}(1-e^{-2n\delta_\la})=1$. Next, we can conclude from Proposition \ref{prop:rnn} (\ref{pit:rnn1}) that there exists a constant $K_5>0$ such that
\begin{equation*}
   \sum_{j=1}^{n}\bigg|\frac{(F_\la^{j-1}(z))^3}{(F_\la^j)'(z)}\bigg|
   \leeq \frac{K_4}{2}\bigg|\frac{\delta_\la e^{2n\delta_\la}}{e^{2n\delta_\la}-1}\bigg|^{1/2} e^{2\ve n|\delta_\la|}\leeq K_5|\im z|\:e^{3\ve n|\delta_\la|},
\end{equation*}
where $z\in\mc C_n(\la)$, $n\greq1$ and $|\delta_\la|<\eta$. Thus the assertion follows from (\ref{eq:F3}) and definitions (\ref{eq:Psi}), (\ref{eq:psi}).
\end{proof}

\begin{lem}\label{lem:phiPsi}
   There exist $K>0$, $\eta>0$ and a sequence of points $s_n^j\in\mc C_n^j$, such that
 $$\big|\dot\varphi_\la(s_n^j)-\dot\Psi_\la(s_n^j)\big|\leeq K|\mc C_n(\la)|,$$
   where $0<|\delta_\la|<\eta$. Moreover, we can assume that $\tilde f_{\la_\dt}^j(s_n^0)=s_n^j$ and $F_{\la_\dt}(s_{j+1}^j)=s_n^j$ where $1\leeq j\leeq k-1$ and $n\greq1$.
\end{lem}

\begin{proof}
Let us take a sequence $s_n^j$ consisting of preimages of $\beta$-fixed point of $\tilde f_{\la_\dt}$, satisfying $\tilde f_{\la_\dt}^j(s_n^0)=s_n^j$ and $F_{\la_\dt}(s_{j+1}^j)=s_n^j$ where $s_n^j\in\mc C_n^j$. Then $$\dot\varphi_\la(F_{\la_\dt}^{n}(s_n^j))$$
is uniformly bounded for $0<|\delta_\la|<\eta$ (where $\eta>0$ is small enough).

Next, using the bounded distortion property (cf. Proposition \ref{prop:aad} (\ref{pit:aad2})), we conclude that there exist a constant $K>0$, such that
$$|(F^n_\la)'(z)|^{-1}\leeq K|\mc C_n(\la)|,$$
where $z\in\mc C_n(\la)$. Thus, the assertion follows from (\ref{eq:sumas}) and definition (\ref{eq:Psi}).
\end{proof}

Note that if $N$ and $D$ are a complex functions defined on a set $X$, and $x_1, x_2\in X$, then
\begin{equation}\label{eq:ND}
   \bigg|\frac{N(x_1)}{D(x_1)}-\frac{N(x_2)}{D(x_2)}\bigg| \leeq\bigg|\frac{N(x_1)-N(x_2)}{D(x_1)}\bigg|+
   \bigg|\frac{N(x_2)}{D(x_2)}
   \Big(\frac{D(x_2)-D(x_1)}{D(x_1)}\Big)\bigg|.
\end{equation}

\begin{lem}\label{lem:const2}
   For every $\ve>0$ there exist $N\in\N$ and $\eta>0$, such that for every $0\leeq j \leeq k-1$ and $s_1,s_2\in\mc C^{j+}_n$ (or $z_1,z_2\in\mc C^{j-}_n$), we have
 $$\big|\dot\Psi_\la(s_1)-\dot\Psi_\la(s_2)\big|<\ve,$$
    where $n>N$ and $0<|\delta_\la|<\eta$.
\end{lem}

\begin{proof}
%First, note that $\frac{\partial}{\partial\la}F_\la$, $F_\la''$ and $F_\la'$ are Lipschitz functions.
Let $z_1$, $z_2\in \mc C_n^+(\la)$. We have $F_\la^{j-1}(z_1)$, $F_\la^{j-1}(z_2)\in \mc C_{n-j+1}(\la)$, so the fact that $\frac{\partial}{\partial\la}F_\la$ is a Lipschitz function, and proposition \ref{prop:aad} (\ref{pit:aad2}) gives us
 \begin{multline*}
    \bigg|\frac{(\frac{\partial}{\partial\la}F_\la)(F_\la^{j-1}(z_1))- (\frac{\partial}{\partial\la}F_\la)(F_\la^{j-1}(z_2))} {(F_\la^j)'(z_1)}\bigg|\\ \leeq K_1|\mc C_{n-j+1}(\la)| \frac{|\mc C_{n}(\la)|}{|\mc C_{n-j}(\la)|}\leeq K_2|\mc C_{n}(\la)|.
 \end{multline*}
Next, the fact that $\frac{\partial}{\partial\la}F_\la(z)=z+O(z^3)$, again Proposition \ref{prop:aad} (\ref{pit:aad2}), and Koebe Distortion Theorem leads to
 \begin{multline*}
    \bigg|\frac{(\frac{\partial}{\partial\la}F_\la)(F_\la^{j-1}(z_2))}{(F_\la^j)'(z_2)} \bigg(\frac{(F_\la^j)'(z_2)} {(F_\la^j)'(z_1)}-1 \bigg) \bigg|\\ \leeq K_3|\im F_\la^{j-1}(z_2)|\frac{|\mc C_{n}(\la)|}{|\mc C_{n-j}(\la)|}\frac{|\mc C_{n-j}(\la)|}{|\im F_\la^{j}(z_2)|}\leeq K_4|\mc C_{n}(\la)|.
 \end{multline*}
Thus, we get (cf. inequality (\ref{eq:ND}))
 \begin{equation*}
    \bigg|\frac{(\frac{\partial}{\partial\la}F_\la)(F_\la^{j-1}(z_1))} {(F_\la^j)'(z_1)}-
    \frac{(\frac{\partial}{\partial\la}F_\la)(F_\la^{j-1}(z_2))} {(F_\la^j)'(z_2)}\bigg|\leeq K_5|\mc C_{n}(\la)|.
 \end{equation*}
So, definition (\ref{eq:Psi}) and Corollary \ref{col:mod} lead to
   \begin{equation*}
      |\dot\Psi_\la(z_1)-\dot\Psi_\la(z_2)|\leeq K_5 n|\mc C_{n}(\la)|\leeq K_6 n^{-1/2}.
   \end{equation*}
Tee same inequalities (for suitable constants) holds also for $z_1$, $z_2\in \mc C_n^{j+}(\la)$ ($z_1$, $z_2\in \mc C_n^{j-}(\la)$) where $0\leeq j\leeq k-1$, thus the statement follows.
\end{proof}

\begin{lem}\label{lem:ppi}
   There exist $K_1$, $K_2>0$ and $\eta>0$ such that if $0<|\delta_\la|<\eta$ then
 \begin{equation*}
    |\dot\varphi_\la(s_1)|\leeq K_1|\dot\varphi_\la(s_2)|+K_2.
 \end{equation*}
   where $s_1$, $s_2\in \mc J_{\la_\dt}$, and $\tilde f_{\la_\dt}^j(s_1)=s_2$ or $\tilde f_{\la_\dt}^j(s_2)=s_1$, for $1\leeq j\leeq k-1$.
\end{lem}

\begin{proof}
We have $\tilde f_\la^j(\varphi_\la(s))=\varphi_\la(\tilde f_{\la_\dt}^j(s))$, therefore
   $$\frac{\partial \tilde f_\la^j}{\partial \la}(\varphi_\la(s))+\frac{\partial \tilde f_\la^j}{\partial z}(\varphi_\la(s))\cdot\dot\varphi_\la(s)=\dot\varphi_\la(\tilde f_{\la_\dt}^j(s)).$$
If $0<|\delta_\la|<\eta$, then we can assume that $\big|\frac{\partial \tilde f_\la^j}{\partial \la}(\varphi_\la)\big|$, $\big|\frac{\partial \tilde f_\la^j}{\partial z}(\varphi_\la)\big|$ are uniformly bounded, and $\frac{\partial \tilde f_\la^j}{\partial z}(\varphi_\la)$ is separated from $0$. So, the assertion follows.
\end{proof}

\begin{lem}\label{lem:const3}
   For every $\ve>0$ there exist $N\in\N$ and $\eta>0$, such that for every $0\leeq j \leeq k-1$ and $s_1,s_2\in\mc C^j_n$, we have
 \begin{equation*}
    \bigg|\re\Big(\frac{\mathfrak{D}^2_{F_\la}(\varphi_\la(s_1),\dot\Psi_\la(s_1))} {F_\la'(\varphi_\la(s_1))}\Big)- \re\Big(\frac{\mathfrak{D}^2_{F_\la}(\varphi_\la(s_2),\dot\Psi_\la(s_2))} {F_\la'(\varphi_\la(s_2))}\Big)\bigg|\leeq \ve e^{\ve n|\delta_\la|},
 \end{equation*}
   where $n>N$ and $0<|\delta_\la|<\eta$.
\end{lem}

\begin{proof}
Fix $\ve>0$. We can assume that $s_1,s_2\in\mc C^{j+}_n$. It follows from definition (\ref{eq:D2}) that
\begin{multline*}
   \mathfrak{D}^2_{F_\la}(\varphi_\la(s_1),\dot\Psi_\la(s_1))- \mathfrak{D}^2_{F_\la}(\varphi_\la(s_2),\dot\Psi_\la(s_2))\\ =\Big(\frac{\partial}{\partial \la} F_\la'\Big)(\varphi_\la(s_1))-\Big(\frac{\partial}{\partial \la} F_\la'\Big)(\varphi_\la(s_2))\\+F_\la''(\varphi_\la(s_1)) \big(\dot\Psi_\la(s_1) -\dot\Psi_\la(s_2)\big)+\dot\Psi_\la(s_2) \big(F_\la''(\varphi_\la(s_1))- F_\la''(\varphi_\la(s_2))\big).
\end{multline*}
Since $\frac{\partial}{\partial \la} F_\la'$ and $F_\la''$ are Lipschitz functions on $\mc J_\la$, Lemma \ref{lem:const2} gives us
\begin{equation}\label{eq:01}
   \big|\mathfrak{D}^2_{F_\la}(\varphi_\la(s_1),\dot\Psi_\la(s_1))- \mathfrak{D}^2_{F_\la}(\varphi_\la(s_2),\dot\Psi_\la(s_2))\big| \leeq \ve+|\dot\Psi_\la(s_2)|K_7|\mc C_{n}(\la)|.
\end{equation}
where $n>N$, $|\delta_\la|<\eta$, for suitably chosen $N\in\N$ and $\eta>0$.
Next, definition (\ref{eq:D2}) leads to
\begin{multline}\label{eq:02}
   \big|\mathfrak{D}^2_{F_\la}(\varphi_\la(s_2),\dot\Psi_\la(s_2)) \big(F_\la'(\varphi_\la(s_2)) - F_\la'(\varphi_\la(s_1))\big)\big|\\ \leeq \big(K_1+K_2|\dot\Psi_\la(s_2)|\big)|\mc C_{n}(\la)|.
\end{multline}
Since $F_\la'$ is separated from $0$, (\ref{eq:01}), (\ref{eq:02}) combined with (\ref{eq:ND}) gives us
\begin{equation}\label{eq:55}
   \bigg|\frac{\mathfrak{D}^2_{F_\la}(\varphi_\la(s_1),\dot\Psi_\la(s_1))} {F_\la'(\varphi_\la(s_1))}- \frac{\mathfrak{D}^2_{F_\la}(\varphi_\la(s_2),\dot\Psi_\la(s_2))} {F_\la'(\varphi_\la(s_2))}\bigg|\leeq \big(K_3+K_4|\dot\Psi_\la(s_2)|\big)|\mc C_{n}(\la)|.
\end{equation}

Since there exist points for which $\dot\Psi_\la$ is close to $\dot\varphi_\la$ (cf. Lemma \ref{lem:phiPsi}), Lemma \ref{lem:ppi} combined with Lemma \ref{lem:const2} lead to
   $$|\dot\Psi_\la(s_2)|\leeq K_5|\dot\Psi_\la(s_2^0)|+K_6,$$
where $\tilde f_{\la_\dt}^j(s_2^0)=s_2$ (i.e. $s_2^0\in\mc C_n$).
Using Corollary \ref{col:modpsi} and Lemma \ref{lem:xi} we obtain
\begin{equation*}
   \big|\dot\Psi_\la(s_2^0)\big|\leeq K_{7}\frac{e^{\ve n|\delta_\la|}}{|\im z|},
\end{equation*}
where $\varphi_\la(s_2^0)=z\in\mc C_n(\la)$. So, Proposition \ref{prop:C/im} gives us
\begin{equation*}
    \big(K_3+K_4|\dot\Psi_\la(s_2)|\big)|\mc C_{n}(\la)|\leeq \big(K_8+K_{9}\frac{e^{\ve n|\delta_\la|}}{|\im z|}\big)|\mc C_{n}(\la)|\leeq \ve e^{\ve n|\delta_\la|}.
\end{equation*}
Therefore assertion follows from (\ref{eq:55}).
\end{proof}

Before stating the next result we recall that $F_\la\circ\varphi_\la=\varphi_\la\circ F_{\la_\dt}$, moreover $\tilde f^k_\la=F_\la$ and $\tilde f_\la\circ\varphi_\la=\varphi_\la\circ\tilde f_{\la_\dt}$.

\begin{lem}\label{lem:pp}
   For every $\ve>0$ there exist $N\in\N$ and $\eta>0$ such that
 \begin{equation*}
    \bigg|\re\Big(\frac{\mathfrak{D}^2_{F_\la}(\varphi_\la(s),\dot\Psi_\la(s))} {F_\la'(\varphi_\la(s))}\Big)- \re\Big(\frac{\mathfrak{D}^2_{F_\la}(\varphi_\la(\tilde f^j_{\la_\dt}(s)),\dot\Psi_\la(\tilde f^j_{\la_\dt}(s)))} {F_\la'(\varphi_\la(\tilde f^j_{\la_\dt}(s)))}\Big)\bigg|\leeq \ve e^{\ve n|\delta_\la|},
 \end{equation*}
   where $s\in \mc C_n$, $1\leeq j\leeq k-1$, $n>N$ and $0<|\delta_\la|<\eta$.
\end{lem}

\begin{proof}
Fix $\ve>0$.
Of course it is enough to consider only consecutive points from a trajectory, i.e. $s$, and $\tilde f_{\la_\dt}(s)$ where $s\in\mc C_n^j$, $0\leeq j\leeq k-2$.

First, we will deal with the function $\dot\varphi_\la$ instead of $\dot\Psi_\la$.
Note that there exists a constant $K_1>0$ (cf. definition (\ref{eq:D2}) and (\ref{eq:D2cor})), such that
\begin{equation}\label{eq:20}
   \bigg|\frac{\partial}{\partial\la}\log\big|F'_\la(\varphi_\la)\big|- \re\Big(\frac{\mathfrak{D}^2_{F_\la}(\varphi_\la,\dot\Psi_\la)} {F_\la'(\varphi_\la)}\Big)\bigg|\leeq K_1\big|\dot\varphi_\la-\dot\Psi_\la\big|.
\end{equation}

{\em Step 1.}
The chain rule leads to
\begin{multline}\label{eq:d}
   \frac{\partial}{\partial\la}\log\big|F'_\la(\varphi_\la)\big|- \frac{\partial}{\partial\la}\log\big|F'_\la(\tilde f_\la(\varphi_\la))\big|\\= \frac{\partial}{\partial\la}\log\big|\tilde f'_\la(\varphi_\la)\big|- \frac{\partial}{\partial\la}\log\big|\tilde f'_\la(F_\la(\varphi_\la))\big|\\ = \re\Big(\frac{\frac{\partial}{\partial\la}(\tilde f_\la'(\varphi_\la))}{\tilde f_\la'(\varphi_\la)}\Big)-\re\Big(\frac{\frac{\partial}{\partial\la}(\tilde f_\la'(\varphi_\la(F_{\la_\dt})))}{\tilde f_\la'(\varphi_\la(F_{\la_\dt}))}\Big).
\end{multline}
Since $F_{\la_\dt}$ can be replaced by $\overline F_{\la_\dt}$ in the rightmost expression, we obtain
\begin{equation}\label{eq:de}
   \Big|\frac{\partial}{\partial\la}\log\big|F'_\la(\varphi_\la)\big|- \frac{\partial}{\partial\la}\log\big|F'_\la(\tilde f_\la(\varphi_\la))\big|\Big| \leeq\bigg|\frac{\frac{\partial}{\partial\la}(\tilde f_\la'(\varphi_\la(\overline F_{\la_\dt})))}{\tilde f_\la'(\varphi_\la(\overline F_{\la_\dt}))}-\frac{\frac{\partial}{\partial\la}(\tilde f_\la'(\varphi_\la))}{\tilde f_\la'(\varphi_\la)}\bigg|.
\end{equation}

{\em Step 2.} Note that
\begin{equation}\label{eq:dd/laz}
   \frac{\partial}{\partial\la}\big(\tilde f'_\la(\varphi_\la)\big)= \Big(\frac{\partial}{\partial\la}\tilde f'_\la\Big)(\varphi_\la)+ \tilde f_\la''(\varphi_\la)\cdot\dot\varphi_\la.
\end{equation}
Therefore
\begin{multline*}
   \frac{\partial}{\partial\la}\big(\tilde f_\la'(\varphi_\la(\overline F_{\la_\dt}))\big)-\frac{\partial}{\partial\la}\big(\tilde f_\la'(\varphi_\la)\big)=
   \Big(\frac{\partial}{\partial\la}\tilde f'_\la\Big)(\varphi_\la(\overline F_{\la_\dt}))-\Big(\frac{\partial}{\partial\la}\tilde f'_\la\Big)(\varphi_\la)\\
   +\tilde f_\la''(\varphi_\la(\overline F_{\la_\dt}))\cdot\dot\varphi_\la(\overline F_{\la_\dt})- \tilde f_\la''(\varphi_\la)\cdot\dot\varphi_\la.
\end{multline*}
We have $s\in\mc C_n^{j}$, so $\varphi_\la(s)\in\mc C_n^{j}(\la)$ and $\varphi_\la(\overline F_{\la_\dt}(s))\in\mc C_{n-1}^{j}(\la)$. Moreover $\varphi_\la(s)$, and $\varphi_\la(\overline F_{\la_\dt}(s))$ are on the same side of the real line.
Since $\frac{\partial}{\partial\la}\tilde f'_\la$ is a Lipschitz function on the Julia set, whereas bounded distortion gives us
  $$\big|\varphi_\la(\overline F_{\la_\dt}(s))-\varphi_\la(s)\big|\leeq K_2|\mc C_n(\la)|,$$
the fact that $\tilde f_\la''(z)$ does not depend on $z$ (cf. (\ref{eq:ftilde})), leads to
\begin{multline*}
   \Big|\frac{\partial}{\partial\la}\big(\tilde f_\la'(\varphi_\la(\overline F_{\la_\dt}(s)))\big)- \frac{\partial}{\partial\la}\big(\tilde f_\la'(\varphi_\la(s))\big)\Big|\\
   \leeq K_3|\mc C_n(\la)|+K_4\big|\dot\varphi_\la(\overline F_{\la_\dt}(s))- \dot\varphi_\la(s)\big|.
\end{multline*}

Next, using (\ref{eq:dd/laz}) and the fact that $\tilde f_\la'$ is a Lipschitz function, we obtain
\begin{equation*}
   \Big|\frac{\partial}{\partial\la}\big(\tilde f_\la'(\varphi_\la(s))\big)\Big(\tilde f_\la'(\varphi_\la(\overline F_{\la_\dt}(s)))-\tilde f_\la'(\varphi_\la(s))\Big)\Big|\leeq \big(K_5+K_6|\dot\varphi_\la(s)|\big)|\mc C_n(\la)|.
\end{equation*}
Since $\tilde f_\la'(\varphi_\la(\overline F_{\la_\dt}))$ and $\tilde f_\la'(\varphi_\la)$ are separated from $0$, the above inequalities combined (\ref{eq:de}) and (\ref{eq:ND}) gives us
\begin{multline}\label{eq:21}
   \Big|\frac{\partial}{\partial\la}\log\big|F'_\la(\varphi_\la(s))\big|- \frac{\partial}{\partial\la}\log\big|F'_\la(\tilde f_\la(\varphi_\la(s)))\big|\Big|\\ \leeq K_7\big|\dot\varphi_\la(s) -\dot\varphi_\la(\overline F_{\la_\dt}(s))\big|+ \big(K_8+K_9|\dot\varphi_\la(s)|\big)|\mc C_n(\la)|
\end{multline}

{\em Step 3.}
Let $\{s^j_n\}_{n\greq1}$ be a sequence from Lemma \ref{lem:phiPsi}. Then, using Lemma \ref{lem:phiPsi} and Lemma \ref{lem:const2}, we can get
\begin{equation*}
   \big|\dot\varphi_\la(s^j_n) -\dot\varphi_\la(\overline F_{\la_\dt}(s^j_n))\big|\leeq K_{10}|\mc C_n(\la)|+\ve
\end{equation*}
where $n>N$ ($N$ is large enough), and $0<|\delta_\la|<\eta$ ($\eta>0$ is close to $0$).
Next, Lemma \ref{lem:ppi} and Lemma \ref{lem:phiPsi} leads to
\begin{equation*}
   \big(K_8+K_9|\dot\varphi_\la(s_n^j)|\big)|\mc C_n(\la)|\leeq \big(K_{11}+K_{12}|\dot\Psi_\la(s^0_n)|\big)|\mc C_n(\la)|.
\end{equation*}

%Let $s_n:=s^0_n$.
Thus, (\ref{eq:21}) combined with (\ref{eq:20}), Lemma \ref{lem:phiPsi} and the above estimates, gives us
\begin{multline}\label{eq:33}
   \bigg|\re\Big(\frac{\mathfrak{D}^2_{F_\la}(\varphi_\la(s^j_n),\dot\Psi_\la(s^j_n))} {F_\la'(\varphi_\la(s^j_n))}\Big)- \re\Big(\frac{\mathfrak{D}^2_{F_\la}(\varphi_\la(\tilde f_\la(s^j_n)),\dot\Psi_\la(\tilde f_\la(s^j_n)))} {F_\la'(\varphi_\la(\tilde f_\la(s^j_n)))}\Big)\bigg|\\
   \leeq \big(K_{13}+K_{14}|\dot\Psi_\la(s^0_n)|\big)|\mc C_n(\la)|+\ve.
\end{multline}

Using Corollary \ref{col:modpsi}, Lemma \ref{lem:xi} and next Proposition \ref{prop:C/im} we obtain
\begin{equation*}
   \big(K_{13}+K_{14}|\dot\Psi_\la(s^0_n)|\big)|\mc C_n(\la)|+\ve\leeq K_{15}\frac{e^{\ve n|\delta_\la|}}{|\im z^0_n|}|\mc C_n(\la)|+\ve\leeq \ve K_{16}e^{\ve n|\delta_\la|},
\end{equation*}
where $\varphi_\la(s_n^0)=z_n^0\in\mc C_n(\la)$, $n>N$ and $0<|\delta_\la|<\eta$.
Thus, the statement follows from (\ref{eq:33}) and Lemma \ref{lem:const3}.
\end{proof}

Propositions \ref{prop:re=}, \ref{prop:beta} combined with definition (\ref{eq:D2}), Lemma  \ref{lem:xi}, and Lemmas \ref{lem:const3}, \ref{lem:pp}, lead to

\begin{prop}\label{prop:reDF}
   For every $\ve>0$ there exist $N\in\N$ and $\eta>0$ such that
 \begin{equation*}
    6\Gamma(n\delta_\la)-1-\ve e^{\ve n|\delta_\la|}\leeq \re\Big(\frac{\mathfrak{D}^2_{F_\la}(\varphi_\la(s),\dot\Psi_\la(s))} {F_\la'(\varphi_\la(s))}\Big)\leeq 6\Gamma(n\delta_\la)-1+\ve e^{\ve n|\delta_\la|},
 \end{equation*}
   where $s\in\textbf{C}_n$, $n>N$ and $0<|\delta_\la|<\eta$.
\end{prop}

\section{Integral over $\textbf{B}_N$}\label{sec:bn}

The goal of this section is to prove Proposition \ref{prop:far} (cf. \cite[Proof of Proposition 4.1]{HZi} and \cite[Section 12]{J}), and next Proposition \ref{prop:limb}.

%Now we will study integral of $|\dot\varphi_\la|$ over $\textbf{B}_N$ (cf. \cite[Proof of Proposition 4.1]{HZi} and \cite[Section 12]{J}).

We begin with the following lemma:

\begin{lem}\label{lem:derivativefn}
   There exists $\eta>0$ and for every $N\in\N$ there is a constant $K(N)>0$, such that if $z\in \mc J_\la$ and $|\delta_\la|<\eta$, then
      $$F_\la^n(z)\in\textbf{B}_N(\la)\Rightarrow
      \frac{1}{|(F_\la^n)'(z)|}\leeq\frac{K(N)}{n^{3/2}}.$$
\end{lem}

This lemma can be proven in the same way as \cite[Lemma 12.1]{J}, with one change. Instead of using the fact $|F_\la'(z)|>1$ for $z\in\mc B_N$ (which is not clear in the general case), we can use Lemma \ref{lem:exp}.

For every $N_0\in\N$ we define
$$\textbf{A}^{N_0}_{\infty}:=F_{\la_\dt}^{-N_0}(\textbf{P}).$$
Note that $\textbf{A}^{N_0}_{\infty}$ is a finite set (see definition of $\textbf{P}$ in section \ref{ssc:cylinders}).

Let us fix $N\in\N$.
Then, for every $N_0\in\N$ we define a family of sets $\{\textbf{A}^{N_0}_{N,n}\}_{n\greq0}$, which form a partition of $\textbf{B}_N\sms\textbf{A}^{N_0}_{\infty}$. Let
\begin{equation*}
   \textbf{A}^{N_0}_{N,n}:=
 \left\{
  \begin{array}{ll}
    F_{\la_\dt}^{-N_0}(\textbf{C}_{n})\cap\textbf{B}_N, & \hbox{for $n\greq1$;} \\ F_{\la_\dt}^{-N_0}(\textbf{B}_0)\cap\textbf{B}_N, & \hbox{for $n=0$.}
 \end{array}
   \right.
\end{equation*}
Moreover, for $n\greq1$, and $0\leeq j\leeq k-1$ let
\begin{equation*}
   A^{N_0,j}_{N,n}:=F_{\la_\dt}^{-N_0}(\mc C^j_{n})\cap\textbf{B}_N.
\end{equation*}
Thus we have $\textbf{A}^{N_0}_{N,n}=\bigcup_{j=0}^{k-1}A^{N_0,j}_{N,n}$, and let $A^{N_0}_{N,n}:=A^{N_0,0}_{N,n}$.

Next, the sets $\textbf{A}^{N_0}_{N,n}(\la)$, $A^{N_0,j}_{N,n}(\la)$, $A^{N_0}_{N,n}(\la)$ we define as the images of $\textbf{A}^{N_0}_{N,n}$, $A^{N_0,j}_{N,n}$, $A^{N_0}_{N,n}$ under $\varphi_\la$.
Observe, that this notation is different from \cite[Proof of Proposition 4.1]{HZi} and \cite[Section 12]{J}.

Repeating the proof of \cite[Proposition 5.1]{Jii}, we obtain:
\begin{lem}\label{lem:1}
   There exist $K>0$ and $\eta>0$ such that for every $N\in\N$, $N_0\greq1$, $n\greq1$, $0\leeq j\leeq k-1$ and $|\delta_\la|<\eta$ we have
      $$\tilde{\mu_\la}(A^{N_0,j}_{N,n})\leeq K\cdot N_0\cdot\tilde{\omega_\la}(\mc C^j_{n}).$$
\end{lem}

\begin{prop}\label{prop:far}
   There exists $\eta>0$ such that for every $N\in\N$ there is a constant $K(N)>0$ such that
      $$\int_{ \textbf{B}_N}\big|{\dot\varphi_\la}\big|d\tilde{\mu_\la}\leeq K(N),$$
   provided $0<|\delta_\la|<\eta$.
\end{prop}

\begin{proof}
Let us fix $N\in\N$. % and let $N_0\in\N$ be an integer large enough.

{\em Step 1.} Since the sets $\{\textbf{A}^{N_0}_{N,n}\}_{n\greq0}$ form the partition of $\textbf{B}_N$ (minus a finite set), we obtain
\begin{equation}\label{eq:rooz}
   \int_{\textbf{B}_N}|\dot\varphi_\la|d\tilde{\mu_\la}=
   \sum_{n=0}^\infty\int_{\textbf{A}^{N_0}_{N,n}}|\dot\varphi_\la|d\tilde{\mu_\la}\\
   =\int_{\textbf{A}^{N_0}_{N,0}}|\dot\varphi_\la|d\tilde{\mu_\la}+ \sum_{n=1}^\infty\sum_{j=0}^{k-1} \int_{A^{N_0,j}_{N,n}}|\dot\varphi_\la|d\tilde{\mu_\la}.
\end{equation}

%Since $\tilde f_\la^j(\varphi_\la)=\varphi_\la(\tilde f_{\la_\dt}^j)$, we have
   %$$\frac{\partial \tilde f_\la^j}{\partial \la}(\varphi_\la)+\frac{\partial \tilde %f_\la^j}{\partial z}(\varphi_\la)\cdot\dot\varphi_\la=\dot\varphi_\la(\tilde       %f_{\la_\dt}^j).$$
%Thus, there exist $K_1$, $K_2>0$ such that
Lemma \ref{lem:ppi} gives us
   $$|\dot\varphi_\la(s_1)|\leeq K_1|\dot\varphi_\la(s_2)|+K_2,$$
where $\tilde f_{\la_\dt}^j(s_1)=s_2$, or $\tilde f_{\la_\dt}^j(s_2)=s_1$, for $0\leeq j\leeq k-1$.
Hence, for $n\greq1$
\begin{equation}\label{eq:an}
   \sum_{j=0}^{k-1}\int_{A^{N_0,j}_{N,n}}|\dot\varphi_\la|d\tilde{\mu_\la}\leeq kK_1\int_{A^{N_0}_{N,n}}|\dot\varphi_\la|d\tilde{\mu_\la}+ K_2\tilde{\mu_\la}(\textbf{A}^{N_0}_{N,n}).
\end{equation}

Therefore, (\ref{eq:rooz}) leads to
\begin{equation}\label{eq:rooz1}
   \int_{\textbf{B}_N}|\dot\varphi_\la|d\tilde{\mu_\la} \leeq\int_{\textbf{A}^{N_0}_{N,0}}|\dot\varphi_\la|d\tilde{\mu_\la}+ kK_1\sum_{n=1}^\infty \int_{A^{N_0}_{N,n}}|\dot\varphi_\la|d\tilde{\mu_\la}+ K_2\tilde{\mu_\la}(\textbf{B}_N\sms\textbf{A}^{N_0}_{N,0}).
\end{equation}
%where $\mu_\la(\textbf{B}_N\sms\textbf{A}^{N_0}_{N,0})$ is bounded by a constant depending on $N$.

{\em Step 2.}
We will use the formula (\ref{eq:sumas}). First, if $s\in\textbf{A}^{N_0}_{N,0}$, we get
\begin{equation}\label{eq:199}
   \dot\varphi_\la(s)= -\sum_{j=1}^{N_0}\frac{(\frac{\partial}{\partial\la}F_\la)(F_\la^{j-1}(\varphi_\la(s)))} {(F_\la^j)'(\varphi_\la(s))}+ \frac{\dot\varphi_\la(F_{\la_\dt}^{N_0}(s))}{(F_\la^{N_0})'(\varphi_\la(s))}.
\end{equation}
Next, for $s\in A^{N_0}_{N,n}$, $n\greq1$, we can split $\dot\varphi_\la(s)$ into three parts:
\begin{multline}\label{eq:200}
   \dot\varphi_\la(s)=
   -\sum_{j=1}^{N_0}
   \frac{(\frac{\partial}{\partial\la}F_\la)(F_\la^{j-1}(\varphi_\la(s)))}
   {(F_\la^j)'(\varphi_\la(s))}
   \\
   -\frac{1}{(F_\la^{N_0})'(\varphi_\la(s))}
   \sum_{j=1}^{n}\frac
   {(\frac{\partial}{\partial\la}F_\la)(F_\la^{N_0+j-1}(\varphi_\la(s)))}
   {(F_\la^j)'(F_\la^{N_0}(\varphi_\la(s)))}+
   \frac{\dot\varphi_\la(F_{\la_\dt}^{N_0+n}(s))}{(F_\la^{N_0+n})'(\varphi_\la(s))}.
\end{multline}

{\em Step 3.}
Now, according to (\ref{eq:rooz1}), we will estimate integral of the "tails" of (\ref{eq:199}), (\ref{eq:200}), i.e.
\begin{equation*}
   \int_{\textbf{A}^{N_0}_{N,0}} \Big|\frac{\dot\varphi_\la(F_{\la_\dt}^{N_0})} {(F_\la^{N_0})'(\varphi_\la)}\Big|d\tilde{\mu_\la}+ kK_1\sum_{n=1}^\infty \int_{A^{N_0}_{N,n}}\Big|\frac{\dot\varphi_\la(F_{\la_\dt}^{N_0+n})} {(F_\la^{N_0+n})'(\varphi_\la)}\Big|d\tilde{\mu_\la}.
\end{equation*}
We have $F_{\la_\dt}^{N_0}(\textbf{A}^{N_0}_{N,0})\subset \textbf{B}_0$ and $F_{\la_\dt}^{N_0+n}(A^{N_0}_{N,n})\subset \textbf{B}_0$. So, using Lemma \ref{lem:derivativefn} and the fact that measure $\tilde{\mu_\la}$ is $F_{\la_\dt}$-invariant, the above expression can be estimated by
\begin{multline}\label{eq:201}
   K_3\int_{\textbf{A}^{N_0}_{N,n}}\frac{\big|\dot\varphi_\la(F_{\la_\dt}^{N_0})\big|}
   {{N_0}^{3/2}}d\tilde{\mu_\la}+
   K_4\sum_{n=1}^\infty \int_{A^{N_0}_{N,n}}\frac{\big|\dot\varphi_\la(F_{\la_\dt}^{N_0+n})\big|}
   {(N_0+n)^{3/2}}d\tilde{\mu_\la}
   \\\leeq
   \sum_{n=0}^\infty\frac{K_5}{(N_0+n)^{3/2}}
   \int_{\textbf{B}_0}\big|\dot\varphi_\la\big|d\tilde{\mu_\la}\leeq \frac{K_6}{N_0^{1/2}}\int_{\textbf{B}_0}\big|\dot\varphi_\la\big|d\tilde{\mu_\la}.
\end{multline}

{\em Step 4.}
Now we will deal with the second expression from the right-hand side of (\ref{eq:200}).
In this proof, we need to estimate integral of this expression over the set $\bigcup_{n=1}^{\infty}A^{N_0}_{N,n}$. But, later on we will need slightly more general result, thus we will give an estimate over the set $\bigcup_{n=\tilde n}^{\infty}A^{N_0}_{N,n}$, where $\tilde n\greq1$.

First, using the fact that $\frac{\partial}{\partial\la}F_\la(z)=z+O(z^3)$, and definition (\ref{eq:psi}), we get
\begin{multline*}
   kK_1\sum_{n=\tilde n}^\infty\int_{A^{N_0}_{N,n}}\Big|\frac{1}{(F_\la^{N_0})'(\varphi_\la)} \sum_{j=1}^{n} \frac{(\frac{\partial}{\partial\la}F_\la)(F_\la^{N_0+j-1}(\varphi_\la))} {(F_\la^j)'(F_\la^{N_0}(\varphi_\la))}\Big|d\tilde{\mu_\la}\\\leeq K_7\sum_{n=\tilde n}^\infty\int_{A^{N_0}_{N,n}}\Big| \sum_{j=1}^{n} \frac{F_\la^{N_0+j-1}(\varphi_\la)} {(F_\la^j)'(F_\la^{N_0}(\varphi_\la))}\Big|d\tilde{\mu_\la}= K_7\sum_{n=\tilde n}^\infty\int_{A^{N_0}_{N,n}}|\dot\psi_\la(F_\la^{N_0}(\varphi_\la))| d\tilde{\mu_\la}
\end{multline*}
Next, Lemma \ref{lem:1} leads to
\begin{multline*}
   K_7\sum_{n=\tilde n}^\infty\int_{A^{N_0}_{N,n}}|\dot\psi_\la(F_\la^{N_0}(\varphi_\la))| d\tilde{\mu_\la}\leeq
   K_7\sum_{n=\tilde n}^\infty\sup_{z\in \mc C_{n}(\la)}\big|\dot\psi_\la(z)\big|\cdot\mu_\la(A^{N_0}_{N,n}(\la))\\\leeq K_8
   N_0\sum_{n=\tilde n}^\infty\sup_{z\in \mc C_{n}(\la)}\big|\dot\psi_\la(z)\big|\cdot\omega_\la(\mc C_{n}(\la)).
\end{multline*}

%Now, we estimate $|\dot\psi_\la(z)|$. Let $z\in\mc C_n(\la)$. We see from definition (\ref{eq:beta}), combined with Lemma \ref{lem:realpart}, that there exist $\ve>0$ and $K_9>1$, such that
%   $$\big|\dot\psi_\la(z)\big|\leeq K_9\frac{\beta_\la(z)}{|\im z|}+\ve,$$
%where $n>N$, $0<|\delta_\la|<\eta$, for suitably chosen $N\in\N$ and $\eta>0$. Next, because the function $|\Gamma|$ is bounded by $1/2$, Proposition \ref{prop:beta} leads to
%\begin{equation}\label{eq:p}
%   \big|\dot\psi_\la(z)\big|\leeq K_{10}\frac{\,e^{\ve n|\delta_\la|}}{|\im z|}.
%\end{equation}
%Moreover, increasing constant $K_{10}$ if necessary, we can assume that (\ref{eq:p}) holds for $z\in \mc C_n(\la)$ where $n\greq1$, $0<|\delta_\la|<\eta$

Fix $\ve>0$ (small), and let $z_{n}$ be an arbitrary point from $\mc C_{n}(\la)$. Then, Corollary \ref{col:modpsi} and next Corollary \ref{col:odl} gives us
\begin{multline*}
   K_8 N_0 \sum_{n=\tilde n}^\infty\sup_{z\in\mc C_{n}(\la)}\big|\dot\psi_\la(z)\big|\omega_\la\mc(C_{n}(\la))\\ \leeq K_{9} N_0\sum_{n=\tilde n}^\infty\sup_{z\in\mc C_{n}(\la)}\frac{\omega_\la(\mc C_{n}(\la))}{|\im z|}\,e^{\ve n|\delta_\la|}
   \leeq K_{10} N_0\sum_{n=\tilde n}^\infty\frac{\omega_\la(\mc C_{n}(\la))}{|\im z_{n}|}\,e^{\ve n|\delta_\la|}.
\end{multline*}
Next, the fact that $\omega_\la$ is a geometric measure, and Proposition \ref{prop:rn} leads to
\begin{multline*}
   K_{10} N_0\sum_{n=\tilde n}^\infty\frac{\omega_\la(\mc C_{n}(\la))}{|\im z_{n}|}\,e^{\ve n|\delta_\la|}\leeq K_{11} N_0\sum_{n=\tilde n}^\infty\frac{|\mc C_{n}(\la)|^{\mc D(\la)}}{|\im z_{n}|}\,e^{\ve n|\delta_\la|}\\
   \leeq K_{12} N_0\sum_{n=\tilde n}^\infty|\im z_{n}|^{3\mc D(\la)-1}e^{\mc D(\la)(-2n\delta+\ve n|\delta_\la|)}e^{\ve n|\delta_\la|},
\end{multline*}
where $n\greq1$ and $0<|\delta_\la|<\eta$, for suitably chosen $\eta>0$.

We  split the above sum into two parts (for $n\leeq 1/|\delta_\la|$ and $n>1/|\delta_\la|$).

If $n|\delta_\la|\leeq1$, then the exponential function is bounded. So, using Corollary \ref{col:odl} (\ref{cit:odl1}), the first part (which may by empty) can be estimated as follows:
\begin{equation*}\label{eq:300}
   K_{13} N_0\sum_{n=\tilde n}^{[1/|\delta_\la|]}|\im z_{n}|^{3\mc D(c)-1}\leeq
   K_{14} N_0\sum_{n=\tilde n}^{[1/|\delta_\la|]}n^{-\frac{3}{2}\mc D(\la)+\frac{1}{2}}\\
   \leeq K_{15} N_0 \tilde n^{-\frac{3}{2}\mc D(\la)+\frac32}.
\end{equation*}
For $n\delta_\la>1$ (so $\delta_\la>0$) Corollary \ref{col:odl} (\ref{cit:odl2}), and the fact that $\mc D(\la)>1$, gives
\begin{multline*}
   K_{12} N_0\sum_{n=[1/\delta_\la]+1}^{\infty}|\im z_{n}|^{3\mc D(c)-1}e^{\mc D(\la)(-2n\delta+\ve n\delta_\la)+\ve n\delta_\la}\\
   \leeq K_{16} N_0\sum_{n=[1/\delta_\la]+1}^{\infty}\delta_\la^{\frac{3}{2}\mc D(\la)-\frac{1}{2}}e^{-2 n\delta_\la}
   \leeq K_{17}\delta_\la^{\frac{3}{2}\mc D(\la)-\frac32}.
\end{multline*}
If $n\delta_\la<-1$ (so $\delta_\la<0$), then Corollary \ref{col:odl} (3) leads to
\begin{multline*}
   K_{12} N_0\sum_{n=[1/|\delta_\la|]+1}^{\infty}|\im z_{n}|^{3\mc D(c)-1}e^{\mc D(\la)(-2n\delta+\ve n|\delta_\la|)+\ve n|\delta_\la|}\\
   \leeq K_{18} N_0\sum_{n=[1/|\delta_\la|]+1}^{\infty}|\delta_\la|^{\frac{3}{2}\mc D(\la)-\frac{1}{2}}e^{-(\mc D(\la)-1-4\ve\mc D(\la))n|\delta_\la|}.
   %\leeq K|\delta_\la|^{\frac{3}{2}\mc D(\la)-\frac32}
\end{multline*}
We can assume that $(\mc D(\la)-1-4\ve\mc D(\la))>0$, therefore the above expression (similarly as in the case $\delta_\la>0$) can be estimated by $K_{19}\delta_\la^{3\mc D(\la)/2-3/2}$.

Thus, the above estimates gives us
\begin{multline}\label{eq:202}
   kK_1\sum_{n=\tilde n}^{\infty}\int_{A^{N_0}_{N,n}}\Big|\frac{1}{(F_\la^{N_0})'(\varphi_\la)}
   \sum_{j=1}^{n} \frac{(\frac{\partial}{\partial\la}F_\la)(F_\la^{N_0+j-1}(\varphi_\la))}
   {(F_\la^j)'(F_\la^{N_0}(\varphi_\la))}\Big|d\tilde{\mu_\la}\\ \leeq K_{12} N_0\sum_{n=\tilde n}^\infty|\im z_{n}|^{3\mc D(\la)-1}e^{\mc D(\la)(-2n\delta+\ve n|\delta_\la|)}e^{\ve n|\delta_\la|}\\
   \leeq K_{20}N_0(\tilde n^{-\frac{3}{2}\mc D(\la)+\frac32}+|\delta_\la|^{\frac{3}{2}\mc D(\la)-\frac32}).
\end{multline}

{\em Step 5.}
Now we will estimate integral of first expression from the right-hand side of (\ref{eq:199}) and (\ref{eq:200}).
Since $(\frac{\partial}{\partial\la}F_\la)(F_\la^{j-1}(\varphi_\la))$ is uniformly boun\-ded, estimate (\ref{eq:rooz1}) leads to
\begin{multline}\label{eq:203}
   \int_{\textbf{A}^{N_0}_{N,0}}\Big|\sum_{j=1}^{N_0}
   \frac{(\frac{\partial}{\partial\la}F_\la)(F_\la^{j-1}(\varphi_\la))}
   {(F_\la^j)'(\varphi_\la)}\Big|d\tilde{\mu_\la}\\ +kK_1\sum_{n=1}^\infty\int_{A^{N_0}_{N,n}}\Big|\sum_{j=1}^{N_0}
   \frac{(\frac{\partial}{\partial\la}F_\la)(F_\la^{j-1}(\varphi_\la))}
   {(F_\la^j)'(\varphi_\la)}\Big|d\tilde{\mu_\la}\leeq
   N_0\cdot K_{21}\cdot\tilde{\mu_\la}(\textbf{B}_N).
\end{multline}

{\em Step 6.}
We conclude from (\ref{eq:rooz1}), and (\ref{eq:199}), (\ref{eq:200}) combined with (\ref{eq:203}), (\ref{eq:201}), (\ref{eq:202}) for $\tilde n=1$, that
\begin{multline*}
   \int_{\textbf{B}_N}\big|\dot\varphi_\la\big|d\tilde{\mu_\la}\leeq
   N_0\cdot K_{21}\cdot\tilde{\mu_\la}(\textbf{B}_N)+K_{20}N_0(1+|\delta_\la|^{\frac{3}{2}\mc D(\la)-\frac32})\\+
   \frac{K_6}{N_0^{1/2}}
   \int_{\textbf{B}_0}\big|\dot\varphi_\la\big|d\tilde{\mu_\la}+ K_2\tilde{\mu_\la}(\textbf{B}_N\sms\textbf{A}^{N_0}_{N,0}).
\end{multline*}
Chosen suitable large $N_0$, we get $K_{6}/N_0^{1/2}\leeq1/2$. Therrefore, the fact that $B_0\subset B_N$ leads to
\begin{equation*}
   \frac12\int_{\textbf{B}_N}\big|\dot\varphi_\la\big|d\tilde{\mu_\la}\leeq
   N_0\cdot K_{21}\cdot\tilde{\mu_\la}(\textbf{B}_N)\\+K_{20}N_0(1+|\delta_\la|^{\frac{3}{2}\mc D(\la)-\frac32})+K_2\tilde{\mu_\la}(\textbf{B}_N).
\end{equation*}
Since $\mc D(\la)>1$, whereas $\tilde{\mu_\la}(\textbf{B}_N)$ is bounded (cf. Lemma \ref{lem:L}), the assertion follows.
\end{proof}

Let $U_{N,\tilde n,\la}^{N_0}:\textbf{B}_N\rightarrow\C$, where
%\begin{equation*}
%\begin{array}{ll}
%   U_{N,\la}^{N_0}(s)= &-\sum_{j=1}^{N_0+n} \frac{(\frac{\partial}{\partial\la}F_\la) (F_\la^{j-1}(\varphi_\la(s)))}{(F_\la^j)'(\varphi_\la(s))} \textrm{ for }s\in A^{N_0}_{N,n} \textrm{ where } n< N\\&-\sum_{j=1}^{N_0} \frac{(\frac{\partial}{\partial\la}F_\la)(F_\la^{j-1}(\varphi_\la(s)))} {(F_\la^j)'(\varphi_\la(s))} \textrm{ for }s\in A^{N_0}_{N,n} \textrm{ where } n\greq N
%\end{array}
%\end{equation*}
\begin{equation*}
   U_{N,\tilde n,\la}^{N_0}(s)=\left\{
  \begin{array}{lll}
    -\sum_{j=1}^{N_0+n} \frac{(\frac{\partial}{\partial\la}F_\la) (F_\la^{j-1}(\varphi_\la(s)))}{(F_\la^j)'(\varphi_\la(s))}, & \hbox{for $s\in \textbf{A}^{N_0}_{N,n}$ where $n<\tilde n$;} \\
    -\sum_{j=1}^{N_0} \frac{(\frac{\partial}{\partial\la}F_\la) (F_\la^{j-1}(\varphi_\la(s)))}{(F_\la^j)'(\varphi_\la(s))}, & \hbox{for $s\in \textbf{A}^{N_0}_{N,n}$ where $n\greq\tilde n$;}\\
    0, & \hbox{for $s\in\textbf{A}^{N_0}_{\infty}\cap\textbf{B}_N $.}
  \end{array}
\right.
\end{equation*}
Note that $\textbf{A}^{N_0}_{\infty}\cap\textbf{B}_N$ is a finite set.

\begin{col}\label{col:a-b}
For every $\ve>0$, $N\in\N$ there exist $N_0\in\N$, $\tilde n\in\N$ and $\eta>0$ such that
   $$\int_{\textbf{B}_N}\big|\dot\varphi_\la-U_{N,\tilde n,\la}^{N_0}\big|d\tilde{\mu_\la}<\ve,$$
where $0<|\delta_\la|<\eta$.
\end{col}

\begin{proof}
Formula (\ref{eq:sumas}) (cf. (\ref{eq:199}), and (\ref{eq:200})), leads to
\begin{multline}\label{eq:250}
  \int_{\textbf{B}_N}\big|\dot\varphi_\la-U_{N,\tilde n,\la}^{N_0} \big|d\tilde{\mu_\la}\\
   =\sum_{n=0}^{\tilde n-1}\int_{\textbf{A}^{N_0}_{N,n}} \bigg|\frac{\dot\varphi_\la(F_{\la_\dt}^{N_0+n})} {(F_\la^{N_0+n})'(\varphi_\la)}\bigg|d\tilde{\mu_\la}
    +\sum_{n=\tilde n}^\infty\int_{\textbf{A}^{N_0}_{N,n}} \bigg|\frac{\dot\varphi_\la(F_{\la_\dt}^{N_0})} {(F_\la^{N_0})'(\varphi_\la)}\bigg|d\tilde{\mu_\la}.
    %+\sum_{n=\tilde n}^\infty\int_{\textbf{A}^{N_0}_{N,n}} \bigg|\frac{1}{(F_\la^{N_0})'(\varphi_\la)} \sum_{j=1}^{n} \frac{(\frac{\partial}{\partial\la}F_\la)(F_\la^{N_0+j-1}(\varphi_\la))} {(F_\la^j)'(F_\la^{N_0}(\varphi_\la))}\bigg|d\tilde{\mu_\la}
\end{multline}

Since $\tilde f_\la^j(\varphi_\la(F_{\la_\dt}^{N_0}))=\varphi_\la(F_{\la_\dt}^{N_0}(\tilde f_{\la_\dt}^j))$ (cf. proof of Proposition \ref{prop:far},\emph{ Step 1}), we have
   $$\frac{\partial \tilde f_\la^j}{\partial \la}\big(\varphi_\la(F_{\la_\dt}^{N_0})\big) +\frac{\partial \tilde f_\la^j}{\partial z}\big(\varphi_\la(F_{\la_\dt}^{N_0})\big)\cdot \dot\varphi_\la(F_{\la_\dt}^{N_0})= \dot\varphi_\la\big(F_{\la_\dt}^{N_0}(\tilde f_{\la_\dt}^j)\big).$$
Thus, there exist $K_1$, $K_2>0$ such that
   $$\big|\dot\varphi_\la(F_{\la_\dt}^{N_0}(s_1))\big|\leeq K_1\big|\dot\varphi_\la(F_{\la_\dt}^{N_0}(s_2))\big|+K_2,$$
where $\tilde f_{\la_\dt}^j(s_1)=s_2$, or $\tilde f_{\la_\dt}^j(s_2)=s_1$, for $0\leeq j\leeq k-1$.
Hence, for $n\greq1$
\begin{multline}\label{eq:311}
  \int_{\textbf{A}^{N_0}_{N,n}} \bigg|\frac{\dot\varphi_\la(F_{\la_\dt}^{N_0})} {(F_\la^{N_0})'(\varphi_\la)}\bigg|d\tilde{\mu_\la}\\
   \leeq \sup_{s\in\textbf{A}^{N_0}_{N,n}}\frac{1}{|(F_\la^{N_0})'(\varphi_\la(s))|} \bigg(kK_1\int_{A^{N_0}_{N,n}} \big|\dot\varphi_\la(F_{\la_\dt}^{N_0})\big|d\tilde{\mu_\la}+ K_2\tilde{\mu_\la}(\textbf{A}^{N_0}_{N,n})\bigg).
\end{multline}
We conclude from (\ref{eq:sumas}) that
\begin{multline*}
  \int_{A^{N_0}_{N,n}} \big|\dot\varphi_\la(F_{\la_\dt}^{N_0})\big|d\tilde{\mu_\la} \leeq
   \int_{A^{N_0}_{N,n}} \bigg|\sum_{j=1}^{n}\frac {(\frac{\partial}{\partial\la}F_\la)(F_\la^{N_0+j-1}(\varphi_\la))} {(F_\la^j)'(F_\la^{N_0}(\varphi_\la))}\bigg|d\tilde{\mu_\la}\\+
    \int_{A^{N_0}_{N,n}}\bigg|\frac{\dot\varphi_\la(F_{\la_\dt}^{N_0+n})} {(F_\la^{n})'(F_\la^{N_0}(\varphi_\la))}\bigg|d\tilde{\mu_\la}
\end{multline*}
Thus, (\ref{eq:311}) gives us
\begin{multline*}
  \int_{\textbf{A}^{N_0}_{N,n}} \bigg|\frac{\dot\varphi_\la(F_{\la_\dt}^{N_0})} {(F_\la^{N_0})'(\varphi_\la)}\bigg|d\tilde{\mu_\la}\leeq K_3\int_{A^{N_0}_{N,n}} \bigg|\sum_{j=1}^{n}\frac {(\frac{\partial}{\partial\la}F_\la)(F_\la^{N_0+j-1}(\varphi_\la))} {(F_\la^j)'(F_\la^{N_0}(\varphi_\la))}\bigg|d\tilde{\mu_\la}\\
   +K_4\sup_{s\in\textbf{A}^{N_0}_{N,n}}\frac{1}{|(F_\la^{N_0})'(\varphi_\la(s))|} \int_{A^{N_0}_{N,n}} \bigg|\frac{\dot\varphi_\la(F_{\la_\dt}^{N_0+n})} {(F_\la^{n})'(F_\la^{N_0}(\varphi_\la))}\bigg|d\tilde{\mu_\la}+ K_5\tilde{\mu_\la}(\textbf{A}^{N_0}_{N,n}).
\end{multline*}
Next, estimate (\ref{eq:202}), Lemma \ref{lem:1}, and the fact that distortion of $F_\la^n$ on $F_\la^{N_0}(A^{N_0,j}_{N,n}(\la))\subset\mc C_n^j(\la)$ is bounded, leads to
\begin{multline*}\label{eq:245}
  \sum_{n=\tilde n}^{\infty}\int_{\textbf{A}^{N_0}_{N,n}} \bigg|\frac{\dot\varphi_\la(F_{\la_\dt}^{N_0})} {(F_\la^{N_0})'(\varphi_\la)}\bigg|d\tilde{\mu_\la}\leeq
   K_6N_0(\tilde n^{-\frac{3}{2}\mc D(\la)+\frac32}+|\delta_\la|^{\frac{3}{2}\mc D(\la)-\frac32})\\
    +K_4\sum_{n=\tilde n}^{\infty} \sup_{s\in\textbf{A}^{N_0}_{N,n}}\frac{1}{|(F_\la^{N_0+n})'(\varphi_\la(s))|} \int_{A^{N_0}_{N,n}} \big|\dot\varphi_\la(F_{\la_\dt}^{N_0+n})\big|d\tilde{\mu_\la}\\ +K_7 N_0\sum_{n=\tilde n}^\infty\sum_{j=0}^{k-1}\tilde{\omega_\la}(\mc C^{j}_{n}).
\end{multline*}
Note that $F_{\la_\dt}^{N_0+n}(A^{N_0}_{N,n})\subset \textbf{B}_0$, and the measure $\tilde{\mu_\la}$ is $F_{\la_\dt}$-invariant. So, the above estimate combined with (\ref{eq:250}), and Lemma \ref{lem:derivativefn}, gives
\begin{multline*}
  \int_{\textbf{B}_N}\big|\dot\varphi_\la-U_{N,\tilde n,\la}^{N_0} \big|d\tilde{\mu_\la}
   \leeq K_6N_0\big(\tilde n^{-\frac{3}{2}\mc D(\la)+\frac32}+|\delta_\la|^{\frac{3}{2}\mc D(\la)-\frac32}\big)\\
    +K_8\sum_{n=0}^{\infty}\frac{1}{(N_0+n)^{3/2}} \int_{\textbf{B}_0} \big|\dot\varphi_\la\big|d\tilde{\mu_\la}+ K_7 N_0\sum_{n=\tilde n}^\infty\sum_{j=0}^{k-1}\tilde{\omega_\la}(\mc C^{j}_{n}).
\end{multline*}
Finally, using Corollary \ref{col:mod}, we obtain
   $$\sum_{n=\tilde n}^\infty\sum_{j=0}^{k-1}\tilde{\omega_\la}(\mc C^{j}_{n})\leeq K_9k\sum_{n=\tilde n}^\infty|\mc C_n|^{\mc D(\la)}\leeq K_{10}\sum_{n=\tilde n}^\infty n^{-\frac32\mc D(\la)}\leeq K_{11}{\tilde n}^{-\frac32\mc D(\la)+1}$$
Since $\int_{\textbf{B}_0} |\dot\varphi_\la|d\tilde{\mu_\la}$ is bounded (see Proposition \ref{prop:far}), the assertion follows.
\end{proof}

\begin{prop}\label{prop:limb}
   For every $N\in\N$ the limit
      $$\lim_{\la\rightarrow-1} \int_{\textbf{B}_N} \re\Big(\frac{\frac{\partial}{\partial\la}F'_\la(\varphi_\la)} {F_\la'(\varphi_\la)}\Big)d\tilde{\mu_\la}$$
   exists.
\end{prop}

\begin{proof}
The functions $\varphi_\la$ converge uniformly to $\varphi_{-1}$, when $\la\rightarrow-1$.
Moreover, for every $N$, $N_0$, $\tilde n\in\N$, the functions $U_{N,\tilde n,\la}^{N_0}$ converge uniformly to $U_{N,\tilde n,-1}^{N_0}$. So, Corollary \ref{col:lim} leads to
\begin{equation}\label{eq:c1}
   \lim_{\la\rightarrow-1}\int_{\textbf{B}_N} \re\bigg(\frac{\mathfrak{D}^2_{F_\la}(\varphi_\la,U_{N,\tilde n,\la}^{N_0})} {F_\la'(\varphi_\la)}\bigg)d\tilde{\mu_\la}=
   \int_{\textbf{B}_N} \re\bigg(\frac{\mathfrak{D}^2_{F_{-1}}(\varphi_{-1},U_{N,\tilde n,-1}^{N_0})} {F_{-1}'(\varphi_{-1})}\bigg)d\tilde{\mu}_{-1}.
\end{equation}

We have (cf. definition (\ref{eq:D2}) and (\ref{eq:D2cor}))
   $$\frac{\partial}{\partial\la}F'_\la(\varphi_\la)- \mathfrak{D}^2_{F_\la}(\varphi_\la,U_{N,\tilde n,\la}^{N_0})= \big(\dot\varphi_\la-U_{N,\tilde n,\la}^{N_0}\big)F_\la''(\varphi_\la).$$
Thus, using Corollary \ref{col:a-b}, we see that there exists a constant $K$ such that for every $\ve>0$ we can get
\begin{equation*}
   \bigg|\int_{\textbf{B}_N} \re\bigg(\frac{\frac{\partial}{\partial\la}F'_\la(\varphi_\la)} {F_\la'(\varphi_\la)}\bigg)d\tilde{\mu_\la}-
   \int_{\textbf{B}_N} \re\bigg(\frac{\mathfrak{D}^2_{F_\la}(\varphi_\la,U_{N,\tilde n,\la}^{N_0})} {F_\la'(\varphi_\la)}\bigg)d\tilde{\mu_\la}\bigg|
   \leeq K\ve\tilde\mu_\la(\textbf{B}_N),
\end{equation*}
where $\tilde n$, $N_0$ are large enough, and $\la$ is close to $-1$. Since $\tilde\mu_\la(\textbf{B}_N)$ is bounded (cf. Lemma \ref{lem:L}) we can assume that $K\ve\tilde\mu_\la(\textbf{B}_N)$ is small, so the statement follows from (\ref{eq:c1}).
\end{proof}

\section{Integral over $\textbf{M}_N$}\label{sec:mn}

First, note that formula (\ref{eq:wzor}) gives us
\begin{equation}\label{eq:formula1}
   \mc D'(\la)=-\mc D(\la)\frac{\int_{\textbf{M}_N}\re\Big(\frac{\frac{\partial}{\partial \la} (F_\la'(\varphi_\la))}{F_\la'(\varphi_\la)}\Big)d\tilde{\mu_\la}+\int_{\textbf{B}_N}\re\Big(\frac{\frac{\partial}{\partial \la} (F_\la'(\varphi_\la))}{F_\la'(\varphi_\la)}\Big)d\tilde{\mu_\la}}{\int_{\mc J_{\la_\dt}}\log|F'_\la(\varphi_\la)|d\tilde{\mu_\la}}.
\end{equation}
Thus, we have to estimate integral over $\textbf{M}_N$ (cf. Propositions \ref{prop:den}, \ref{prop:limb}). Of course $\mc D(\la)\rightarrow\mc D(-1)$ when $\la\rightarrow-1$.

We begin with the following lemma:
\begin{lem}\label{lem:f-P}
   There exists $\eta>0$ and $K>0$ such that for every $N\in\N$ we have
 \begin{equation*}
    \int_{\textbf{M}_N}\bigg|\re\Big(\frac{\frac{\partial}{\partial \la} (F_\la'(\varphi_\la))}{F_\la'(\varphi_\la)}\Big)-\re\Big(\frac{\mathfrak{D}^2_{F_\la}(\varphi_\la,\dot\Psi_\la)} {F_\la'(\varphi_\la)}\Big)\bigg|d\tilde{\mu_\la}
    \leeq \frac{K}{N^{1/2}},
 \end{equation*}
   where $0<|\delta_\la|<\eta$.
\end{lem}

\begin{proof}
We see from inequality (\ref{eq:20}) (cf. (\ref{eq:D2}) and (\ref{eq:D2cor})) that it is enough to estimate integral of $|\dot\varphi_\la-\dot\Psi_\la|$.
Formula (\ref{eq:sumas}) and definition (\ref{eq:Psi}) gives us
   \begin{equation*}
      \int_{ \textbf{M}_N}\big|\dot\varphi_\la-\dot\Psi_\la\big|d\tilde{\mu_\la}= \sum_{n>N}\int_{ \textbf{C}_n} \Big|\frac{\dot\varphi_\la(F^n_{\la_\dt})}{(F^n_\la)'(\varphi_\la)}\Big| d\tilde{\mu_\la}.
   \end{equation*}
If $\varphi_\la(s)\in\textbf{C}_n(\la)$, then $F^n_\la(\varphi_\la(s))\in\textbf{B}_0(\la)$, therefore Lemma \ref{lem:derivativefn} leads to
   \begin{equation*}
      \int_{ \textbf{M}_N}\big|\dot\varphi_\la-\dot\Psi_\la\big|d\tilde{\mu_\la}\leeq \sum_{n>N}\frac{K(0)}{n^{3/2}}\int_{ \textbf{C}_n} \big|\dot\varphi_\la(F^n_{\la_\dt})\big| d\tilde{\mu_\la}.
   \end{equation*}
Since $F_{\la_\dt}^{n}(\textbf{C}_n)\subset \textbf{B}_0$, and the measure $\tilde{\mu_\la}$ is $F_{\la_\dt}$-invariant, we conclude that
   \begin{equation*}
      \int_{ \textbf{M}_N}\big|\dot\varphi_\la-\dot\Psi_\la\big|d\tilde{\mu_\la}\leeq \frac{K_1}{N^{1/2}}\int_{ \textbf{B}_0} \big|\dot\varphi_\la\big| d\tilde{\mu_\la}.
   \end{equation*}
Thus, the statement follows from Proposition \ref{prop:far}.
\end{proof}

\subsection{The case $\mc D(-1)<4/3$}

Let us write
   $$G^h_+(s):=Mk\int^\infty_{s}\Lambda_{0}^{h}(u) du, \;\;\; G^h_-(s):=Mk\int_{-\infty}^{s}\Lambda_{0}^{h}(u) du,$$
where $M$ is the constant from Lemma \ref{lem:est}.

For $h\in[1,4/3)$ we define
   $$\Upsilon_+(h):=\int_0^\infty(6\Gamma(s)-1)G_+^h(s)ds, \;\;\; \Upsilon_-(h):=\int_{-\infty}^0(6\Gamma(s)-1)G_-^h(s)ds.$$

\begin{prop}\label{prop:mn}
   For every $\varepsilon>0$ there exist $\eta>0$, $N\in\N$ such that if $0<|\delta_\la|<\eta$ and $\mc D(-1)<4/3$ then
 \begin{multline*}
    \Upsilon_\pm(\mc D(-1))-\varepsilon\leeq |\delta_\la|^{-\frac{3}{2}\mc D(\la)+2}\int_{\textbf{M}_N}\re\Big(\frac{\frac{\partial}{\partial \la} (F_\la'(\varphi_\la))}{F_\la'(\varphi_\la)}\Big)d\tilde{\mu_\la}\\ \leeq \Upsilon_\pm(\mc D(-1))+\varepsilon.
 \end{multline*}
   We take $\Upsilon_+$ for $\delta_\la>0$, and $\Upsilon_-$ for $\delta_\la<0$.
\end{prop}

\begin{proof}
Fix $\ve>0$. Proposition \ref{prop:reDF} combined with Lemma \ref{lem:ee} (\ref{lit:ee2}) (note that the function $\Gamma$ is bounded) leads to
 \begin{multline*}
    \bigg|\int_{\textbf{M}_N} \re\Big(\frac{\mathfrak{D}^2_{F_\la}(\varphi_\la,\dot\Psi_\la)} {F_\la'(\varphi_\la)}\Big)d\tilde{\mu_\la}- \sum_{n>N}(6\Gamma(n\delta_\la)-1)\tilde{\mu_\la}(\textbf{C}_n)\bigg|\\ \leeq\ve K_1\tilde{\mu_\la}(\textbf{M}_N),
 \end{multline*}
where $N\in\N$ is large enough $\delta_\la$ is close to $0$, and $K_1$ does not depend on $\ve$.

Note that Lemma \ref{lem:mj} (\ref{lit:mj1}) gives us an estimate of $\tilde{\mu_\la}(\textbf{M}_N)$. Thus, we conclude from Lemma \ref{lem:mm}, Lemma \ref{lem:C} and definition of $G_\pm^h$, that
 \begin{multline*}
    \bigg|\int_{\textbf{M}_N}\re\Big(\frac{\mathfrak{D}^2_{F_\la} (\varphi_\la,
    \dot\Psi_\la)} {F_\la'(\varphi_\la)}\Big)d\tilde{\mu_\la}\\- |\delta_\la|^{\frac{3}{2}\mc D(\la)-1} \sum_{n>N}(6\Gamma(n\delta_\la)-1)G_\pm^{\mc D(\la)}(n\delta_\la)\bigg| \leeq\ve K_2|\delta_\la|^{\frac32\mc D(\la)-2}.
 \end{multline*}
Next, as in \cite[Proposition 11.1]{Ji}, we can obtain the following formula (cf. definition of $\Upsilon_\pm$),
 \begin{multline*}
    \Upsilon_\pm(\mc D(-1))-\ve K_3\leeq |\delta_\la|^{-\frac{3}{2}\mc D(\la)+2}\int_{\textbf{M}_N}\re\Big(\frac{\mathfrak{D}^2_{F_\la} (\varphi_\la,
   \dot\Psi_\la)} {F_\la'(\varphi_\la)}\Big)d\tilde{\mu_\la}\\ \leeq \Upsilon_\pm(\mc D(-1))+\ve K_3.
 \end{multline*}
where $K_3$ does not depend on $\ve$. Thus, the assertion follows from Lemma \ref{lem:f-P}.
\end{proof}

\begin{lem}\label{lem:>0}
   If $h\in(1,4/3)$ then $\Upsilon_+(h)>\Upsilon_-(h)>0$, whereas $\Upsilon_-(1)=0$.
\end{lem}

\begin{proof}
The fact that  $\Upsilon_-(1)=0$ is proven in (\cite{Ji}, Lemma 11.3) by an exact computation.\\
 Next we first notice that
 $$\Lambda_0^h(-u)=e^{hu}\Lambda_0^h(u).$$
 We also compute
 $$v(x)=\int_0^x(6\Gamma(s)-1)ds=\frac{3xe^{2x}}{e^{2x}-1}-x-\frac 32.$$
 Integrating by part, we then obtain
 \begin{align*}
 \Upsilon_+(h)&=\int_0^{+\infty}v(x)\Lambda_0^h(x)dx,\\
\Upsilon_-(h)&=-\int_{-\infty}^0v(x)\Lambda_0^h(x)dx\\
&=-\int_0^{+\infty}e^{hx}v(-x)\Lambda_0^h(x)dx,
\end{align*}
Now we prove that $\Upsilon_-(h)>0$ for $1<h<4/3$.
Note that the function $6\Gamma(x)-1$ is monotone increasing, whereas $6\Gamma(0)-1=1/2$ and $\lim_{x\rightarrow-\infty}6\Gamma(x)-1=-1$. So, we conclude that there exist unique point $x_0>0$, such that $-v(-x_0)=0$. Moreover, the function $-v(-x)$ is positive on the interval $(0,x_0)$, and negative on $(x_0,\infty)$.
Write $e^{hx}\Lambda_0^h(x_0):=c_h$. Because the function $e^{x}\Lambda_0^1(x)$ is positive and decreasing, we see that the function $e^{(h-1)x}\Lambda_0^{h-1}(x)$ is also positive decreasing. We have
$$e^{(h-1)x}\Lambda_0^{h-1}(x)=\frac{e^{hx}\Lambda_0^{h}(x)}{e^{x}\Lambda_0^{1}(x)},$$
therefore if $x<x_0$ then
$$\frac{e^{hx}\Lambda_0^{h}(x)}{e^{x}\Lambda_0^{1}(x)}>\frac{e^{hx_0}\Lambda_0^{h}(x_0)}{e^{x_0}\Lambda_0^{1}(x_0)},\textrm{ so } 
\frac{e^{hx}\Lambda_0^{h}(x)}{c_h}>\frac{e^{x}\Lambda_0^{1}(x)}{c_1}.$$
For $x>x_0$ we obtain the opposite inequality, thus the function 
$$\frac{e^{hx}\Lambda_0^{h}(x)}{c_h}-\frac{e^{x}\Lambda_0^{1}(x)}{c_1}$$
is positive on $(0,x_0)$, and negative on $(x_0,\infty)$.
Therefore
\begin{multline*}
\int_0^{+\infty}(-v(-x))\frac{e^{hx}\Lambda_0^h(x)}{c_h}dx-
\int_0^{+\infty}(-v(-x))\frac{e^{x}\Lambda_0^1(x)}{c_1}dx\\
=\int_0^{+\infty}(-v(-x))\Big(\frac{e^{hx}\Lambda_0^h(x)}{c_h}-\frac{e^{x}\Lambda_0^1(x)}{c_1}\Big)dx>0,
\end{multline*}
so
$$\Upsilon_-(h)=\int_0^{+\infty}(-v(-x))e^{hx}\Lambda_0^h(x)dx>0.$$
We now prove that 
$$\Upsilon_+(h)>\Upsilon_(h)$$
on the interval $(1,4/3)$.
\begin{align*}
 \Upsilon_+(h)-\Upsilon_-(h)&=\int_0^{+\infty}(v(x)+e^{hx}v(-x))\Lambda_0^h(x)dx\\
 &\geq\int_0^{x_0}(v(x)+e^{hx}v(-x))\Lambda_0^h(x)dx\\
 &\geq \int_0^{x_0}(v(x)+e^{2x}v(-x))\Lambda_0^h(x)dx.
 \end{align*}
 and we conclude by noticing that, for $x\geq 0$, $$v(x)+e^{2x}v(-x)=\frac{4x-3\sinh(2x)+2x\cosh(2x)}{1-e^{-2x}}\geq 0,$$ as can be easily seen.

\end{proof}

\begin{proof}[Proof of Theorem \ref{thm:twopetals} in the case $\mc D(\la)<4/3$]
Since $|\delta_\la|^{\frac{3}{2}\mc D(\la)-2}\rightarrow\infty$ when
$\delta_\la\rightarrow0$, Proposition \ref{prop:mn} and Lemma \ref{lem:>0} combined with Propositions \ref{prop:den}, \ref{prop:limb} and formula (\ref{eq:formula1}) leads to
 \begin{equation*}
    -(1+\ve)K_\pm|\delta_\la|^{\frac{3}{2}\mc D(\la)-2}\leeq\mc D'(\la)\leeq-(1-\ve)K_\pm|\delta_\la|^{\frac{3}{2}\mc D(\la)-2}
 \end{equation*}
for suitable constants $K_\pm$. Next $\mc D(\la)$ can be replaced by $\mc D(-1)$ as in \cite[proof of Theorem 1.1]{Ji}.
\end{proof}

\subsection{The case $\mc D(-1)=4/3$}

\begin{lem}\label{lem:4/3m}
   If $\mc D(-1)=4/3$, then for every $\ve>0$, $\alpha>0$ and $N\in\N$ there exist $\eta>0$ such that
 \begin{multline*}
      \Big(\frac12-\ve\Big)\tilde{\mu_\la}(\textbf{M}_N\sms\textbf{M}_{[\alpha/|\delta_\la|]})\leeq \int_{\textbf{M}_N}\re\Big(\frac{\frac{\partial}{\partial \la} (F_\la'(\varphi_\la))}{F_\la'(\varphi_\la)}\Big)d\tilde{\mu_\la}\\\leeq \Big(\frac12+\ve\Big)\tilde{\mu_\la}(\textbf{M}_N\sms\textbf{M}_{[\alpha/|\delta_\la|]}),
 \end{multline*}
   where $0<|\delta_\la|<\eta$. In particular, the above integral tends to infinity when $\la\rightarrow-1$.
\end{lem}

\begin{proof}
Fix $\alpha, \ve>0$ (small). We can also assume that $\alpha<\ve<\min(\ve_0,K^{-1})$, where $\ve_0$ and $K$ are constants from Lemma \ref{lem:ee}.
Since the function $|\Gamma|$ is bounded by $1/2$, we conclude from Proposition \ref{prop:reDF}, that there exists $\eta$ such that
 \begin{equation*}
    \int_{\textbf{M}_{[\alpha/|\delta_\la|]}}\bigg|\re\Big(\frac{\mathfrak{D}^2_{F_\la} (\varphi_\la,
    \dot\Psi_\la)} {F_\la'(\varphi_\la)}\Big)\bigg|d\tilde{\mu_\la}\leeq \sum_{n>[\alpha/|\delta_\la|]}(2+\ve e^{\ve n|\delta_\la|})\tilde{\mu_\la}(\mc C_n).
 \end{equation*}
So, Lemma \ref{lem:ee} (\ref{lit:ee2}), assumption that $\mc D(-1)=4/3$, and Lemma \ref{lem:a/d} gives us
 \begin{multline}\label{eq:ml}
    \int_{\textbf{M}_{[\alpha/|\delta_\la|]}}\bigg|\re\Big(\frac{\mathfrak{D}^2_{F_\la} (\varphi_\la,
    \dot\Psi_\la)} {F_\la'(\varphi_\la)}\Big)\bigg|d\tilde{\mu_\la}\leeq (2+\ve K)\tilde{\mu_\la}(\textbf{M}_{[\alpha/|\delta_\la|]})\\ \leeq3\tilde{\mu_\la}(\textbf{M}_{[\alpha/|\delta_\la|]})\leeq 3\ve\tilde{\mu_\la}(\textbf{M}_N\sms\textbf{M}_{[\alpha/|\delta_\la|]}),
 \end{multline}
where  $0<|\delta_\la|<\eta$, for suitably chosen $\eta>0$.

Now we will estimate integral over the set $\textbf{M}_N\sms\textbf{M}_{[\alpha/|\delta_\la|]}$. So, we have $N<n\leeq[\alpha/|\delta_\la|]$, in particular $-\alpha\leeq n\delta_\la\leeq\alpha$.
Recall that $\Gamma(0)=1/4$, and $\Gamma'(0)=1/6$. Thus, if $\alpha$ is small enough (possibly changing $\eta$), we get
 \begin{equation*}
    \frac12-2\alpha<6\Gamma(n\delta_\la)-1<\frac12+2\alpha.
 \end{equation*}
We can assume that $e^{\ve n |\delta_\la|}<2$, so Proposition \ref{prop:reDF} leads to
 \begin{multline*}
    \Big(\frac12-2\alpha-2\ve\Big) \tilde{\mu_\la}(\textbf{M}_N\sms\textbf{M}_{[\alpha/|\delta_\la|]})\leeq \int_{\textbf{M}_N\sms\textbf{M}_{[\alpha/|\delta_\la|]}}\re\Big(\frac{\mathfrak{D}^2_{F_\la} (\varphi_\la,
    \dot\Psi_\la)} {F_\la'(\varphi_\la)}\Big)d\tilde{\mu_\la}\\\leeq \Big(\frac12+2\alpha
    +2\ve\Big) \tilde{\mu_\la}(\textbf{M}_N\sms\textbf{M}_{[\alpha/|\delta_\la|]})
 \end{multline*}
Thus, for $\alpha>0$ small enough and $N\in\N$ large enough, statement follows from (\ref{eq:ml}) and Lemma \ref{lem:f-P}

But, we know that $\tilde{\mu_\la}(\textbf{M}_0\sms\textbf{M}_{N})$ is bounded, whereas $\tilde{\mu_\la}(\textbf{M}_N\sms\textbf{M}_{[\alpha/|\delta_\la|]})\rightarrow\infty$. Moreover $\tilde{\mu_\la}(\textbf{M}_N\sms\textbf{M}_{[\alpha/|\delta_\la|]})$ is close to $\tilde{\mu_\la}(\textbf{M}_N)$ (see Lemma \ref{lem:a/d}). Therefore statement holds for every $\alpha>0$ and $N\in\N$.
\end{proof}

\begin{col}\label{col:d4}
   If $\mc D(-1)=4/3$, then there exist $K>0$ and for every $\ve>0$ there exist $\eta>0$ such that
 \begin{equation*}
    (1-\ve)K\frac{|\delta_\la|^{\frac32\mc D(\la)-2}-1}{\frac32\mc D(\la)-2}\leeq \mc D'(\la) \leeq(1+\ve)K\frac{|\delta_\la|^{\frac32\mc D(\la)-2}-1}{\frac32\mc D(\la)-2}
 \end{equation*}
   where $0<|\delta_\la|<\eta$. Moreover $\mc D'(\la)\rightarrow-\infty$, when $\delta_\la\rightarrow0$.
\end{col}

\begin{proof}
The fact that $\mc D'(\la)\rightarrow-\infty$ follows from Lemma \ref{lem:4/3m} combined with Propositions \ref{prop:den}, \ref{prop:limb}, and formula (\ref{eq:formula1}). Moreover, we see that it is enough to estimate $\tilde{\mu_\la}(\textbf{M}_N\sms\textbf{M}_{[\alpha/|\delta_\la|]})$.

Fix $\ve>0$. We can find $\alpha>0$ and $\N\in\N$ such that
 \begin{equation*}
    (1-\ve)K\sum_{n=N+1}^{[\alpha/|\delta_\la|]} n^{-\frac32\mc D(\la)+1}\leeq\sum_{n=N+1}^{[\alpha/|\delta_\la|]}\tilde{\mu_\la}(\textbf{C}_n)\leeq(1+\ve)K\sum_{n=N+1}^{[\alpha/|\delta_\la|]} n^{-\frac32\mc D(\la)+1}
 \end{equation*}
(see Proposition \ref{prop:Csmall}).
Since $\mc D'(\la)\rightarrow-\infty$, we can assume that $\mc D(\la)\neq4/3$. Thus
 \begin{equation*}
    \int_{N}^{\alpha/|\delta_\la|}x^{-\frac32\mc D(\la)+1} dx= \alpha^{-\frac32\mc D(\la)+2} \frac{|\delta_\la|^{\frac32\mc D(\la)-2}-(\frac{\alpha}{N})^{\frac32\mc D(\la)-2}}{-\frac32\mc D(\la)+2}.
 \end{equation*}
%Write $h(\delta_\la):=\frac32\mc D(\la)-2$ and $A:=\frac{\alpha}{N}$.
Next, one can get
 \begin{equation*}
    \frac{|\delta_\la|^{\frac32\mc D(\la)-2}-(\frac{\alpha}{N})^{\frac32\mc D(\la)-2}}{|\delta_\la|^{\frac32\mc D(\la)-2}-1}= 1-\frac{(\frac{\alpha}{N})^{\frac32\mc D(\la)-2}-1}{|\delta_\la|^{\frac32\mc D(\la)-2}-1}\rightarrow1.
 \end{equation*}
Therefore, using the fact that $\alpha^{-\frac32\mc D(\la)+2}\rightarrow1$, we obtain
 \begin{equation*}
    (1-2\ve)K\frac{|\delta_\la|^{\frac32\mc D(\la)-2}-1}{-\frac32\mc D(\la)+2}\leeq \tilde{\mu_\la}(\textbf{M}_N\sms\textbf{M}_{[\alpha/|\delta_\la|]}) \leeq(1+2\ve)K\frac{|\delta_\la|^{\frac32\mc D(\la)-2}-1}{-\frac32\mc D(\la)+2},
 \end{equation*}
and the statement follows.
\end{proof}

\begin{lem}
   If $\mc D(\la)\rightarrow4/3$, then
 $$\lim_{\delta_\la\rightarrow0}|\delta_\la|^{\frac32\mc D(\la)-2}=1.$$
\end{lem}

\begin{proof}
Fix $0<\tilde\ve<1/4$. Let $K$ and $\eta$ be constants from Corollary \ref{col:d4}, such that the estimates of $\mc D'(\la)$ holds for $\tilde\ve$.

Note that if $\delta_\la^g=\alpha$, then $g=\log a/\log \delta_\la$. So, let us define
$$g_{\alpha,\ve}(\delta_\la):=\frac{\log \alpha}{\log \delta_\la}+\ve.$$
Moreover, write
$$h(\delta_\la):=\frac32\mc D(\la)-2.$$

\emph{Step 1. The case $\delta_\la>0$.}
Since we can assume that $\mc D'(\la)<0$ we conclude that $h(\delta_\la)<0$, therefore $\delta_\la^{h(\delta_\la)}>1$.

Suppose, towards a contradiction, that there exists a sequence $\delta_{\la_n}\rightarrow0$ and $a>1$ such that
   $$\lim_{n\rightarrow\infty}\delta_{\la_n}^{h(\delta_{\la_n})}\greq a.$$
i.e. $h(\delta_{\la_n})\leeq \log a/\log \delta_{\la_n}$. There exist an element $\delta_{\la_m}<\eta$ such that
\begin{equation*}
   2K \delta_\la^{1/2}a |\log^3 \delta_\la|\leeq\log^2 a,
\end{equation*}
for $\delta_\la\in(0,\delta_{\la_m})$. Moreover, we can find $1<\tilde a<a$ and $\ve>0$ such that
\begin{equation*}
   2K \delta_\la^{1-\ve}\tilde a |\log^3 \delta_\la|\leeq\log^2 \tilde a,
\end{equation*}
where
$$\frac{\log a}{\log \delta_{\la_m}}=\frac{\log\tilde a}{\log \delta_{\la_m}-\ve}=g_{\tilde a,-\ve}(\delta_{\la_m}),$$
and $\delta_\la\in(0,\delta_{\la_m})$.
Since $h(\delta_{\la_n})\leeq \log a/\log \delta_{\la_n}$, we get $h(\delta_{\la_m})\leeq g_{\tilde a,-\ve}(\delta_{\la_m})$.
Next we have
\begin{equation*}
   2K\delta_\la^{1-\ve}(\tilde a- \delta_\la^{\ve})|\log^3 \delta_\la|\leeq\log \tilde a(\log \tilde a-\ve\log \delta_\la),
\end{equation*}
and then
\begin{equation*}
   2K \frac{\delta_\la^{\frac{\log\tilde a}{\log\delta_\la}-\ve}-1}{\frac{\log\tilde a}{\log\delta_\la}-\ve}= 2K \frac{\tilde a\delta_\la^{\ve}-1}{\log \tilde a-\ve\log \delta_\la}\log \delta_\la\greq
   - \frac{\log\tilde a}{\delta_\la\log^2 \delta_\la}=g_{\tilde a,-\ve}'(\delta_\la).
\end{equation*}
So, if $h(\delta_\la)=g_{\tilde a,-\ve}(\delta_\la)$, then Corollary \ref{col:d4} gives us
\begin{equation*}
    h'(\delta_\la)=\frac32\mc D'(\la)\greq g_{\tilde a,-\ve}'(\delta_\la),
\end{equation*}
where $\delta_\la\in(0,\delta_{\la_m})$
Thus, the assumption $h(\delta_{\la_m})\leeq g_{\tilde a,-\ve}(\delta_{\la_m})$ leads to
\begin{equation*}
    \lim_{\delta_\la\rightarrow0}h(\delta_\la)\leeq g_{\tilde a,-\ve}(0)=-\ve<0.
\end{equation*}
This is contradiction to the fact that $\lim_{\delta_\la\rightarrow0}\mc D(\la)=4/3$.

\emph{Step 2. The case $\delta_\la<0$.}
Since $\mc D'(\la)<0$ we conclude that $h(\delta_\la)>0$ and $|\delta_\la|^{h(\delta_\la)}<1$

As before suppose that there exist a sequence $\delta_{\la_n}\rightarrow0$ and $b<1$ such that
   $$\lim_{n\rightarrow\infty}|\delta_{\la_n}|^{h(\delta_{\la_n})}\leeq  b.$$
i.e. $h(\delta_{\la_n})\greq \log b/\log |\delta_{\la_n}|$. There exist an element $|\delta_{\la_m}|<\eta$ such that
\begin{equation*}
   2K \delta_\la\log^3 |\delta_\la|\leeq\log^2 b,
\end{equation*}
for $\delta_\la\in(\delta_{\la_m},0)$. We can find $b<\tilde b<1$ and $\ve>0$ such that
\begin{equation*}
   2K (1-b|\delta_\la|^\ve)\delta_\la\log^3 |\delta_\la|\leeq\log \tilde b(log \tilde b+\ve\log |\delta_\la|),
\end{equation*}
where
$$\frac{\log b}{\log |\delta_{\la_m}|}=\frac{\log\tilde b}{\log |\delta_{\la_m}|+\ve}=g_{\tilde b,\ve}(\delta_{\la_m})$$
and $\delta_\la\in(0,\delta_{\la_m})$.
Thus $h(\delta_{\la_m})\greq g_{\tilde b,\ve}(\delta_{\la_m})$. Next
\begin{equation*}
   2K \frac{|\delta_\la|^{\frac{\log\tilde b}{\log|\delta_\la|}+\ve}-1}{\frac{\log\tilde b}{\log|\delta_\la|}+\ve}=
   2K \frac{\tilde b|\delta_\la|^{\ve}-1}{\log \tilde b+\ve\log |\delta_\la|}\log |\delta_\la|
   \greq -\frac{\log\tilde b}{\delta_\la\log^2 |\delta_\la|}=g_{\tilde b,\ve}'(\delta_\la).
\end{equation*}
So, we see that if $h(\delta_\la)=g_{\tilde b,\ve}(\delta_\la)$ then Corollary \ref{col:d4} gives us
\begin{equation*}
    h'(\delta_\la)=\frac32\mc D'(\la)\greq g_{\tilde b,\ve}'(\delta_\la),
\end{equation*}
where $\delta_\la\in(\delta_{\la_m},0)$
Thus, the assumption $h(\delta_{\la_m})\greq g_{\tilde b,\ve}(\delta_{\la_m})$ leads to
\begin{equation*}
    \lim_{\delta_\la\rightarrow0}h(\delta_\la)\greq g_{\tilde b,\ve}(0)=\ve>0.
\end{equation*}
This is contradiction to the fact that $\lim_{\delta_\la\rightarrow0}\mc D(\la)=4/3$.
\end{proof}

\begin{proof}[Proof of Theorem \ref{thm:twopetals} in the case $\mc D(\la)=4/3$]  Fix $\ve>0$.
Since
$$e^{(\frac32\mc D(\la)-2)\log|\delta_\la|}=\delta_\la^{\frac32\mc D(\la)-2}\rightarrow1,$$
we get $(\frac32\mc D(\la)-2)\log|\delta_\la|\rightarrow0$, hence
 \begin{equation*}
   |\delta_\la|^{\frac32\mc D(\la)-2}-1=\Big(\frac32\mc D(\la)-2\Big)\log|\delta_\la|+ o\Big(
   \Big(\frac32\mc D(\la)-2\Big)\log|\delta_\la|\Big).
 \end{equation*}
Thus, we obtain
 \begin{equation*}
    (1-\ve)\log|\delta_\la|\leeq\frac{|\delta_\la|^{\frac32\mc D(\la)-2}-1}{\frac32\mc D(\la)-2}\leeq(1+\ve)\log|\delta_\la|,
 \end{equation*}
where $0<|\delta_\la|<\eta$ for sufficiently chosen $\eta>0$. So, statement follows from Corollary \ref{col:d4}.
\end{proof}

\subsection{The case $\mc D(-1)>4/3$}

\begin{prop}\label{prop:mnn}
   For every $\varepsilon>0$ there exist $\eta>0$ and $N\in\N$ such that if $0<|\delta_\la|<\eta$ then
 \begin{equation*}
    \int_{\textbf{M}_N}\bigg|\re\Big(\frac{\frac{\partial}{\partial \la} (F_\la'(\varphi_\la))}{F_\la'(\varphi_\la)}\Big)\bigg|d\tilde{\mu_\la}\leeq \ve.
 \end{equation*}
\end{prop}

\begin{proof}
Fix $\ve>0$. Let $K_1$ be a constant from Lemma \ref{lem:f-P} and let $N_0$ be an integer such that
 \begin{equation}\label{eq:n0}
    \frac{K_1}{N_0^{1/2}}\leeq\frac{\ve}{2}.
 \end{equation}

Fix $\tilde\ve>0$ less than $\ve_0$ from Lemma \ref{lem:ee}.
Since the function $|\Gamma|$ is bounded by $1/2$, we conclude from Proposition \ref{prop:reDF}, that there exist $N_1$ and $\eta$ such that
 \begin{equation*}
    \int_{\textbf{M}_N}\bigg|\re\Big(\frac{\mathfrak{D}^2_{F_\la} (\varphi_\la,
    \dot\Psi_\la)} {F_\la'(\varphi_\la)}\Big)\bigg|d\tilde{\mu_\la}\leeq \sum_{n>N}(2+\tilde\ve e^{\tilde\ve n|\delta_\la|})\tilde{\mu_\la}(\mc C_n),
 \end{equation*}
where $N\greq N_1$ and $|\delta_\la|<\eta$. Next, Lemma \ref{lem:ee} (\ref{lit:ee2}) and \ref{lem:mj} (\ref{lit:mj3}) leads to
 \begin{multline}\label{eq:n1}
    \int_{\textbf{M}_N}\bigg|\re\Big(\frac{\mathfrak{D}^2_{F_\la} (\varphi_\la,
    \dot\Psi_\la)} {F_\la'(\varphi_\la)}\Big)\bigg|d\tilde{\mu_\la}\leeq 2\tilde{\mu_\la}(\textbf{M}_N)+\tilde\ve K_2\tilde{\mu_\la}(\textbf{M}_N)\\\leeq (2+\tilde\ve K_2)\frac{K_3}{\frac32\mc D(\la)-2}N^{-\frac32\mc D(\la)+2}.
 \end{multline}
Since we can find an integer $N>\max(N_0,N_1)$ and $\eta>0$ such that
   $$(2+\tilde\ve K_2)\frac{K_3}{\frac32\mc D(\la)-2}N^{-\frac32\mc D(\la)+2}\leeq\frac{\ve}{2},$$
where $|\delta_\la|<\eta$, the assertion follows from (\ref{eq:n0}) and (\ref{eq:n1}).
\end{proof}

\begin{proof}[Proof of Theorem \ref{thm:twopetals} in the case $\mc D(\la)>4/3$]
The statement follows from formula (\ref{eq:formula1})
combined with Propositions \ref{prop:den}, \ref{prop:limb} and Proposition \ref{prop:mnn}.
\end{proof}

\end{document}